\documentclass[12pt]{article}

\usepackage{hyperref,graphicx,amsthm,amsmath,amsfonts,color,algorithm,enumerate,fullpage,tikz}

\newtheorem{theorem}{Theorem}[section]

\newtheorem{proposition}[theorem]{Proposition}

\newtheorem{example}[theorem]{Example}
\newtheorem{lemma}[theorem]{Lemma}

\newtheorem{remark}[theorem]{Remark}

\newtheorem{defn}[theorem]{Definition}

\DeclareMathOperator{\diag}{\mathrm{diag}}
\DeclareMathOperator{\dnf}{\mathrm{dnf}}
\DeclareMathOperator{\err}{\mathrm{err}}
\DeclareMathOperator{\ber}{\mathrm{Ber}}
\DeclareMathOperator{\unif}{\mathrm{Unif}}

\newcommand{\1}{\mathbf 1}
\newcommand{\0}{\mathbf 0}
\newcommand{\e}{\mathbf e}

\newcommand{\F}{\mathrm F}
\newcommand{\R}{\mathbb R}

\newcommand{\naturals}{\mathbb N}
\newcommand{\rmk}[1]{\begin{remark}\rm{#1}\end{remark}}
\newcommand{\ex}[1]{\begin{example}\rm{#1}\end{example}}

\begin{document}
\title{Clusters in Markov Chains via Singular Vectors of Laplacian Matrices}

\author{Sam Cole\thanks{Department of Mathematics, University of Missouri, Columbia, MO, U.S.A. (\tt s.cole@missouri.edu).}
	\and  
	Steve Kirkland\thanks{Department of Mathematics, University of Manitoba, Winnipeg, MB, Canada (\tt stephen.kirkland@umanitoba.ca).}
}

\maketitle
\begin{abstract} Suppose that $T$ is a stochastic matrix. We propose an algorithm for identifying clusters in the Markov chain associated with $T$. The algorithm is recursive in nature, and in order to identify clusters, it uses the sign pattern of a left singular vector associated with the second smallest singular value of the Laplacian matrix $I-T.$ 
We prove a number of results that justify  the algorithm's approach, and illustrate the algorithm's performance with several numerical examples.  
\end{abstract}


\noindent \textbf{AMS classification numbers}: 15A18, 15B51, 60J22, 65C40.

\tableofcontents

\section{Introduction and preliminaries}



There is a good deal of work on identifying  clusters in Markov chains. This work is motivated in part by the numerous domains in which Markov chains are applied, including molecular conformation dynamics, vehicle traffic networks, economics, and wireless network design. 
Clusters (also referred to as  almost invariant aggregates, or metastable sets) correspond to a partition of the state space into subsets with the property that the density of transitions within the same subset is high, while the density of transitions between different subsets is low.
Within the context of domains of application,  clusters in a Markov chain may, for example, correspond to metastable chemical conformations of biomolecules \cite{DHFS},  or neighbourhoods in an urban traffic network \cite{CKS}.  Further, in the area of complex networks, there is interest in identifying community structure within such networks; again, a Markov chain, specifically the random walk on the corresponding graph, is a key tool for finding clusters in complex networks such as the world trade network and networks of scientific collaborators \cite{Pic}, and identifying well--connected regions in protein--protein interaction networks \cite{Azadetal}.

If a Markov chain exhibits clustering, then it is natural to expect that the corresponding transition matrix has several eigenvalues close to $1$,
and for this reason there is a body of work focusing on the use of eigenvalues  and eigenvectors of the transition matrix to detect the presence of clustering and identify the clusters themselves;   see 
\cite{DHFS}, \cite{DW},  \cite{CKS}, and \cite{BCFKS} for work along those lines. 
An alternate approach to cluster identification using the singular values and singular vectors of the transition matrix is described in \cite{fritzsche2008svd}. In particular, the right singular vector corresponding to the second largest singular value of the transition matrix is used to identify clusters. We note however that several issues with  approach of \cite{fritzsche2008svd} are identified in \cite{Tif}; in particular, \cite{Tif}   furnishes some counterexamples to some of the theoretical work in \cite{fritzsche2008svd}, and discusses additional technical hypotheses under which the conclusions of \cite{fritzsche2008svd} are valid. 

 
In this paper, we present an algorithm based on the SVD of the Laplacian for the Markov chain that can be used to
\begin{enumerate}[a)]
\item	Detect the presence of clusters in an arbitrary Markov chain.
\item	Recover clusters in the event that they are present.
\end{enumerate}
The algorithm works as follows. Given an $n \times n$  stochastic matrix $T,$ we consider the singular value decomposition of the associated \emph{Laplacian matrix},  $I-T.$ Note that $I-T$ is singular, and its nullity, say $k$, coincides with the multiplicity of $1$ as an eigenvalue of $T$, or equivalently, the number of direct summands in $T$ that have $1$ as a simple eigenvalue. In particular, if $T$ is a direct sum of $k$ irreducible stochastic matrices (the \emph{completely decoupled case}), it follows that $0$ is a singular value of $I-T$ of multiplicity $k$. 
That observation prompts the intuition that if an irreducible stochastic matrix $T$ is a small  perturbation of the transition matrix for a completely decoupled Markov chain, then $I-T$ will have at least one singular value that is small and positive. 


Thus we use the presence of a small positive singular value (for $I-T$) to detect clustering in $T$. In order to identify the clusters, we rely on a corresponding left singular vector, specifically its sign pattern.  Indeed the positive and negative entries of that singular vector yield a partition of the states into subsets which exhibit clustering in the sense defined in~\cite{fritzsche2008svd}.
As our method requires computing a left singular vector associated with a small singular value, the entire SVD is not needed, and we may rely on specialised methods for computing singular vectors for small singular values such as those appearing in 
 \cite{vanHuff1} and \cite{vanHuff2}. 
 
 Our paper makes several novel contributions. First, by working with singular vectors (as opposed to eigenvectors), our algorithm maintains some of the advantages outlined in \cite{fritzsche2008svd}, such as orthogonality of the singular vectors, and the fact that only one singular vector  needs to be computed at each iterative step. However, by working with the Laplacian matrix instead of the transition matrix, we avoid  the pitfalls that are identified in \cite{Tif}. 
Further, we prove results that
underscore and quantify the connections  between small singular values of the Laplacian matrix and several notions of clustering.
Finally, our clustering algorithm employs the \emph{dangling node fix} introduced by Page and Brin as part of Google's PageRank algorithm~\cite{brin1998anatomy,page1999pagerank}.
 Our analysis suggests that, in the context of a clustering algorithm, the DNF is an effective method for generating a stochastic matrix from a sub-stochastic matrix.  Use of the DNF enables us to work with a stochastic matrix  in each recursive call of the algorithm; this is in contrast to the approach of \cite{fritzsche2008svd}, which works with sub-stochastic matrices.    
 
\subsection{Outline}

Our paper is organized as follows:
\begin{itemize}
\item	In Section~\ref{subsec:clusterdef} we define several notions of clustering used in this paper.
\item	In Section~\ref{subsec:alg} we present our clustering algorithm.
\item	Sections~\ref{subsec:nearlydecoupled=>smallsing}-\ref{subsec:couplingmatrix} show that the existence of a small positive singular value is a good heuristic for detecting the presence of clustering: 
\begin{itemize}
\item	In Section~\ref{subsec:nearlydecoupled=>smallsing} we show that if an irreducible stochastic matrix $T$ is a small  perturbation of the transition matrix for a completely decoupled Markov chain, then $I-T$ will have at least one singular value that is small and positive.  
\item	Conversely, Sections~\ref{subsec:smallsing=>clustering} and~\ref{subsec:couplingmatrix} show that if $I - T$ has a small positive singular value, then it exhibits some clustering behaviour.  
\end{itemize}
\item	In Section~\ref{subsec:dnf} we discuss the \emph{dangling node fix}---a way of stochasticising a sub-stochastic matrix which is optimal in the sense that it minimizes the difference with the original sub-stochastic matrix in several norms.  This normalization is used in the recursive step of our algorithm.

\item	In Section~\ref{sec:numerical} we present numerical evidence of our algorithm's effectiveness.  We demonstrate its performance on several real-world and simulated examples.  

\item	In Section~\ref{sec:openproblems} we discuss several open problems arising from our work.

\end{itemize}


\subsection{Notation and definitions}

We will use standard notation for notions from matrix theory.  All matrices and vectors considered in this paper have real entries, unless otherwise stated.  We will also use the following less-standard notation:

\begin{itemize}

\item	For an $n \times n$ real symmetric matrix $A$, we will denote its eigenvalues by $\lambda_1(A) \geq \ldots \geq \lambda_n(A)$, or simply $\lambda_1 \geq \ldots \geq \lambda_n$ if the matrix $A$ is clear from context.  $\lambda_1$ is called the spectral radius of $A$ and is denoted by $\rho(A)$.

\item	For an $n \times n$ matrix $A$, we will denote its singular values by $\sigma_1(A) \geq \ldots \geq \sigma_n(A)$, or simply $\sigma_1 \geq \ldots \geq \sigma_n$ if the matrix $A$ is clear from context.


\item	For matrices $A$ and $B$ of the same dimension, we write $A \geq B$ 
 if all entries of $A$ are greater than or equal to 
  the corresponding entries of $B$.  If all entries of $A$ are nonnegative 
   we write $A \geq 0$. 
     The same notation applies if $A$ and $B$ are vectors of the same dimension.

\item	$I_n$, $J_n$, and $\1_n$ will denote the $n \times n$ identity matrix, $n \times n$ ones matrix, and $n$-dimensional ones vector, respectively.  We will omit the subscript if the dimension is clear from context.

\item	For an $n \times n$ matrix $A$ and $R, S \subseteq \{1, \ldots, n\}$, we will denote the submatrix of $A$ with row indices in $R$ and column indices in $S$ by $A[R, S]$, and we define $A[R] := A[R, R]$ to be the principal submatrix of $A$ with row and column indices in $R$.  For an $n$-dimensional vector $v$, $v[R]$ will denote the entries of $v$ corresponding to indices in $R$.

\item	For $p \in [1, \infty]$, $||\cdot||_p$ will denote the $\ell_p$-norm of a vector or the corresponding induced matrix norm.  $||\cdot||_\F$ will denote the Frobenius norm of a matrix.  We refer the reader to~\cite{horn2012matrix} for definitions.

\item	$|\cdot|$ will denote the entrywise absolute value of a matrix or vector.

\item	For square matrices $A, B$, we will let $A \oplus B := \diag(A, B)$, i.e.\ the direct sum of $A$ and $B$.

\end{itemize}

Finally, the following definitions will be crucial to our algorithm and analysis:

\begin{defn}[Stochastic, sub-stochastic]
A matrix $T \in \R^{n \times n}$ with nonnegative entries is called \emph{stochastic} if its  each of its rows sums to 1.  It is called \emph{sub-stochastic} if each of its rows sums to at most 1.
\end{defn}

Our clustering algorithm is based on the \emph{Laplacian} the transition matrix of a Markov chain:

\begin{defn}[Laplacian]
For an $n \times n$ matrix $A$ with row sums $r_1, \ldots, r_n$, the \emph{Laplacian} of $A$ is defined as $L(A) := \diag(r_1, \ldots, r_n) - A$.
\end{defn}

\noindent	Note that if $A$ is stochastic then $L(A) = I - A$.

Finally, we introduce the \emph{dangling node fix (DNF)} of a matrix, which will allow us to approximate sub-stochastic matrices with stochastic matrices.

\begin{defn}[Dangling node fix]\label{defn:dnf}
Let $T$ be an $n \times n$ sub-stochastic matrix.  The \emph{dangling node fix} of $T$ is the matrix $\dnf(T) := T + \frac1n(J - TJ)$.
\end{defn}

\noindent	Thus, the DNF is the matrix which results from distributing $1$ minus the sum of each row evenly among the entries of that row.  We note in passing that the term ``dangling node fix'' is a reference to Google's PageRank algorithm~\cite{brin1998anatomy,langville2006google,page1999pagerank}, where an analogous approach is used to transform a sub-stochastic matrix into a stochastic matrix.

\section{Main results}

\subsection{Definitions of clustering}\label{subsec:clusterdef}

Our primary objective is to be able to identify and recover clusters in a Markov chain; however, ``cluster'' is a vague term which simply means part of a network that is ``more connected'' than the network overall.  Thus, care must be taken to precisely define what we mean by clusters.  We will use two notions of clustering in this work.

   First, we will consider a Markov chain to be ``clustered'' if its transition matrix is a small perturbation of that of a \emph{completely decoupled} Markov chain:

\begin{defn}[Completely decoupled]\label{defn:purelyclustered}
A Markov chain or its transition matrix $T$ is called \emph{completely decoupled} if $T$ is the direct sum of at least two stochastic matrices.
\end{defn}

This leads to our first notion of clustering:

\begin{defn}[Nearly decoupled]\label{def:nrlyde}
Given $\epsilon > 0$ and a matrix norm $||\cdot||$, a Markov chain or its transition matrix $T \in \R^{n \times n}$ is called \emph{($\epsilon, ||\cdot||)$-nearly decoupled} if there exists a completely decoupled Markov chain with transition matrix $S \in \R^{n \times n}$ such that $||S - T|| \leq \epsilon$.
\end{defn}

\begin{remark}\label{rmk:inessential}\rm{
In any completely decoupled Markov chain, its transition matrix has eigenvalue 1 with multiplicity at least 2.  There are Markov chains with this property that are \emph{not} completely decoupled in the sense of Definition~\ref{defn:purelyclustered}.  For example:
\[T =
\begin{bmatrix}
1	& 0		& 0	\\
0	& 1		& 0	\\
.5	& .5	& 0
\end{bmatrix}
.\]
A Markov chain with this transition matrix could be considered ``clustered'' because state 3 is inessential and transitions to each of the two essential states with equal probability: if the process starts in state 3, it immediately transitions to one of the other two states and stays there forever.  Thus, the two essential states could be considered clusters, while the inessential state could be considered as not belonging to any cluster.

We will not consider such Markov chains in this work, and instead will use the more restrictive definition of completely decoupled laid out above.
}
\end{remark}

Next, we recall the notion of coupling matrices defined in~\cite{DHFS} and \cite{fritzsche2008svd}.  Such matrices can be used to quantify clustering in a Markov chain.

\begin{defn}[Coupling matrix]
For a block matrix $A = [A_{ij}]_{i, j = 1}^k \in R^{n \times n}$ and a vector $u = [u_1^\top| \ldots | u_k^\top]^\top \in \R^n$ with positive entries partitioned conformally with $A$, the \emph{coupling matrix} of $A$ with respect to $u$ is the $k \times k$ matrix $W_u(A)$ whose $(i, j)$ entry is given by \[\omega_u(A; i, j) := \frac{u_i^\top A_{ij}\1}{u_i^\top\1}.\]
\end{defn}

\noindent	Put simply, $\omega_u(A; i, j)$ computes the weighted average row sum of $A_{ij}$, with weights given by the entries of $u_i$.  Hence, a diagonally dominant coupling matrix indicates clustering, and the higher the diagonal entries are the stronger the clustering.

Note that the partition of the indices of $u$ and $A$ is an implicit parameter of the function $\omega_u$.  Also observe that if $A$ is the transition matrix of a completely decoupled Markov chain, then its coupling matrix is the identity.  We will use two weight vectors in this work: the ones vector, $\1$, and the \emph{left-iterative weight vector} (Definition~\ref{defn:liwv}). 

\rmk{
Both  \cite{DHFS} and \cite{fritzsche2008svd} actually define the notion of a nearly decoupled Markov chain in terms of a coupling matrix. The former paper uses the coupling matrix generated by the stationary vector of the transition matrix, while the latter paper uses the coupling matrix generated by the all ones vector. With those respective coupling matrices in hand, both papers define a Markov chain to be nearly decoupled if all of the diagonal entries of the coupling matrix are sufficiently large.

While Definition \ref{def:nrlyde} defines a nearly decoupled Markov chain    differently from   \cite{DHFS} and \cite{fritzsche2008svd}, the two notions are related. Specifically, suppose that $T$ is an ($\epsilon, ||\cdot||_\infty)$-nearly decoupled transition matrix, and write $T=S+E,$ where $S$ is a direct sum of $k$ stochastic matrices and $||E||_\infty \le \epsilon.$ Partition $T, E$ conformally with $S$ as $\left[\begin{array}{c} T_{ij} \end{array}\right]_{i,j=1, \ldots, k}$ and  $\left[\begin{array}{c} E_{ij} \end{array}\right]_{i,j=1, \ldots, k}$, respectively. For each $j=1, \ldots, k,$ we note that $|E_{jj}\1| \le ||E||_\infty \1,$ and hence $T_{jj}\1 = S_{jj}\1 + E_{jj}\1 = \1 +E_{jj}\1 \ge \1 -  ||E||_\infty \1 \ge (1-\epsilon)\1.$ Consequently, for any  positive vector $u,$ we have $\omega_u(T; j, j) \ge 1-\epsilon, j=1, \ldots, k,$ so that the diagonal entries of any coupling matrix are bounded below by  $1-\epsilon$.  Theorem \ref{thm:directsum_stochastic} below provides a complementary result in the case of the coupling matrix arising from the all ones vector. 
}





\subsection{The left singular vector algorithm}\label{subsec:alg}

We now state our clustering algorithm, which can be used both to detect the presence of clusters and to identify the clusters themselves.  Given a stochastic matrix $T$, we use the second smallest singular value $\sigma$ of $I - T$ to determine whether clusters are present or not.  If clusters are detected in this way, we use the positive and negative entries of a corresponding left singular vector $u$ to partition the indices of $T$ into potential clusters.  We then refine this partition by recursing on the principal submatrix induced by each part of the partition, implementing the dangling node fix  on each submatrix to make it stochastic.  We stop when the second smallest singular value is no longer sufficiently small in any of the recursive calls, and we return the refined partition.

\begin{algorithm}[H]

\caption{The Left Singular Vector algorithm}\label{alg}
Input: Stochastic matrix $T$ with index set $S$, tolerance $\tau$\\
Output: A set of disjoint subsets (clusters) of the index set of $T$


\begin{enumerate}
\item	\label{step:leftsingvec}Let $\sigma$ be the second smallest singular value of $I - T$ and $u$ a corresponding left singular vector with mixed signs.



\item	\label{step:partition}If $\sigma > \tau$, do nothing.  Otherwise, let $S_1$ and $S_2$ be the sets of indices corresponding to positive and negative entries of $u$, respectively.  

\item	\label{step:dnf}For $i = 1, 2$, let $\tilde T_i := \dnf(T[S_i])$.  
\item	Recurse on $\tilde T_1$  and $\tilde T_2$.  The recursive call on $\tilde T_i$ returns a set of clusters $\mathcal P_i$, i.e., a set of disjoint subsets of $S_i$.  Note that there may be some unclustered vertices in $S_i$ which do not belong to any of the sets in $\mathcal P_i$.  

\item	Return $\mathcal P := \mathcal P_1 \cup P_2$, i.e.\ the set of all clusters found.


\end{enumerate}
\end{algorithm}

\rmk{
We would like to point out the following:
\begin{itemize}

\item	If the input matrix $T$ is completely decoupled, then $\sigma = 0$ and hence it has left singular vectors with uniform sign.  However, in this case the left singular subspace corresponding to $\sigma$ will have dimension greater than 1, so it's always possible to choose a singular vector $u$ with mixed signs.

\item	In practice, when we partition the index set in Step~\ref{step:partition}, we also permute the indices of $T$ so that $S_1$ and $S_2$ are contiguous blocks of indices.  This will help visualize the clusters.  If this is done, care should be taken to keep track of the indices in the \emph{original} matrix corresponding to the submatrix currently being worked on in each recursive call.


\end{itemize}
}

During the course of the algorithm, we implicitly compute a natural weight vector, which we call the \emph{left-iterative} weight vector. This vector defines a coupling matrix which can be used to evaluate the quality of the clustering produced.

\begin{defn}[Left-iterative weight vector]\label{defn:liwv}
Let $T$ be the transition matrix of a Markov chain and $\tau > 0$ a tolerance.  The \emph{left-iterative} weight vector of $T$ with respect to $\tau$, denoted $v(T, \tau)$, is constructed recursively as part of Algorithm~\ref{alg} as follows:
\begin{itemize}
\item	Initialize $v(T, \tau)$ to $\1$.
\item	If $\sigma > \tau$ in Step~\ref{step:partition}, do not update $v(T, \tau)$.
\item	Otherwise, replace the entries of $v(T, \tau)$ corresponding to the indices in $S$ (where $S$ is the index set of the current recursive call) with $|u|$, where $u$ is the left singular vector computed in Step~\ref{step:leftsingvec}.
\end{itemize}
\end{defn}

\begin{remark}
\rm{We emphasize that $v(T, \tau)$ is defined \emph{iteratively}, not recursively in terms of $v(\tilde T_1, \tau)$ and $v(\tilde T_2, \tau)$.  It is a single vector of size $n$ (the size of the original input matrix) which is initialized to $\1$ in the top-level call to Algorithm~\ref{alg} and updated in every recursive call.  We show how $v(T, \tau)$ is computed explicitly along with the clustering algorithm in Appendix~\ref{sec:liwv}.}
\end{remark}

\rmk{Our motivation for using the left-iterative weight vector to generate a coupling matrix is as follows. In the case of a completely decoupled transition matrix $T$, running Algorithm \ref{alg} with tolerance zero yields a left-iterative weight vector $v(T,0)$ such that 
for each of the direct summands of $T$, the associated subvector  of $v(T,0)$ is a scalar multiple of the corresponding stationary distribution.  Hence, we may expect that in the nearly decoupled case, the subvector of the left-iterative weight vector corresponding to a cluster will approximate a left Perron vector of the associated principal submatrix; in that case we also expect a large diagonal entry in the coupling matrix.     Theorem~\ref{thm:small_sing} below provides partial support for that intuition.
}

The following example illustrates the algorithm's performance on a standard example of a clustered transition matrix.

\begin{example}\rm{
Consider  the Courtois matrix \cite[Appendix III]{courtois} (note that in Courtois' book, there is a typographical error in the $(6,2)$ entry; it is corrected here): 
\[C=\begin{bmatrix}
0.8500   &      0   &   0.1490   &   0.0009   &        0     & 0.00005  &          0     & 0.00005\\
0.1000   &   0.6500  &    0.2490  &         0   &   0.0009  &    0.00005   &        0  &    0.00005\\
0.1000   &   0.8000  &     0.0996  &     0.0003  &         0     &      0 &     0.0001  &         0\\
0  &    0.0004  &         0  &    0.7000  &    0.2995  &         0 &     0.0001  &         0\\
0.0005        &   0   &   0.0004  &    0.3990  &    0.6000  &    0.0001   &        0  &         0\\
0   &   0.00005  &         0      &     0   &   0.00005  &    0.6000  &    0.2499 &     0.1500\\
0.00003         &  0 &     0.00003  &    0.00004   &        0  &    0.1000 &     0.8000  &    0.0999 \\
0  &    0.00005   &        0    &       0    &  0.00005  &    0.1999   &   0.2500 &     0.5500
\end{bmatrix} .
\]
Here we  apply  Algorithm \ref{alg} to $C$, with tolerance $\tau=0.1.$  The singular values of $I-C$ are approximately $0, 0.0002, 0.0015, 0.2354, $ $  0.4935, 0.6053, 0.7063,$ and $  1.2824,$ and the left singular vector  corresponding to $0.0002$ is 
given by $$\left[\begin{array}{cccccccc}
  0.2997 &
0.3114&
0.1359&
0.5272&
0.3955&
-0.2262&
-0.5221&
-0.1914
\end{array} \right]^\top;$$ hence 
 $S_1=\{1, \ldots, 5\}$ and $S_2=\{6,7,8\}.$ We observe that the weight vector arising from that singular vector yields the following $2 \times 2$ coupling matrix:   
$\left[ \begin{array}{cc}
0.9999   & 0.0001\\
0.0001    &0.9999 \end{array} \right].$

When we apply the dangling node fix to $C[S_2,S_2]$ we find that the second smallest  singular value of $I-\dnf(C[S_2,S_2])$ is approximately $0.4935,$ and so no further iterations are performed on $C[S_2,S_2]$. 
Next, we perform the dangling node fix on $C[S_1,S_1],$ and compute the singular values of $I-\dnf(C[S_1,S_1])$. These are approximately $ 0, 
0.0015 ,  0.2354, 0.7063,$ and  $1.2824,$ and the left singular vector 
 corresponding to $0.0015$ is
 given by $$\left[\begin{array}{ccccc}    -0.5447&
 	-0.5659&
 	-0.2471&
 	0.4539&
 	0.3406\end{array}\right]^\top.$$ 
 Setting 
 $\tilde{S}_1=\{1, 2, 3\}$ and  $\tilde{S}_2=\{4,5\},$ we find that the second smallest singular values for $I-\dnf(C[\tilde{S}_1,\tilde{S}_1])$ and 
 $I-\dnf(C[\tilde{S}_2,\tilde{S}_2])$ 
are $0.2354$ and $0.7063,$ respectively, and so no further iterations are performed.

Thus our algorithm produces the clusters $\{1,2,3\}, \{4,5\},$ and $ \{6,7,8\}.$ 
The left-iterative weight vector resulting from the algorithm is 
$$u=\left[\begin{array}{ccc|cc|ccc}
 0.5447&
0.5659&
0.2471&
0.4539&
0.3406& 
   0.2262&
0.5221&
0.1914
\end{array}
\right]^\top, $$ and the corresponding coupling matrix is 
$$ W_u(C) = 
\begin{bmatrix}
0.9991  &  0.0008  &  0.0001\\
0.0006   & 0.9993 &   0.0001\\
0.0001    &0.0000&    0.9999
\end{bmatrix}. $$
}
\end{example}


\subsection{Nearly decoupled implies small singular value}\label{subsec:nearlydecoupled=>smallsing}

In this subsection, we will show that if a stochastic matrix $T$ is a perturbation of the transition matrix of a \emph{completely decoupled} Markov chain (i.e., a direct sum of $k$ irreducible stochastic matrices), then $I - T$ has $k$ small singular values.  This suggests that small singular values of $I - T$ can be used as a heuristic to detect if $T$ has clusters.

The following lemma allows us to upper bound the operator norm of a stochastic matrix, which will in turn allow us to upper bound the operator norm of the perturbing matrix noted above.

\begin{lemma} \label{2norm} Let $T$ be a stochastic matrix of order $n$. Then $\sigma_1(T)\le \sqrt{n},$ with equality holding if and only if $T = \1 e_j^{\top}$ for some $j=1, \ldots, n.$  
	\end{lemma}
\begin{proof}
	Denote the singular values of $T$ by $\sigma_1 \ge \ldots \ge \sigma_n.$  We have $\sum_{j=1}^n \sigma_j^2 = \sum_{j=1}^n \sum_{k=1}^n t_{j,k}^2 = \sum_{j=1}^n ||e_j^{\top} T||_2^2.$ Since $||e_j^{\top}T||_1 = 1,$ it follows that $||e_j^{\top} T||_2\le 1, j=1, \ldots, n,$ and note that $||e_j^{\top} T||_2= 1$ if and only if $e_j^{\top}T=e_k^{\top} $ for some $k$. 		Consequently, $\sigma_1^2 \le   \sum_{j=1}^n \sigma_j^2 \le n,$ which yields  $\sigma_1\le \sqrt{n}.$ 
		
		Suppose that   $\sigma_1 = \sqrt{n}.$ Then each row of $T$ contains a single 1, and since $\sigma_j=0, j=2, \ldots, n,$ $T$ is also rank 1. It now follows that $T = \1 e_j^{\top}$ for some $j=1, \ldots, n.$   Conversely, if $T = \1 e_j^{\top}$ for some $j$ then $\sigma_1=\sqrt{n}.$   
		\end{proof}
		
If a stochastic matrix is a perturbation of a completely decoupled stochastic matrix, then the perturbing matrix $E$ can be written as the difference of two stochastic matrices, and hence $||E||_2 \leq 2\sqrt n$ by Lemma~\ref{2norm}.  We can leverage this into a better bound as follows: blow up $E$ as much as possible such that it can still be written as the difference of two stochastic matrices.  We still get the same  bound of $2\sqrt n$ on the operator norm, but then we can divide by the blowup factor to get a better bound!  Lemmas~\ref{lemma:blowup} and~\ref{lemma:2normdiff} below formalize this approach.

\begin{lemma}\label{lemma:blowup}
Suppose that $T_1, T_2$ are two stochastic matrices of order $n,$ and let $E=T_1-T_2$. Then we can write $E$ as $E=\epsilon \tilde{E},$ where: $||\tilde{E}||_\infty = 2, 0 \le \epsilon := \frac{||E||_\infty}2 \le 1$ and $\tilde{E} = \tilde{T}_1-\tilde{T}_2$ for stochastic matrices $\tilde{T}_1, \tilde{T}_2.$
\end{lemma} 
\begin{proof}
	Without loss of generality we assume that $T_1 \ne T_2.$ 
	Define the following nonnegative matrices $\tilde{E}^+, \tilde{E}^-$ via  
$\tilde{E}^+_{j,k} = \max(\tilde{E}_{j,k}, 0), j,k=1, \ldots, n$ and 
$\tilde{E}^-_{j,k} = -\min(\tilde{E}_{j,k}, 0), j,k=1, \ldots, n.$ Then $\tilde{E}= \tilde{E}^+ - \tilde{E}^-$ and observe that for each $j=1, \ldots, n, 1 \ge  ||e_j^{\top} \tilde{E}^+||_1 =  ||e_j^{\top} \tilde{E}^-||_1.$ 

Now, we construct  $\tilde{T}_1, \tilde{T}_2$ as follows: for each $j=1, \ldots, n,$ we  set $e_j^{\top} \tilde{T}_1 = e_j^{\top} \tilde{E}^+ + (1-||e_j^{\top} \tilde{E}^+||_1)e_1^{\top}$ and $e_j^{\top} \tilde{T}_2 = e_j^{\top} \tilde{E}^- + (1-||e_j^{\top} \tilde{E}^-||_1)e_1^{\top}.$ It is now readily verified that $\epsilon, \tilde{E}, \tilde{T}_1, \tilde{T}_2$ have the desired properties. 
\end{proof}

		\begin{lemma}\label{lemma:2normdiff} Suppose that  $T_1, T_2$ are two stochastic matrices of order $n,$ and write  $E=T_1-T_2 = \epsilon\tilde E$, where $\epsilon = ||E||_\infty / 2$. Then $\sigma_1(E) \le 2 \epsilon \sqrt{n}. $ Equality holds if and only if $E = \epsilon \1(e_{j_1}^{\top}-e_{j_2}^{\top})$ for distinct indices $j_1, j_2.$ 
			\end{lemma} 
			\begin{proof}
				Write $\tilde{E} =  \tilde{T}_1-\tilde{T}_2,$ where $ \tilde{T}_1, \tilde{T}_2$ are as in Lemma \ref{lemma:blowup}. Then $\sigma_1(E) = \epsilon \sigma_1(\tilde{E}) \le \epsilon (\sigma_1(\tilde{T}_1) +\sigma_1(\tilde{T}_2)  ) \le 2 \epsilon \sqrt{n},$ the last inequality following from Lemma \ref{2norm}. 
				
				If $\sigma_1(E) = 2 \epsilon \sqrt{n},$ then again by Lemma \ref{2norm}, there are indices  $j_1, j_2$ such that $\tilde{T}_1 = \1 e_{j_1}^{\top}, \tilde{T}_2 = \1 e_{j_2}^{\top},$ so that $E = \epsilon \1(e_{j_1}^{\top}-e_{j_2}^{\top}).$ Conversely if $E = \epsilon \1(e_{j_1}^{\top}-e_{j_2}^{\top}),$ we find readily that 
				$\sigma_1(E) = 2 \epsilon \sqrt{n}.$
			\end{proof}

Our main result in this section shows that if we have a stochastic matrix that is a direct sum of $k$ irreducible stochastic matrices, then for any nearby stochastic matrix, the corresponding Laplacian matrix must have $k$ small singular values (one of which is necessarily $0$).

\begin{theorem}
	\label{small_sing}
	Suppose that $S$ is a stochastic matrix of order $n$ that is a direct sum of $k \ge 2$ irreducible stochastic matrices. Suppose that $T=S+E$ is another stochastic matrix, and write $E = \epsilon \tilde{E},$ where $\epsilon = ||E||_\infty / 2$. Then $I-T$ has $k$ singular values that are bounded above by $2 \epsilon \sqrt{n}.$ 
\end{theorem} 
\begin{proof}
	Without loss of generality, suppose that the leading $m \times m$ principal submatrix of $S$ is stochastic, and let $x_1=\frac{1}{\sqrt{m}} \left[\begin{array}{c}\1_m\\ 0_{n-m}\end{array}\right].$ Observe that $||(I-T)x_1||_2^2 =||(I-S-E)x_1||_2^2 = ||Ex_1||_2^2 \le \sigma_1(E)^2 \le 4 \epsilon^2 n,$ the last inequality by Lemma~\ref{lemma:2normdiff}. We find similarly that there are $k$ orthonormal vectors $x_1, \ldots, x_k$ such that $||(I-T)x_j||_2^2  \le 4 \epsilon^2 n, j=1, \ldots, k.$ The conclusion follows. 
\end{proof}


\begin{remark}\rm{
The best possible blowup factor in Lemma~\ref{lemma:blowup} is given by the \emph{Minkowski norm} (see, e.g.\ \cite[Chapter~2]{koldobsky2005fourier}) of the convex body $K_n := \{E \in \R^{n \times n} : E$ is the difference of two $n \times n$ sub-stochastic matrices$\}$.  This norm on $\R^{n \times n}$ by definition gives the maximum blowup of a matrix $E$ that still allows it to be written as the difference of two (sub)stochastic matrices.  It can be shown that if $E$ is the difference of two stochastic matrices, then its Minkowski norm is upper bounded by $||E||_\infty / 2$; hence, it yields bounds that are at least as good as those given in Lemma~\ref{lemma:2normdiff} and Theorem~\ref{small_sing}.  However, it is unclear at present whether using the Minkowski norm leads to any significantly improved guarantees in those results, so we state them in terms of $||\cdot||_\infty$ for the sake of exposition.}
\end{remark}

The following result provides a theoretical underpinning for the stopping criterion in step 2 of Algorithm \ref{alg}. 

\begin{theorem} \label{thm:frob_sig} Let $T$ be an $n \times n$ stochastic matrix, and let $\sigma$ denote the second smallest singular value of $I-T$. Let $S$ be the transition matrix of a  completely decoupled Markov chain on $n$ states. Then $||T-S||_{\F} \ge \sigma.$ 
\end{theorem} 
\begin{proof}
	Without loss of generality we assume that $S=S_1 \oplus S_2,$ where $S_1,S_2$ are stochastic matrices of orders $k, n-k,$ respectively (here we admit the possibility that $S_1 $ and $S_2$ may themselves be direct sums). Partition $T$ conformally with $S$ as $T=\left[\begin{array}{c|c}T_{11}&T_{12}\\ \hline T_{21}&T_{22}\end{array} \right].$ 
	
	Consider the vector $z = \frac{1}{\sqrt{nk(n-k)}}\left[\begin{array}{c}(n-k)\1 \\ \hline -k \1\end{array}\right],$ 
		where $z$ is partitioned conformally with $T$. Observe that $||z||_2=1, z^{\top}\1=0,$ and  $(I-T)z=\sqrt{\frac{n}{k(n-k)}}\left[\begin{array}{c}T_{12}\1 \\ \hline -T_{21} \1\end{array}\right],$ the last following from the fact that $(I-T_{11})\1 = T_{12}\1, (I-T_{22})\1=T_{21}\1.$ Hence we find that $$  \frac{n}{k(n-k)}(||T_{12}\1||_2^2 +||T_{21}\1||_2^2  )=||(I-T)z||_2^2 \ge \sigma^2 .$$

	Next we observe that $||T-S||_{\F}^2 = ||T_{11}-S_1||_{\F}^2 + ||T_{22}-S_2||_{\F}^2 + ||T_{12}||_{\F}^2 + ||T_{21}||_{\F}^2.   $ Letting $\tilde T_{11}, \tilde T_{22}$ denote the dangling node fixes for $T_{11}, T_{22}$ respectively, it follows from Proposition \ref{prop:dang} below that  
 	$||T_{11}-S_1||_{\F}^2 \ge ||T_{11}-\tilde T_{11}||_{\F}^2 = ||\frac{1}{k}T_{12}J||_{\F}^2 $ and  
	$||T_{22}-S_2||_{\F}^2 \ge ||T_{2}-\tilde T_{22}||_{\F}^2 = ||\frac{1}{n-k}T_{21}J||_{\F}^2 .$ 
	Further, we find from the Cauchy--Schwarz inequality that $||T_{12}||_{\F}^2 \ge \frac{1}{n-k}||T_{12}\1||_2^2$ and $||T_{21}||_{\F}^2 \ge \frac{1}{k}||T_{21}\1||_2^2.$ Assembling these observations, we have 
	\begin{eqnarray*}
	&&||T-S||_{\F}^2 \ge \\
	 &&||\frac{1}{k}T_{12}J||_{\F}^2+ ||\frac{1}{n-k}T_{21}J||_{\F}^2 + \frac{1}{n-k}||T_{12}\1||_2^2 + \frac{1}{k}||T_{21}\1||_2^2 = \\
	&&\left(\frac{1}{k} + \frac{1}{n-k}\right)(||T_{12}\1||_2^2 + ||T_{21}\1||_2^2 ) \ge \\	
&& \sigma^2,
	\end{eqnarray*}
	 as desired. 
\end{proof}

\begin{example}{\rm{
	Suppose that $k \in \naturals$ and that $a \in (0,\frac{1}{2}),$ and consider the stochastic matrix $T$ of order $2k$ given by $T=\left[\begin{array}{c|c} \frac{1-a}{k} J & \frac{a}{k}J \\ \hline \frac{a}{k}J & \frac{1-a}{k} J
	\end{array}\right],$ where the diagonal blocks are $k \times k.$ It is straightforward to determine that the singular values of $I-T$ are $1$ (with multiplicity $2k-2$), $2a$, and $0$. 
Next, consider the matrix $S=\left[\begin{array}{c|c} \frac{1}{k} J &0 \\ \hline 0 & \frac{1}{k} J
\end{array}\right],$ which is the transition matrix of a completely decoupled Markov chain. Then $T-S =  \left[\begin{array}{c|c} -\frac{a}{k} J &\frac{a}{k} J \\ \hline \frac{a}{k} J & -\frac{a}{k} J
\end{array}\right],$ which has Frobenius norm equal to $2a$. Thus we see from this example that for the inequality in Theorem \ref{thm:frob_sig} it is possible for equality to hold. }}
\end{example}

\subsection{Small singular value implies clustering}\label{subsec:smallsing=>clustering}


In this subsection we show that for an irreducible stochastic matrix $T,$ if $I-T$  has a small positive singular value, then the corresponding left singular vector can be used to generate: a) a partitioning of $T$, and b) weight vectors associated with the partition such that the resulting coupling matrix~\cite[Definition~4.6]{fritzsche2008svd} is guaranteed to have at least one large diagonal entry.  The following theorem states this formally:

\begin{theorem}\label{thm:small_sing} 
	Suppose that $T$ is a stochastic matrix of order $n$ with $1$ as a simple eigenvalue. Consider the singular value decomposition of $I-T, I-T=U \Sigma V^\top.$ Let $u,v$ be left and right singular vectors (respectively) corresponding to the singular value $\sigma >0.$ Suppose further that 
	$u^{\top}=\left[\begin{array}{c|c|c}u_1^{\top} & -u_2^{\top}& 0^{\top} \end{array} \right],$ where $u_1 \in \R^\ell, u_2 \in \R^{m-\ell},$ and both are positive vectors. Then  $$\max_{j = 1, 2}\omega_{|u|}(T; j, j) = \max \left \{ \frac{u_1^{\top}T_{11}\1}{u_1^{\top}\1} ,\frac{u_2^{\top}T_{22}\1}{u_2^{\top}\1} \right \} \ge 1 - \sigma \sqrt{m} \ge  1 - \sigma \sqrt{n}.$$
\end{theorem}

To prove Theorem~\ref{thm:small_sing}, we will need the following technical lemma:

\begin{lemma}\label{max_lem}
	Suppose that $\ell, m \in \naturals$ with $1 \le \ell \le m-1.$ Then
	\begin{eqnarray*}
	& \min& \left \{ \max \left \{ \frac{x^{\top}\1}{\sqrt{\ell}}, \frac{y^{\top} \1}{\sqrt{m-\ell}} \right \} \Bigg| x \in \R^\ell, y \in \R^{m-\ell}, x, y \ge 0, x^{\top}x+y^{\top}y=1 \right \}\\ 
	 &=&\frac{1}{\sqrt{m}}.
	 \end{eqnarray*}
\end{lemma}
\begin{proof}
	Suppose that $x \in \R^\ell$ with $e_1^{\top}x=x_1>0, e_2^{\top}x= x_2 >0.$ Let $\tilde{x} =\sqrt{\frac{x^{\top}x}{x^{\top}x+ 2x_1x_2}}(x+ x_2(e_1-e_2)).$ Then $\tilde{x}^{\top}\tilde{x}= x^{\top}x$ and $\tilde{x}^{\top}\1 < x^{\top}\1.$ It follows that in order to attain our desired minimum, we can assume without loss of generality that each of $x,y$ has at most one positive entry. Hence for some $\alpha \in [0,1]$ we have $x=\alpha e_i, y=\sqrt{1-\alpha^2} e_j$ for some indices $i,j.$ It now follows that
	\begin{eqnarray*}
	&& \min \left \{ \max \left \{ \frac{x^{\top}\1}{\sqrt{\ell}}, \frac{y^{\top} \1}{\sqrt{m-\ell}} \right \} \Bigg| x \in \R^\ell, y \in \R^{m-\ell}, x, y \ge 0, x^{\top}x+y^{\top}y=1 \right \} = \\
	&&
	\min \left \{ \max \left \{ \frac{\alpha}{\sqrt{\ell}}, \frac{\sqrt{1-\alpha^2}}{\sqrt{m-\ell}} \right \} \Bigg| \alpha \in [0,1]  \right \}.
		\end{eqnarray*}
	 This last is readily seen to equal $\frac{1}{\sqrt{m}}.$
\end{proof}

\begin{proof}[Proof of Theorem~\ref{thm:small_sing}]
	Before beginning the proof, we note  that the subvector of zeros in the partitioning of $u$ maybe absent. In that case, the theorem and proof go  through with a partitioning into just two subsets, and $m=n.$

   Partition  $T$ and $v$  conformally with  $u$, as   
\begin{equation}\label{3part}
T=\left[\begin{array}{c|c|c} 
T_{11}& T_{12}&T_{13}\\ \hline T_{21}& T_{22}&T_{23}\\ \hline T_{31}& T_{32}&T_{33}
\end{array}\right], v=\left[\begin{array}{c} v_1\\ v_2 \\ v_3 \end{array} \right]. 
\end{equation} 
Since $u^{\top} (I-T) = \sigma v^{\top},$ we find that $u_1^{\top}(I-T_{11})+ u_2^{\top}T_{21} = \sigma v_1^{\top}.$ Hence, $u_1^{\top}(I-T_{11}) \le \sigma v_1^{\top}, $ and we deduce that 
$u_1^{\top} \1 - u_1^{\top} T_{11}\1 \le \sigma v_1^{\top}\1,$ which yields $\frac{u_1^{\top} T_{11}\1}{u_1^{\top} \1} \ge 1 - \sigma \frac{v_1^{\top} \1}{u_1^{\top} \1}.$ Similarly we find that  $\frac{u_2^{\top} T_{22}\1}{u_2^{\top} \1} \ge 1 - \sigma \frac{v_2^{\top} \1}{u_2^{\top} \1}.$  
Since $v_1 \in \R^\ell, v_2 \in \R^{m-\ell},$ it follows from the Cauchy--Schwarz 
 inequality that  
 $\frac{u_1^{\top}T_{11}\1}{u_1^{\top}\1} \ge 1 - \sigma \frac{\sqrt{\ell}}{u_1^{\top}\1}$ and   $\frac{u_2^{\top}T_{22}\1}{u_2^{\top}\1} \ge  1 - \sigma \frac{\sqrt{m-\ell}}{u_2^{\top}\1}.$ From  Lemma \ref{max_lem}, we find  that  $\max \left \{ \frac{u_1^{\top}\1}{\sqrt{\ell}},  \frac{u_2^{\top}\1}{\sqrt{m-\ell}}   \right \} \ge \frac{1}{\sqrt{m}},$ and hence  
$\min \left \{\frac{\sqrt{\ell}}{u_1^{\top}\1}  , \frac{\sqrt{m-\ell}}{u_2^{\top}\1}  \right  \} \le \sqrt{m}.$ We deduce that $\max \left \{ \frac{u_1^{\top}T_{11}\1}{u_1^{\top}\1} ,\frac{u_2^{\top}T_{22}\1}{u_2^{\top}\1} \right \} \ge 1 - \sigma \sqrt{m} \ge  1 - \sigma \sqrt{n},$ as desired. 
\end{proof}

\begin{remark}{\rm{
In the context of Theorem \ref{thm:small_sing}, we may produce an associated coupling matrix of order $n-m+2$ as follows. Partition $T$ by taking the first $\ell$ indices as one subset $S_1$, the next $m-\ell$ indices as a second subset $S_2$, and the remaining  $n-m$ indices as singletons. The weight vector associated with $S_1$ is $\frac{1}{u_1^{\top}\1}u_1,$ the weight vector associated with $S_2$ is $\frac{1}{u_2^{\top}\1}u_2,$ and the weight vectors for the singleton subsets are all $1$. The corresponding coupling matrix is then equal to
$$C=\left[\begin{array}{ccc}\frac{u_1^{\top}  T_{11}\1}{u_1^{\top}\1}& \frac{u_1^{\top}  T_{12}\1}{u_1^{\top}\1}& \frac{1}{u_1^{\top}\1} u_1^{\top}T_{13}\\
\frac{u_2^{\top}  T_{21}\1}{u_2^{\top}\1}& \frac{u_2^{\top}  T_{22}\1}{u_2^{\top}\1}& \frac{1}{u_2^{\top}\1} u_2^{\top}T_{23} \\ 
 T_{31}\1& T_{32}\1&T_{33}
\end{array}\right].
$$
Theorem \ref{thm:small_sing} then guarantees that at least one of the first two diagonal entries of $C$ is bounded below by $1 - \sigma \sqrt{n}.$ In the special case that $u$ has no zero entries, our coupling matrix is $2 \times 2$ and equal to 
$$\left[\begin{array}{cc}\frac{u_1^{\top}  T_{11}\1}{u_1^{\top}\1}& \frac{u_1^{\top}  T_{12}\1}{u_1^{\top}\1}\\
\frac{u_2^{\top}  T_{21}\1}{u_2^{\top}\1}& \frac{u_2^{\top}  T_{22}\1}{u_2^{\top}\1} 
\end{array}\right], 
$$
with the same lower bound on the maximum of the  diagonal entries. }}
\end{remark}

The next example shows that while one of the diagonal entries in the coupling matrix is guaranteed to exceed $ 1 - \sigma \sqrt{n},$ it can be the case that the other diagonal entry in the coupling matrix is less than $ 1 - \sigma \sqrt{n}.$

\begin{example}\label{ex:smalldiag}{\rm{
Consider the following stochastic matrix of order $n \ge 4$: $$	
T = \left[ \begin{array}{c|c|c} 
1-\epsilon & \epsilon & 0^{\top}\\ \hline 0&1-\delta&\frac{\delta}{n-2} \1^{\top} \\ 
\hline x\1&y\1&(1-x-y)I
\end{array}\right].$$
 Here we consider the case that $0< \delta \ll \epsilon \ll x < y$, and  $x+y<1$.  It can be shown that the smallest nonzero singular value of $I - T$ is $\sigma_{n-1} = 
\frac{\epsilon (x+y)\sqrt{n}}{\sqrt{(n-1)(x^2+y^2)+2xy}} + O(\epsilon^2)$, with corresponding left singular vector $$\left[ \begin{array}{c} 1\\
-\delta(nx(x+y)-\sigma_{n-1}^2  )\\ 
-[\delta^2(\frac{x+(n-1)y}{n-2}) + (y-x)(\sigma_{n-1}^2 - \delta^2(\frac{n-1}{n-2})) ]\1
\end{array}\right].$$
Hence, the $(2, 2)$ entry of the coupling matrix is $\mu = 1-x+O(\epsilon^2).$ On the other hand, $$1-\sqrt{n} \sigma_{n-1} = 1 - \frac{\epsilon (x+y)n}{\sqrt{(n-1)(x^2+y^2)+2xy}} + O(\epsilon^2).$$ Since  $\epsilon$ is small, 
$1-x < 1 - \frac{\epsilon (x+y)n}{\sqrt{(n-1)(x^2+y^2)+2xy}} + O(\epsilon^2),$ so that $\mu < 
1-\sqrt{n} \sigma_{n-1}. $  See Appendix~\ref{sec:smalldiag} for details}}
\end{example}

\begin{remark}{\rm{ Maintaining the notation of Theorem \ref{thm:small_sing}, we now derive lower bounds on the spectral radii of $T_{11}$ and $T_{22}$.  We first remark that, as discussed in~\cite[Appendix~6.1]{CKS}, a large spectral radius of one of these submatrices is evidence of clustering.
Let $x_1, x_2$ denote Perron vectors of $T_{11}$ and $T_{22}$ respectively, normalised so that $u_1^{\top}x_1=1, u_2^{\top}x_2=1.$ Denote the spectral radii of $T_{11} $ and $T_{22}$ by $\rho_1, \rho_2,$ respectively. Referring to the proof of Theorem \ref{thm:small_sing}, we have  $u_1^{\top}T_{11} \ge u_1^{\top}-  \sigma v_1^{\top} $ and 
$u_2^{\top}T_{22} \ge u_2^{\top}-  \sigma v_2^{\top}.$ Multiplying on the right by $x_1, x_2$ respectively now yields 
\begin{eqnarray*}
\rho_1 = u_1^{\top}T_{11}x_1 \ge u_1^{\top}x_1-  \sigma v_1^{\top}x_1 = 1- \sigma v_1^{\top}x_1,\\
\rho_2 = u_2^{\top}T_{22}x_2 \ge u_2^{\top}x_2-  \sigma v_2^{\top}x_2 = 1- \sigma v_2^{\top}x_2. 
\end{eqnarray*}
In particular, if $\sigma$ is sufficiently small and $v_1^{\top}x_1, v_2^{\top}x_2$ are not too large, then $\rho_1$ and $\rho_2$ are close to $1$.
}}
\end{remark}

\subsection{Diagonally dominant coupling matrix implies nearly decoupled}\label{subsec:couplingmatrix}

In this section, we show that if a Markov chain is nearly decoupled in the sense of~\cite[Definition~4.6]{fritzsche2008svd}, then it is nearly decoupled in the sense discussed in Section~\ref{subsec:nearlydecoupled=>smallsing}.  The following lemma shows that the transition matrix of such a Markov chain is a perturbation of a direct sum of \emph{sub-stochastic} matrices.


\begin{lemma}\label{lemma:sub-stochastic}
Let $T$ be nearly decoupled in the sense of~~\cite[Definition~4.6]{fritzsche2008svd}, i.e.\ $T = [T_{kl}]_{k, l = 1}^m$ (after permuting the rows and columns), with $\omega_{\1}(T; k, k) \geq 1 - \delta$ for $k = 1, \ldots, m$ for some $\delta \in (0, 1)$.  Then $T = B + E$, where $B = B_1 \oplus \ldots \oplus B_m$, $B_k$ is sub-stochastic with average row sum $\geq 1 - \delta$, and $E$ is small: specifically,
\begin{enumerate}
\item	$||E||_{\mathrm F} \leq \min\{\sqrt{\delta n}, \delta n\}$.
\item	$||E||_\infty \leq \min\{1, \delta n\}$.
\end{enumerate}
\end{lemma}

\begin{proof}
Simply let $B_1, \ldots, B_m$ be the diagonal blocks of $T$, and let $E$ be the off--diagonal blocks (with 0s on the diagonal blocks).  I.e., $B_k := T_{kk}$ and $E = T - B$.  The bound on the average row sum of $B_k$ follows easily from the definition of $\omega_{\1}$.

Now let $S_1, \ldots, S_n \subseteq [n]$ be the sets of indices corresponding to $T_{11}, \ldots, T_{mm}$, and let $n_k := |S_k|$.  Then by definition of $\omega_{\1}$, we have
\[\omega_{\1}(B; k, k) = \frac1{n_k}\1_{S_k}^\top T\1_{S_k} \geq 1 - \delta\]
and 
\begin{equation}\label{eq:diagdom}
\sum_{l \neq k}\omega_{\1}(T; k, l) = \frac1{n_k}\sum_{i \in S_k, j \notin S_k}t_{ij} = 1 - \omega_{\1}(B; k, k) \leq \delta.
\end{equation}
Now let $E_k := [T_{kl}]_{l \neq k} \in \R^{n_k \times (n - n_k)}$.  Then $\1^\top E_k\1 \leq \delta n_k$ by~\eqref{eq:diagdom}, and of course $E_k\1 \leq \1$.

\begin{enumerate}
\item	Let us maximize $||E_k||_\F^2$ subject to the above constraints.  First consider the case when $\delta n_k \geq 1$.  $||E_k||_\F^2$ is maximized by putting a single 1 in $\lfloor\delta n_k\rfloor$ rows of $E_k$, setting one entry of an additional row to $\delta n_k - \lfloor\delta n_k\rfloor$ if necessary, and setting all remaining entries to 0.  Thus, we get $||E_k||_\F^2 \leq \delta n_k \leq \delta^2 n_k^2$.  If $\delta n_k \leq 1$, then $||E_k||_\F^2$ is maximized by setting a single entry of $E_k$ to $\delta n_k$ and the rest 0.  In this case we get $||E_k||_\F^2 \leq \delta^2n_k^2 \leq \delta n_k$.  In both cases we get $||E_k||_\F^2 \leq \min\{\delta n_k, \delta^2 n_k^2\}$.

Now, since $||E||_\F^2 = \sum_{k = 1}^m||E_k||_\F^2$ and $n = \sum_{k = 1}^mn_k$, we have
\[||E||_\F^2 \leq \min\{\delta n, \delta^2(n_1^2 + \ldots + n_m^2)\},\]
hence
\[||E||_\F \leq \min\{\sqrt{\delta n}, \delta\sqrt{n_1^2 + \ldots + n_m^2}\} \leq \min\{\sqrt{\delta n}, \delta n\}.\]

\item	Now let us maximize $||E_k||_\infty$ subject to $\1^\top E_k\1 \leq \delta n_k$ and $E_k\1 \leq \1$.  As $E_k \geq 0$ we have $||E_k||_\infty \leq \1^\top E_k\1 \leq \delta n_k$, and as $E_k\1 \leq \1$ we have $||E_k||_\infty \leq 1$.  Hence, $||E_k||_\infty \leq \min\{1, \delta n_k\}$.  As $||E||_\infty = \max_{k = 1}^m ||E_k||_\infty$, we have
\[||E||_\infty \leq \min\left\{1, \delta \cdot \max_{k = 1}^mn_k\right\} \leq \min\{1, \delta n\}.\qedhere\]
\end{enumerate}
\end{proof}

\begin{remark}\rm{
The results in this section require \emph{all} diagonal entries of the coupling matrix $W_\1$ to be large.  This is consistent with the notion of clustering defined in~\cite[Definition~4.6]{fritzsche2008svd}.

}
\end{remark}

The above result has the limitation that the diagonal blocks are sub-stochastic, not stochastic.  The following theorem shows that with slight modification we can make the diagonal blocks stochastic.

\begin{theorem}\label{thm:directsum_stochastic}
Let $T$ be nearly decoupled in the sense of~\cite[Definition~4.6]{fritzsche2008svd}, i.e.\ $T = [T_{kl}]_{k, l = 1}^m$ (after permuting the rows and columns), with $\omega_{\1}(T; k, k) \geq 1 - \delta$ for $k = 1, \ldots, m$ for some $\delta \in (0, 1)$.  Let $S_k$ denote the set of indices corresponding to $T_{kk}$, $n_k := |S_k|$, and let $r_k$ be the minimum row sum of $T_{kk}$.  Then we can write $T = B + E$, where $B$ is a direct sum of stochastic matrices $B_1, \ldots, B_m$, and 
\begin{enumerate}
\item	$||E||_\infty \leq 2 \cdot \min\{1, \delta n\}$.
\item	$||E||_\F \leq \min\{2\delta n, \sqrt{\delta^2n^2 + \delta m}, \sqrt{\delta n + \delta m}\}$.
\end{enumerate}
\end{theorem}

\begin{proof}
For $k = 1, \ldots, m$, define $B_k$ to be 
$\dnf(T_{kk})$ (see Section~\ref{subsec:dnf}), and let $B=B_1\oplus \ldots \oplus B_m.$  Set $E := T - B$ and partition $E = [E_{kl}]_{k, l = 1}^m$ conformally with $T$.  Hence, we have 
\[E_{kl} = \left\{
\begin{array}{ll}
-(B_k - T_{kk})	& \textrm{if } k = l,	\\
T_{kl}			& \textrm{else.}
\end{array}
\right.\]
Let $S_k \subseteq [n]$ be the set of indices corresponding to $B_k$ and $n_k := |S_k|$.  Define $E_k := [E_{k1} \ldots E_{km}] \in \R^{n_k \times n}$.  
\begin{enumerate}
\item	For the bound on $||E||_\infty$, observe that $E$ is the difference of two stochastic matrices, so $E_k\1 = \0$, $E_{kk} \leq 0$, $E_{kl} \geq 0$ for $l \neq k$, $-E_{kk}\1 \leq \1$, and $\sum_{l \neq k}E_{kl}\1 \leq \1$.  Hence,
\[||E_k||_\infty = 2||E_{kk}||_\infty = 2(1 - r_k),\]
where $r_k := \min_{i \in S_k}\e_i^\top T\1_{S_k}$ is the minimum row sum of $T_{kk}$.  
On the other hand, $||E_{kk}||_\infty$ is at most the sum of the absolute values of all its entries.  As $\omega_{\1}(T; k, k) \geq 1 - \delta$, this is at most $\delta n_k$.  Thus, we have 
\[||E_k||_\infty \leq 2 \cdot \min\{1 - r_k, \delta n_k\}.\]

\item	For the bound on $||E||_\F$, first observe that $||T-B||_\F \le ||T-(T_{11}\oplus \ldots \oplus  T_{mm})||_\F+ ||(T_{11}\oplus \ldots \oplus  T_{mm})-B||_\F.$ From Proposition~\ref{prop:dang} below we find that $||(T_{11}\oplus \ldots \oplus  T_{mm})-B||_\F \le ||T-(T_{11}\oplus \ldots \oplus  T_{mm})||_\F,$ so that $||T-B||_\F \le 2||T-(T_{11}\oplus \ldots \oplus  T_{mm})||_\F.$ Applying 
 Lemma~\ref{lemma:sub-stochastic} now yields  
\begin{equation}\label{eq:dnfub}
||T-B||_\F \le 2   \cdot \min\{\delta n, \sqrt{\delta n}\}.
\end{equation}

  On the other hand, observe that the off--diagonal blocks of $E$ are the same as those constructed in the proof of Lemma~\ref{lemma:sub-stochastic}.  Hence, they contribute at most $\min\{\delta n, \sqrt{\delta n}\}$ to $||E||_\F$.  It remains to estimate the contribution from the diagonal blocks of $E$.  Letting $s_i$ be the sum of the $i$th row of $T_{kk}$, by the definition of the dangling node fix, the contribution to $||E||_\F^2$ from $E_{kk}$ is given by 
\begin{eqnarray*}
\frac1{n_k} \cdot ||(I - T_{kk})\1||_2^2	& =	& \frac1{n_k}\sum_{i \in S_k}(1 - s_i)^2	\\
	& \leq	& \frac1{n_k}\sum_{i \in S_k}(1 - s_i)	\textrm{ (since $0 \leq 1 - s_i \leq 1$ for all $i$)}	\\
	& =		& \frac1{n_k}\left(n_k - \sum_{i \in S_k}s_i\right)	\\
	& \leq	& \frac1{n_k}(n_k - (1 - \delta)n_k)	\\
	& =		& \delta.
\end{eqnarray*}
Hence, the total contribution from the diagonal blocks is at most $\delta m$.  Thus, we get 
\[||E||_\F^2 \leq \min\{\delta^2n^2, \delta n\} + \delta m.\]
Combining this with~\eqref{eq:dnfub} we get the desired bound.\qedhere
\end{enumerate}
\end{proof}

\begin{remark}\label{rmk:1vsLIWV}\rm{
Observe that the results in this section pertain to the coupling matrix with respect to the ones vector, $W_\1$, while Theorem~\ref{thm:small_sing} gives a lower bound on the diagonal entries of $W_{|u|}$, where $u$ is a left singular vector corresponding to the second smallest singular value.  At present, we have not been able to prove that such a lower bound yields a perturbation of a completely decoupled Markov chain (in the sense of Theorem~\ref{thm:directsum_stochastic}).
}

\rmk{
\label{rmk:n^-3/2}
Based on Theorem~\ref{small_sing}, in order to get a meaningful bound on the small singular values of $I - T$ we need $||E||_\infty\sqrt n \ll 1$.  Thus, in order to guarantee this we need $\delta = O(n^{-3/2})$ in Theorem~\ref{thm:directsum_stochastic}.  
This seems like a rather strong requirement to have the average rows sums of the diagonal blocks all be at least $1 - O(n^{-3/2})$.  In particular, if $\delta < 1 / n_k$, then $T_{kk}$ can't have any zero rows (otherwise the average row sum could not be $\geq 1 - \delta$).
}

\end{remark}

\subsection{The dangling node fix}\label{subsec:dnf}


In this subsection, we will show that if we exactly recover a cluster, delete, then apply the dangling node fix (Definition~\ref{defn:dnf}), the error does not blow up too much.  First, we show that for any sub-stochastic matrix, the corresponding dangling node fix is the closest stochastic matrix with respect to several standard norms. 

\begin{proposition}\label{prop:dang} 
	Suppose that $T$ is a sub-stochastic matrix of order $n$.  Then for any $n \times n$ stochastic matrix $S$:
\begin{enumerate}[a)]
	\item	$||T-\dnf(T)||_\infty \le ||T-S||_\infty;$ 
	\item	$||T-\dnf(T)||_2 \le ||T-S||_2;$ 
	\item	$||T-\dnf(T)||_{\F} \le ||T-S||_{\F}.$
\end{enumerate}
\end{proposition}
\begin{proof}
Let $\tilde T = \dnf(T)$.
a):	Evidently $T-\tilde{T} = \frac{1}{n}(J-TJ),$ which is nonnegative with row sums $1-\sum_{k=1}^n t_{jk}, j=1, \ldots, n.$ In comparison, we note that for each $j=1, \ldots, n, \sum_{k=1}^n |s_{jk}-t_{jk}| \ge  \sum_{k=1}^n (s_{jk}-t_{jk}) = 1-\sum_{k=1}^n t_{jk}.$ It now follows that $  ||T-S||_\infty \ge ||T-\tilde{T}||_\infty . $\\
b) and c): First note that $(T-\tilde{T})(T-\tilde{T})^\top = \frac{1}{n^2}(J-TJ)(J-TJ)^\top = \frac{1}{n} (\1 - T\1)(\1-T\1)^\top.$ We find readily that $||T-\tilde{T}||_2^2 = ||T-\tilde{T}||_{\F}^2 = n-2\1^\top T \1 + \1^\top T^\top T \1.$ 

Next we note that $(T-S)(T-S)^\top = T^\top T -S^\top T-T^\top S + S^\top S.$ Hence $||T-S||_2^2 \ge \frac{1}{n} \1^\top (T^\top T -S^\top T-T^\top S + S^\top S)\1 = n-2\1^\top T \1 + \1^\top T^\top T \1.$ We now deduce that 
$||T-S||_{\F} \ge ||T-S||_2 \ge ||T-\tilde{T}||_2 = ||T-\tilde{T}||_{\F},$ as desired.  

 \end{proof}

The following lemma will allow us to prove that if we have a Markov chain whose transition matrix is a perturbation of that of a completely decoupled Markov chain, then if we partition the perturbed chain conformally and perform the dangling node fix on each diagonal block, then we do not increase the difference in several norms between the perturbed matrix and the decoupled matrix.

\begin{lemma}\label{thm:noblowup}
Let $T_0$ and $T$ be $n \times n$ stochastic matrices.  Suppose that $T_0[S]$ is also stochastic, and let $\tilde T = \dnf(T[S])$.  Then $||\tilde T - T_0[S]||_2 \leq ||T - T_0||_2,$  $||\tilde T - T_0[S]||_\F \leq ||T - T_0||_\F,$ and $||\tilde T - T_0[S]||_\infty \leq ||T - T_0||_\infty.$ 
\end{lemma}

\begin{proof}
Let $E = T - T_0$.  By Definition~\ref{defn:dnf} we get $\tilde T = T_0 + E[S]\left(I - \frac1{|S|}J\right)$.  Thus, let us define
\begin{equation*}
\tilde E := E[S]\left(I - \frac1{|S|}J\right).
\end{equation*}
Then we have
\[\tilde E\tilde E^\top = E[S]\left(I - \frac1{|S|}J\right)E[S]^\top = E[S]E[S]^\top - \frac1{|S|}E[S]JE[S]^\top.\]
As $-\frac1{|S|}E[S]JE[S]^\top \preceq 0$, Weyl's inequalities give
\[\lambda_i(\tilde E\tilde E^\top) = \sigma_i(\tilde E)^2 \leq \lambda_i(E[S]E[S]^\top) = \sigma_i(E[S])^2\]
for $i = 1, \ldots, n$.  Hence, by interlacing we have 
\[||\tilde E||_2 = \sigma_1(\tilde E) \leq \sigma_1(E[S]) \leq \sigma_1(E) = ||E||_2,\]
and moreover
\[||\tilde E||_\F^2 = \sum_{i = 1}^n\sigma_i(\tilde E)^2 \leq \sum_{i = 1}^n\sigma_i(E[S])^2 = ||E[S]||_\F^2 \leq ||E||_\F^2.\]
The last inequality follows from the fact that the Frobenius norm of a matrix is at least the Frobenius norm of any submatrix. 

Finally, we note that $\tilde{E}=E[S] - E[S]\frac1{|S|}J = E[S] + \frac1{|S|}E[S,S^c]\1 \1^{\top}.$ Hence, for any index $k$ with $1 \le k \le |S|,$ we have 
$||e_k^{\top} \tilde{E}||_1 \le ||e_k^{\top} E[S]||_1 + \frac1{|S|}||e_k^{\top}E[S,S^c]\1 \1^{\top}||_1 = ||e_k^{\top} E[S]||_1 + |e_k^{\top}E[S,S^c]\1| \le ||e_k^{\top} E[S]||_1 + ||E[S,S^c]||_1 = ||e_k^{\top}E||_1.$ It now follows readily that $ ||\tilde E||_\infty \le ||E||_\infty.$  
\end{proof}

Lemma~\ref{thm:noblowup} then yields the following as a direct corollary:

\begin{theorem}\label{cor:noblowup}
Let $T_i$ be a stochastic matrix of order $n_i$ with index set $S_i$ for $i = 1, \ldots, k$, where $n_1 + \ldots + n_k = n$, and let $T = \bigoplus_{i = 1}^kT_i + xE$.  Assume that $T$ is stochastic.  Let $\emptyset \subset \mathcal I \subset [k]$ and $S := \bigcup_{i \in \mathcal I}S_i$.  Then $\dnf(T[S]) = \bigoplus_{i \in \mathcal I}T_i + x\tilde E$, where $||\tilde E||_2 \leq ||E||_2,$ $||\tilde E||_{\F} \leq ||E||_{\F},$ and $||\tilde E||_\infty \leq ||E||_\infty$.
\end{theorem}

\begin{proof}
Apply Lemma~\ref{thm:noblowup} with $T_0 = \bigoplus_{i \in \mathcal I}T_i$.
\end{proof}

\begin{remark}\rm{
	Here we consider the behaviour of the dangling node fix when it is iterated on principal submatrices of a stochastic matrix. To be concrete, suppose that we have a stochastic matrix $T$ partitioned as $$T=\left[\begin{array}{c|c|c} T_{11} &T_{12}&T_{13}\\ \hline T_{21} &T_{22}&T_{23}\\ \hline T_{31} & T_{32}&T_{33}\end{array} 	
	\right],$$ 
	where the diagonal blocks are of sizes $k, \ell,$ and $m,$ respectively. Applying the dangling node fix to $T_{11}$ is easily seen to yield the $k \times k$ stochastic matrix $T_{11}+\frac{1}{k}(T_{12}\1+T_{13}\1)\1^{\top}.$ 
	
	Next we consider what happens when we i) apply the dangling node fix to the leading $(k+\ell)\times(k+\ell)$ submatrix of $T$ to generate a stochastic matrix $S$, then ii) apply the dangling node fix to the leading $k \times k$ submatrix of $S$. For i), we find that 
	\begin{eqnarray*}
	S &=& \left[\begin{array}{c|c} T_{11} &T_{12}\\ \hline T_{21} &T_{22}\end{array} 	
	\right] + \frac{1}{k+\ell} \left[\begin{array}{c}T_{13}\1 \\ \hline T_{23}\1 \end{array}\right]\1^{\top} \\
	&=& \left[\begin{array}{c|c} T_{11}+ \frac{1}{k+\ell}T_{13}\1 \1^{\top}  &T_{12}+ \frac{1}{k+\ell}T_{13}\1 \1^{\top}\\ \hline T_{21}+ \frac{1}{k+\ell}T_{23}\1 \1^{\top} &T_{22}+ \frac{1}{k+\ell}T_{23}\1 \1^{\top}\end{array} 	
	\right].
\end{eqnarray*}
	
	Since the leading $k \times k$ principal submatrix of $S$ is 
	$T_{11}+ \frac{1}{k+\ell}T_{13}\1 \1^{\top},$ applying the corresponding dangling node fix in ii) yields 
	$$T_{11}+ \frac{1}{k+\ell}T_{13}\1 \1^{\top} + \frac{1}{k}\left(T_{12}\1 + \frac{\ell}{k+\ell}T_{13}\1\right)\1^{\top} = T_{11}+\frac{1}{k}(T_{12}\1+T_{13}\1)\1^{\top}.$$  
	  
	 Thus we find that the dangling node fix has the appealing property that when it is  iteratively applied to nested principal submatrices,  say on index sets $S_0$ and $S_1$ with $S_1 \subset S_0,$ the effect is the same as having applied the dangling node fix to the original submatrix corresponding to the index set $S_1$. 
	  }
\end{remark}

\section{Numerical results}\label{sec:numerical}

In this section we observe our algorithm's performance on various datasets, both real and simulated. All of our tests were run in MATLAB\textsuperscript{\textregistered} Version~9.8.0.1323502~(R2020a) on a Mac~Pro\textsuperscript{\textregistered} 5,1 with 24 6-core Intel\textsuperscript{\textregistered} Xeon\textsuperscript{\textregistered} CPU E5645 2.40GHz processors, running Ubuntu\textsuperscript{\textregistered}~20.04.1.  The relative machine precision was $\texttt{eps} = 2.2204 \times 10^{-16}$.  All of our code is available to download at \url{https://github.com/smpcole/clustered-markov}.

Images were produced in Python\textsuperscript{\textregistered} Version~3.8.10 using Matplotlib Version~3.1.2. All images are vector (rather than raster) images; thus, we invite the reader to zoom in without encountering pixelated images.
\subsection{Airport network} \label{air}

Here we report the results of running Algorithm~\ref{alg} on a data set arising from a network of airports. 
The data set was downloaded from \url{https://toreopsahl.com/datasets/}, and is based on  data available  from the Complex Networks Collaboratory  
\url{https://sites.google.com/site/cxnets/usairtransportationnetwork}. The 
 network represents traffic between $500$ commercial airports in the United States; vertices correspond to the individual airports, and the weight of an edge between two vertices is the number of seats available on flights between the corresponding airports. This results in a weighted undirected graph. From the weighted adjacency matrix of that graph, we normalised each row by dividing by its sum, thus creating a stochastic matrix $T$. Note that $T$ is the transition matrix of the Markov chain arising from a \emph{random walk} on the graph representing the network.

Using the tolerance $\tau =0.15$, the algorithm yields the following results. There are seven clusters, of sizes  $37, 5, 27, 24, 17, 54$ and  $336,$  and the associated coupling matrix is 
$$W_u(T) = 
\left[\begin{array}{ccccccc} 
    0.6857 &   0.0000 &        0 &   0.0217&         0&    0.0245 &   0.2681\\
0.0019    &0.9981      &   0       &  0        & 0        & 0       &  0\\
0         &0 &   0.8832  &  0.1153 &   0.0015    &     0&         0\\
0.0070     &    0   & 0.0014  &  0.9881   & 0.0006   & 0.0006 &   0.0023\\
0         &0 &   0.0020 &   0.1642 &   0.8338&         0       &  0\\
0.1814     &    0       &  0 &   0.0517      &   0    &0.7645   & 0.0024\\
0.0860      &   0        & 0&    0.0075       &  0   & 0.0002    &0.9064
\end{array}\right]
$$
(here $u$ is the left-iterative weight vector computed by Algorithm \ref{alg}). 

We also computed the spectral radii of the principal submatrices of $T$ corresponding to the clusters, and these are  $0.6647, 0.9676, 0.8725, 0.8419, 0.8200, 0.7829$ and $ 0.8874,$ respectively. Observe that these spectral radii are well--correlated with the diagonal entries of $W_u(T)$. 

For comparison purposes, we computed the coupling matrix that arises when the weight vector for each cluster is an all--ones vector, with the following result: 
$$W_{\1}(T)= \left[\begin{array}{ccccccc} 
   0.7826&    0.0000     &    0 &   0.0088     &    0&    0.0191&    0.1894\\
0.0044    &0.9956       &  0     &    0       &  0       &  0     &    0\\
0         &0 &   0.9501 &   0.0485 &   0.0014 &        0   &      0\\
0.0381     &    0&    0.0137   & 0.9139 &   0.0022  &  0.0023&    0.0298\\
0        & 0   & 0.0044    &0.1039   & 0.8917     &    0       &  0\\
0.6353    &     0   &      0&    0.0146 &        0 &   0.3206  &  0.0295\\ 
0.0353     &    0  &       0 &   0.0036&         0  &  0.0016 &   0.9595
\end{array}\right]. 
$$
The most striking difference between $W_u(T)$ and  $W_{\1}(T)$ is in the sixth rows. This may be explained by the fact that the subvector of $u$ corresponding to  the sixth cluster is not close being a scalar multiple of an all--ones vector. In particular, that subvector, normalised so that its entries sum to $1$, has two large entries ($0.2079 $ and $0.1671$) 
 $15$ entries with values ranging between $
 0.0147$ and $0.0724$, and $37$ entries with values ranging between $8790.0015$ and $0.0032$.   We also observe that the diagonal entries of $W_{\1}(T)$ are not as well--correlated with the spectral radii of the corresponding principal submatrices of $T$ as the diagonal entries of $W_u(T)$ are. 
 
 Figure \ref{fig:airport} illustrates the sequence of iterates that is produced by Algorithm \ref{alg}.

 \begin{figure}
 \centering
 	\includegraphics[scale=.33]{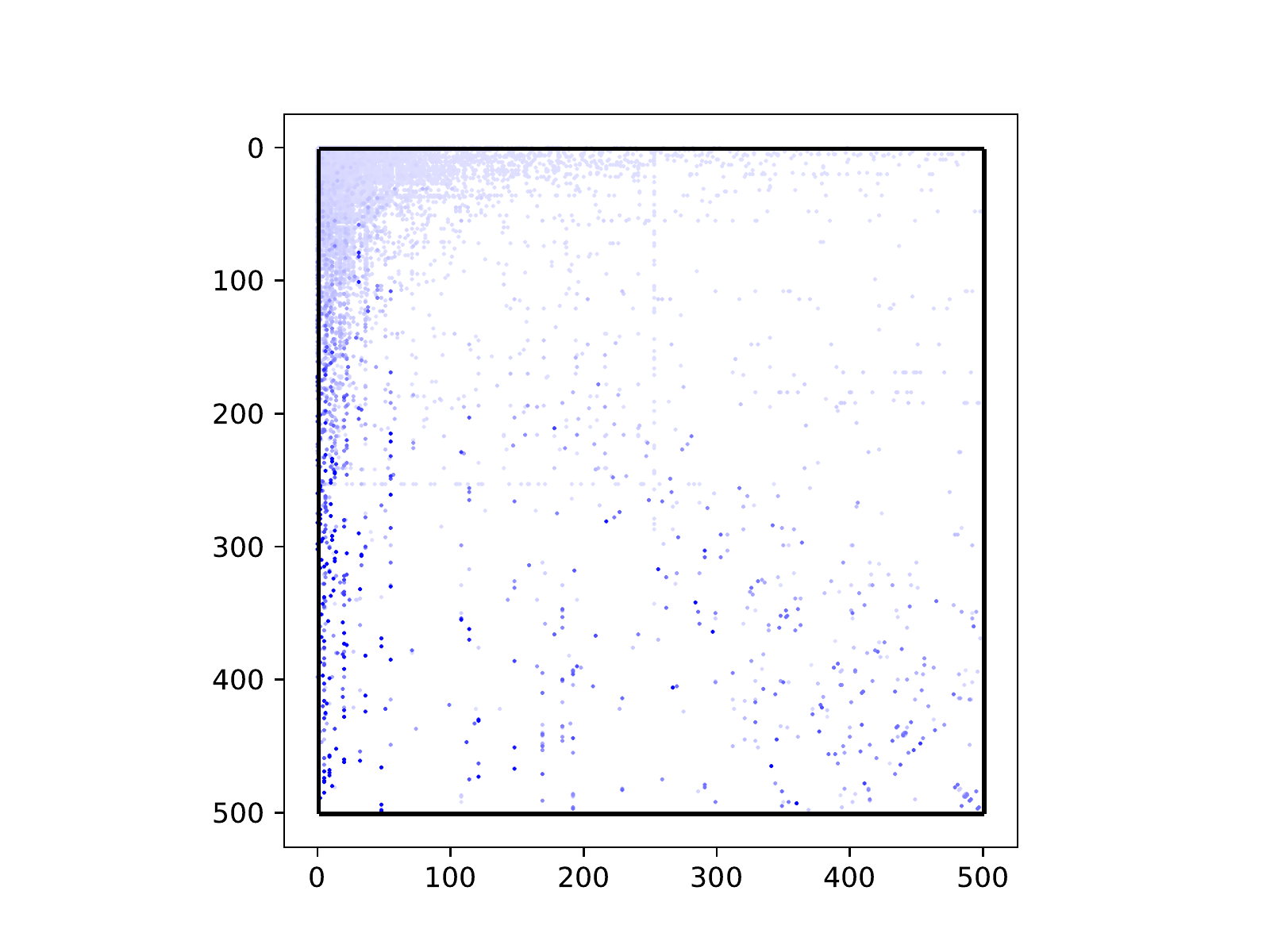}
 	\includegraphics[scale=.33]{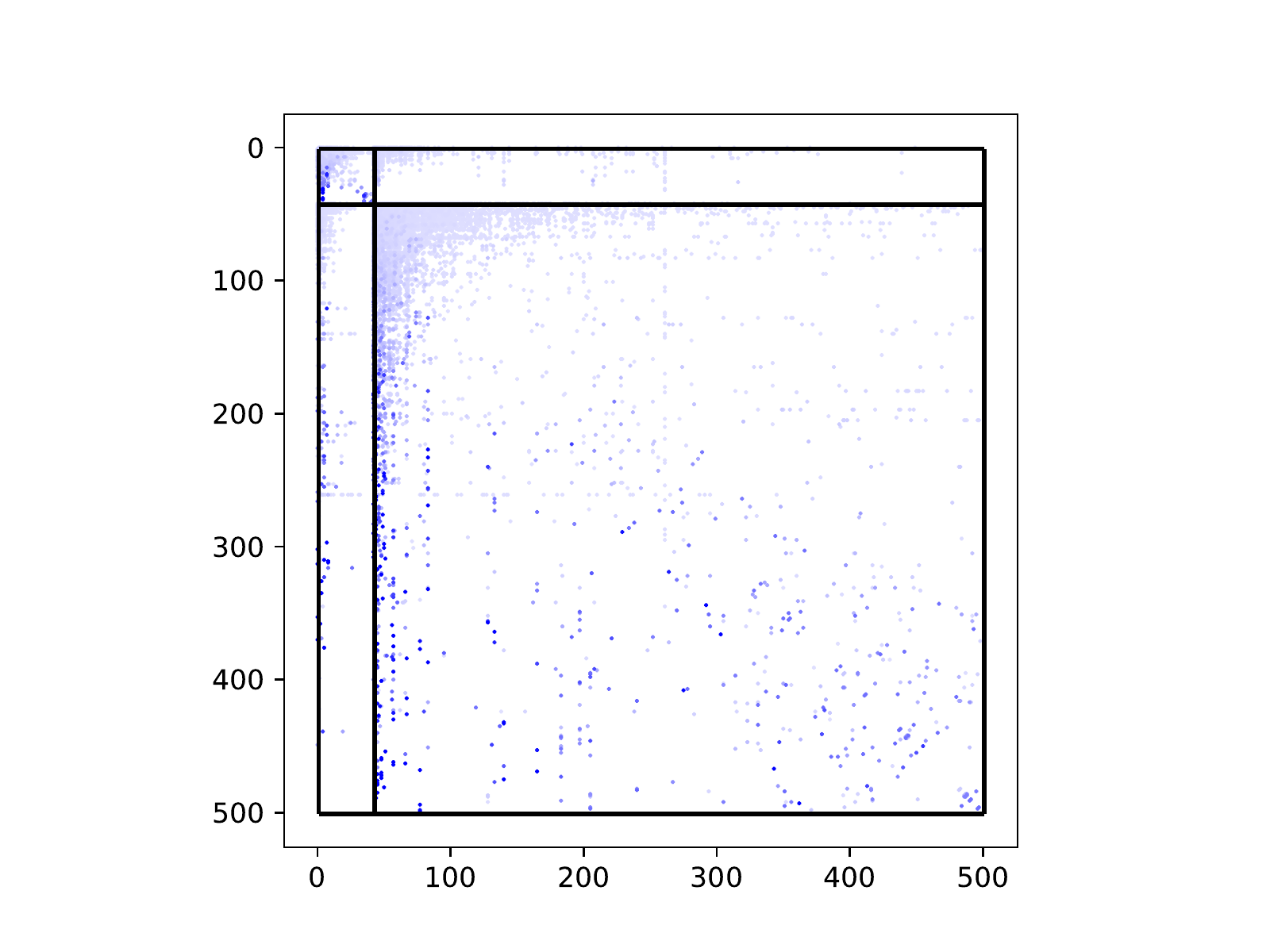}
 	\includegraphics[scale=.33]{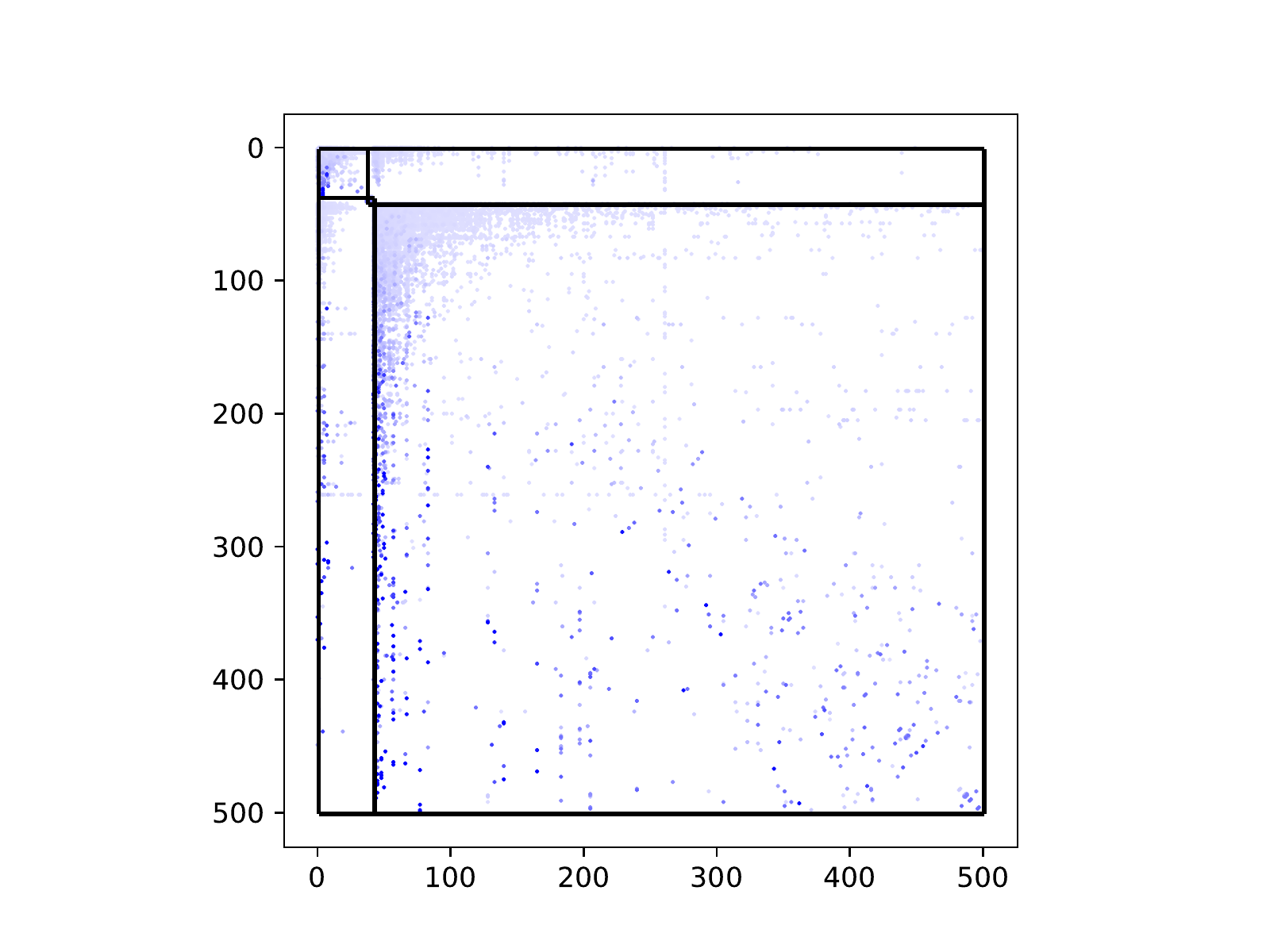}
 		\includegraphics[scale=.33]{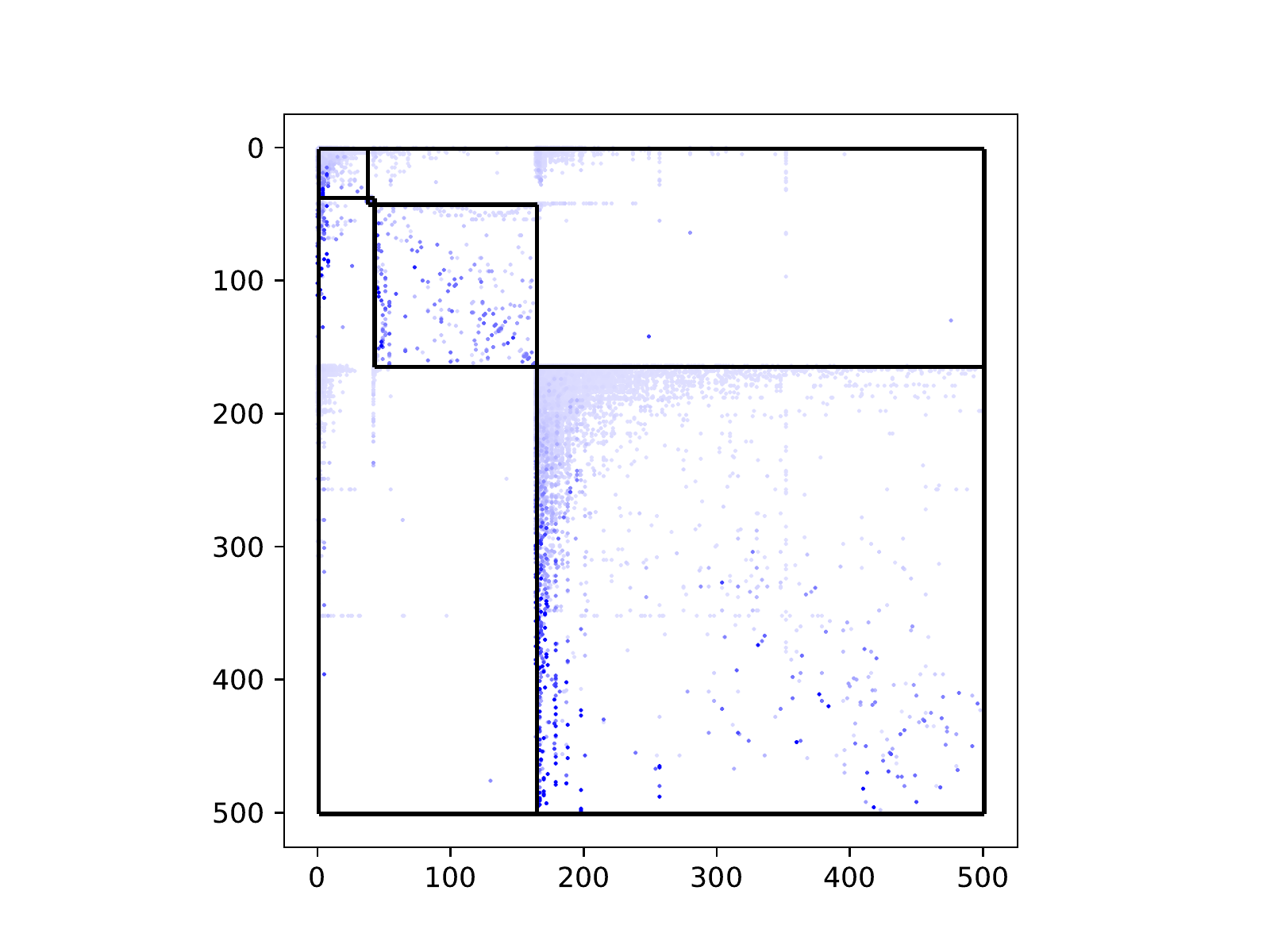}
 			\includegraphics[scale=.33]{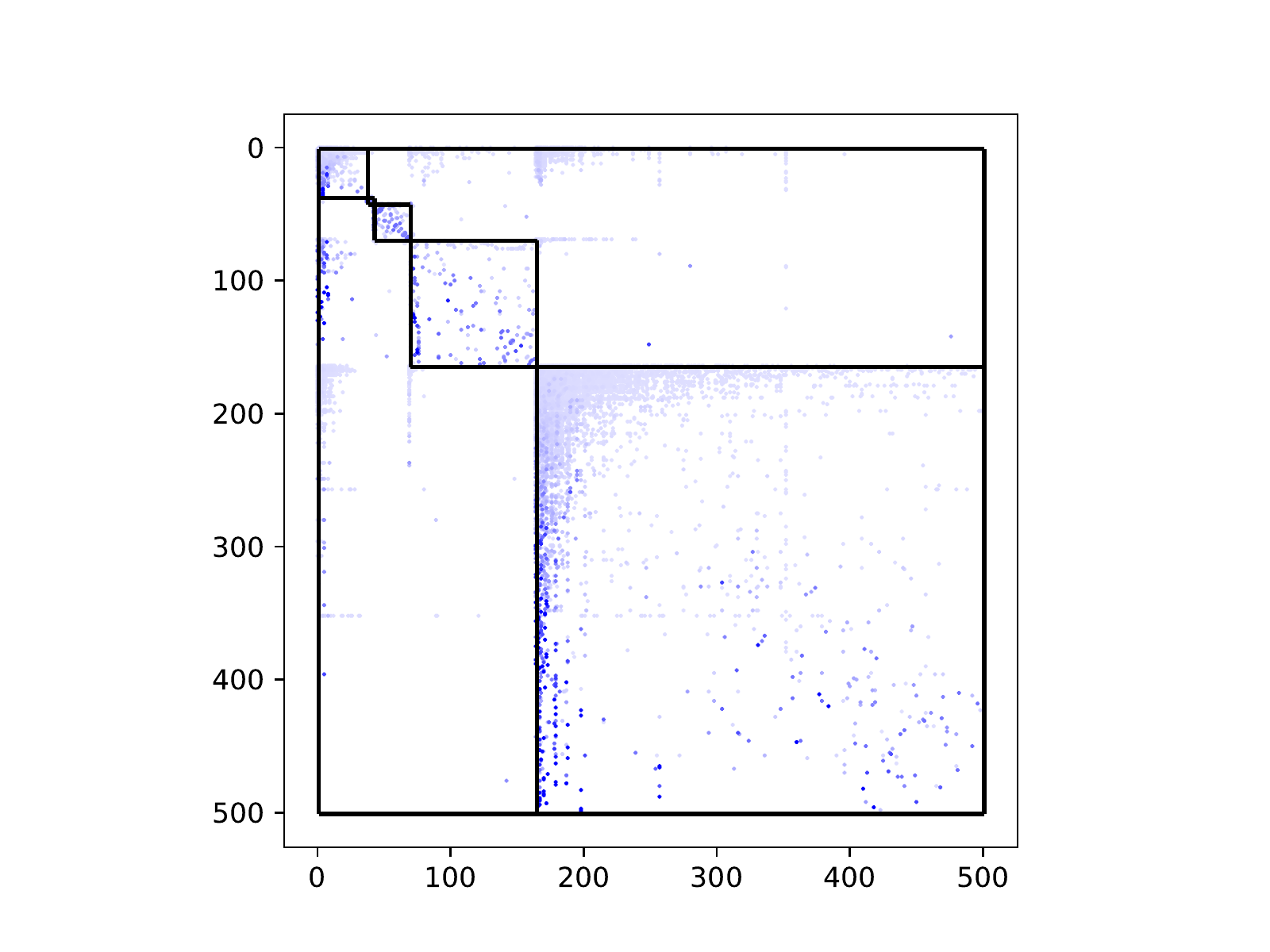}
 				\includegraphics[scale=.33]{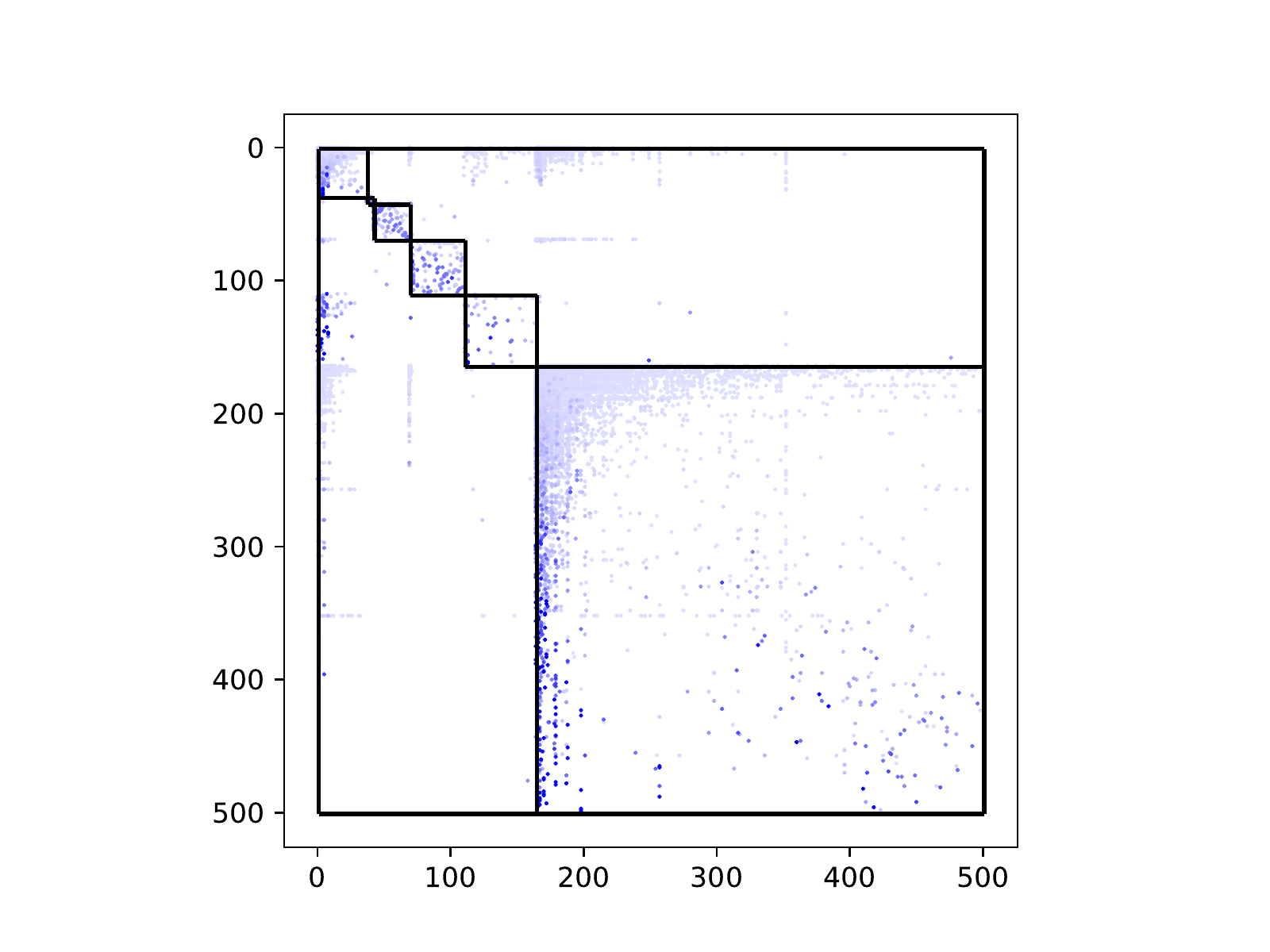}
 					\includegraphics[scale=.33]{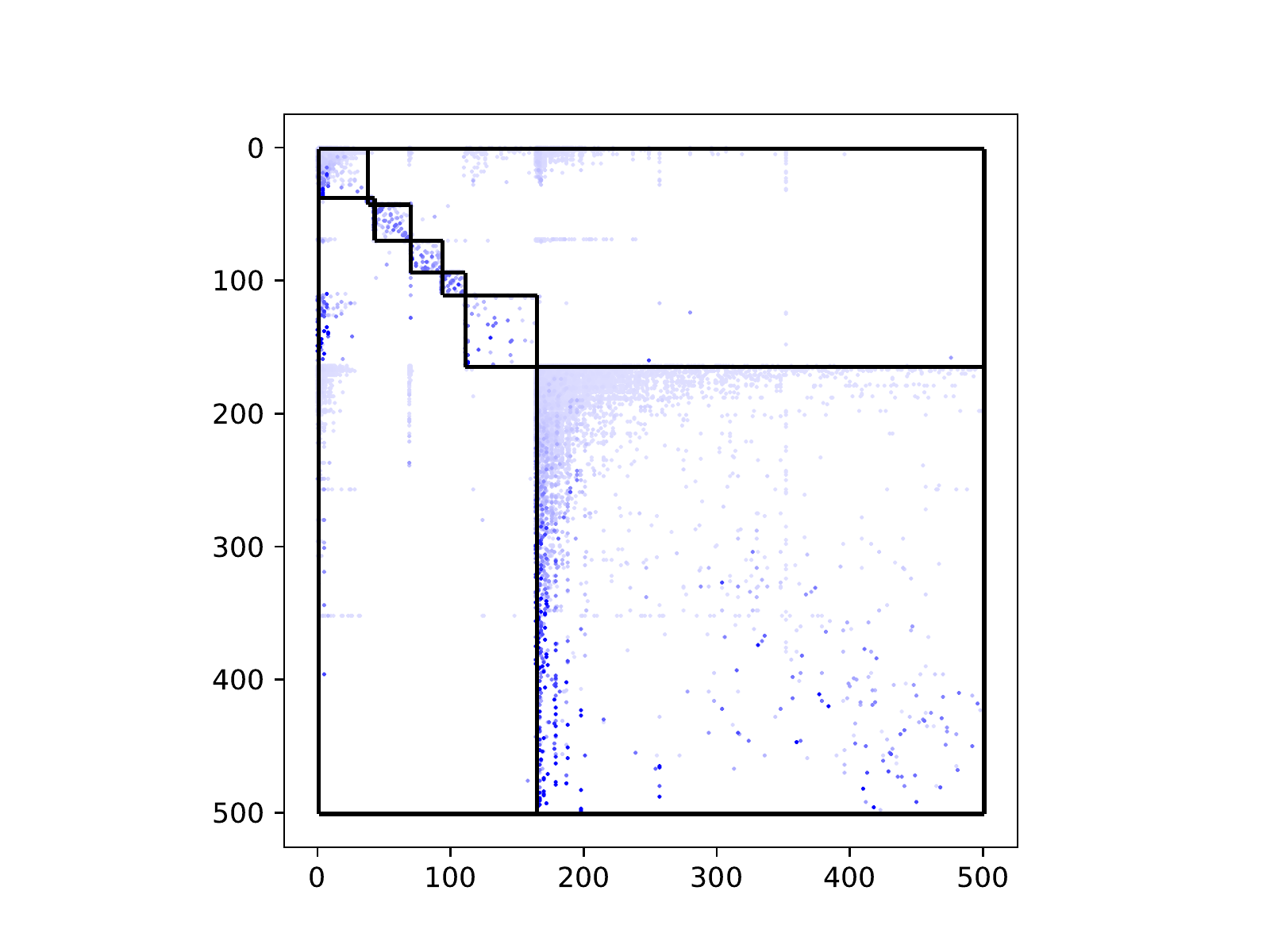}
 	\caption{The results of Algorithm~\ref{alg} applied to the stochastic matrix arising from the network of airports in Section \ref{air}. Darker points correspond   to higher transition probabilities.}\label{fig:airport}. 
 \end{figure}

\subsection{Two--mode networks}\label{sec:twomode}

In this section we observe our algorithm's performance on data sets arising from \emph{two-mode} networks.  A two-mode network can be modeled by a \emph{bipartite graph} $G$---that is, the nodes (vertices) come from two disjoint partite sets (perhaps representing two distinct types of objects, e.g.\ authors and publications), and only pairs of nodes from opposite sets are adjacent.  A $m \times n$ $(0,1)$ matrix $D$ can be used to represent such a network as follows.  Rows represent vertices of the first type (e.g.\ authors), and columns represent vertices of the second type (e.g.\ publications); there is a $1$ in the $(i, j)$ position if the two vertices are adjacent (e.g.\ if author $i$ was an author of publication $j$).  We can then construct a symmetric $(m + n) \times (m + n)$ matrix $A = \left[\begin{array}{c|c} 0&D \\ \hline D^{\top} &0 \end{array}\right],$ then normalise $A$ by dividing each row by its corresponding row sum, to generate a stochastic matrix $T$ which can be used as input to Algorithm~\ref{alg}.  Note that $T$ is the transition matrix of the Markov chain corresponding to a \emph{random walk} on $G$.

Our algorithm appears to perform quite well on Markov chains arising from two-mode networks.  Moreover, it has the potential to be a powerful tool in such applications, as it simultaneously identifies clusters in both partite sets based solely on the connections between them.  One potential real-world application is the problem of \emph{market segmentation} (see, e.g.\ \cite{wedel2012market}), in which one aims to find clusters in a two-mode network of customers and products they purchased.  Once clusters have been identified, an advertiser can suggest products to a customer based on the purchases of other customers in the same cluster.

\ex{\label{ex:deepsouth}

As a warmup, consider a two-mode social network consisting of individuals and activities.  We can represent it with a $(0, 1)$ matrix in which rows represent individuals and columns represent the activities in which they are involved (e.g.\ the events that they attend or the organisations to which they belong), with a $1$ in the $(i,j)$ position if individual $i$ is involved with activity $j$, and a $0$ there otherwise. A classic example with $18$ individuals and $14$ events \cite{davis} yields the following $(0,1)$ matrix: 
$$D=\left[\begin{array}{cccccccccccccc}
     1 &    1  &    1&      1 &     1 &     1   &   0 &     1  &    1 &    0&    0&     0 &   0 &   0 \\
1  &    1&      1     & 0     & 1     & 1     & 1     & 1     & 0     & 0     & 0     & 0   & 0   & 0\\
0&      1     & 1     & 1     & 1     & 1     & 1     & 1     & 1     & 0     & 0     & 0   & 0   & 0\\
1     & 0     & 1     & 1     & 1     & 1     & 1     & 1     & 0     & 0     & 0     & 0   & 0   & 0\\
0     & 0     & 1     & 1     & 1     & 0     & 1     & 0     & 0     & 0     & 0     & 0   & 0   & 0\\
0     & 0     & 1     & 0     & 1     & 1     & 0     & 1     & 0     & 0     & 0     & 0   & 0   & 0\\
0     & 0     & 0     & 0     & 1     & 1     & 1     & 1     & 0     & 0     & 0     & 0   & 0   & 0\\
0     & 0     & 0     & 0     & 0     & 1     & 0     & 1     & 1     & 0     & 0     & 0   & 0   & 0\\
0     & 0     & 0     & 0     & 1     & 0     & 1     & 1     & 1     & 0     & 0     & 0   & 0   & 0\\
0     & 0     & 0     & 0     & 0     & 0     & 1     & 1     & 1     & 0     & 0     & 1   & 0   & 0\\
0     & 0     & 0     & 0     & 0     & 0     & 0     & 1     & 1     & 1     & 0     & 1   & 0   & 0\\
0     & 0     & 0     & 0     & 0     & 0     & 0     & 1     & 1     & 1     & 0  &    1   & 1   & 1\\
0  &    0     & 0     & 0     & 0     & 0     & 1     & 1     & 1     & 1     & 0     & 1   & 1   & 1\\
0     & 0     & 0     & 0     & 0     & 1     & 1     & 0     & 1     & 1     & 1     & 1   & 1   & 1\\
0     & 0     & 0     & 0     & 0     & 0     & 1     & 1     & 0     & 1     & 1     & 1   & 0   & 0\\
0     & 0     & 0     & 0     & 0     & 0     & 0     & 1     & 1     & 0     & 0     & 0   & 0   & 0\\
0     & 0     & 0     & 0     & 0     & 0     & 0     & 0     & 1     & 0     & 1     & 0   & 0   & 0\\
0     & 0     & 0     & 0     & 0     & 0     & 0     & 0     & 1     & 0     & 1     & 0   & 0   & 0 
\end{array}\right].$$ 
In order to identify some clusters within this two--mode network, we  generate a $32\times 32$ stochastic matrix $T$, which is the transition matrix of the random walk on the bipartite graph associated with $D$.  


Next we run Algorithm \ref{alg} on $T$ with tolerance $\tau =0.2;$   
The smallest positive singular value of $I-T$ is approximately $0.1965,$  with the following left singular vector: 
$$x=\left[\begin{array}{c}   -0.2392\\
-0.2245\\
-0.1912\\
-0.2252\\
-0.1599\\
-0.1276\\
-0.0814\\
-0.0022\\
-0.0240\\
0.0716\\
0.1351\\
0.2469\\
0.2363\\
0.2744\\
0.1651\\
0.0324\\
0.1179\\
0.1179\\
-0.1382\\
-0.1266\\
-0.2363\\
-0.1698\\
-0.2582\\
-0.1647\\
-0.0439\\
-0.0052\\
0.2584\\
0.2190\\
0.2299\\
0.2344\\
0.1525\\
0.1525\\
\end{array}\right]. 
$$
This generates the following index sets: $S_1=\{1, \ldots, 9, 19, \ldots, 26\},$  $S_2 = \{10, \ldots, 18, 27, \ldots, 32\}.$ Since the smallest positive singular values of $I-\dnf(T[S_1])$ and $I-\dnf(T[S_2])$ are $0.6025$ and  $0.4187$ respectively, the algorithm terminates, and the left-iterative weight vector is $u=|x|.$ 
The corresponding coupling matrix is $W_{u}(T)=  
\begin{bmatrix}    0.9583&    0.0417\\
0.1570   & 0.8430
\end{bmatrix}.$ 

Referring back to the original matrix $D,$ the indices in $S_1$ and $S_2$ suggest that we should partition the rows of $D$ as $\{1, \ldots, 9,\} \cup \{10, \ldots, 18\},$ and the columns as $\{1, \ldots, 8\} \cup \{9, \ldots, 14\}.$ 
This yields the following partitioned matrix:
$$\left[\begin{array}{cccccccc|cccccc}
1 &    1  &    1&      1 &     1 &     1   &   0 &     1  &    1 &    0&    0&     0 &   0 &   0 \\
1  &    1&      1     & 0     & 1     & 1     & 1     & 1     & 0     & 0     & 0     & 0   & 0   & 0\\
0&      1     & 1     & 1     & 1     & 1     & 1     & 1     & 1     & 0     & 0     & 0   & 0   & 0\\
1     & 0     & 1     & 1     & 1     & 1     & 1     & 1     & 0     & 0     & 0     & 0   & 0   & 0\\
0     & 0     & 1     & 1     & 1     & 0     & 1     & 0     & 0     & 0     & 0     & 0   & 0   & 0\\
0     & 0     & 1     & 0     & 1     & 1     & 0     & 1     & 0     & 0     & 0     & 0   & 0   & 0\\
0     & 0     & 0     & 0     & 1     & 1     & 1     & 1     & 0     & 0     & 0     & 0   & 0   & 0\\
0     & 0     & 0     & 0     & 0     & 1     & 0     & 1     & 1     & 0     & 0     & 0   & 0   & 0\\
0     & 0     & 0     & 0     & 1     & 0     & 1     & 1     & 1     & 0     & 0     & 0   & 0   & 0\\ \hline
0     & 0     & 0     & 0     & 0     & 0     & 1     & 1     & 1     & 0     & 0     & 1   & 0   & 0\\
0     & 0     & 0     & 0     & 0     & 0     & 0     & 1     & 1     & 1     & 0     & 1   & 0   & 0\\
0     & 0     & 0     & 0     & 0     & 0     & 0     & 1     & 1     & 1     & 0  &    1   & 1   & 1\\
0  &    0     & 0     & 0     & 0     & 0     & 1     & 1     & 1     & 1     & 0     & 1   & 1   & 1\\
0     & 0     & 0     & 0     & 0     & 1     & 1     & 0     & 1     & 1     & 1     & 1   & 1   & 1\\
0     & 0     & 0     & 0     & 0     & 0     & 1     & 1     & 0     & 1     & 1     & 1   & 0   & 0\\
0     & 0     & 0     & 0     & 0     & 0     & 0     & 1     & 1     & 0     & 0     & 0   & 0   & 0\\
0     & 0     & 0     & 0     & 0     & 0     & 0     & 0     & 1     & 0     & 1     & 0   & 0   & 0\\
0     & 0     & 0     & 0     & 0     & 0     & 0     & 0     & 1     & 0     & 1     & 0   & 0   & 0 
\end{array}\right].$$ 


}

\ex{

In this next example, we analyze a two-mode network of authors and publications based on the bibliography of~\cite{imrich2000product}.  The dataset can be downloaded here: \url{http://vlado.fmf.uni-lj.si/pub/networks/data/2mode/Sandi/Sandi.htm}.  This network consists of 674 nodes---314 authors and 360 publications---with author $i$ adjacent to publication $j$ if and only if \ $i$ is an author of $j$.  The network contains one large connected component of size 253, and 128 connected components of size at most 14; thus, we henceforth restrict our attention to the single large connected component.

As described at the beginning of Section~\ref{sec:twomode}, we construct an $86 \times 167$ matrix $D$ whose rows correspond to the authors and columns to the publications in the restricted network and whose $(i, j)$ entry is 1 if author $i$ is an author of publication $j$, 0 otherwise.  We then construct the $253 \times 253$ symmetric matrix $A = \left[\begin{array}{c|c} 0&D \\ \hline D^{\top} &0 \end{array}\right]$ and normalise it to get a $253 \times 253$ stochastic matrix $T$ accordingly.  

Running Algorithm~\ref{alg} on $T$ with a tolerance of $\tau = .05$ yields 14 clusters, of sizes 8, 17, 6, 16, 15, 7, 12, 33, 13, 28, 7, 52, 27, and 12. The coupling matrices with respect to the left-iterative weight vector and the ones vector are, respectively,
\begin{center}
\scalebox{.625}{
$W_v = \left[
\begin{array}{cccccccccccccc}
0.9777 & 0.02232 & 0 & 0 & 0 & 0 & 0 & 0 & 0 & 0 & 0 & 0 & 0 & 0 \\
0.000897 & 0.955 & 0 & 0.02154 & 0.01077 & 0 & 0 & 0.00322 & 0 & 0 & 0 & 0.008609 & 0 & 0 \\
0 & 0 & 0.9839 & 0.01607 & 0 & 0 & 0 & 0 & 0 & 0 & 0 & 0 & 0 & 0 \\
0 & 0.03701 & 0.02774 & 0.9353 & 0 & 0 & 0 & 0 & 0 & 0 & 0 & 0 & 0 & 0 \\
0 & 0.006324 & 0 & 0 & 0.9937 & 0 & 0 & 0 & 0 & 0 & 0 & 0 & 0 & 0 \\
0 & 0 & 0 & 0 & 0 & 0.9731 & 0.02688 & 0 & 0 & 0 & 0 & 0 & 0 & 0 \\
0 & 0 & 0 & 0 & 0 & 0.01941 & 0.9612 & 0.01941 & 0 & 0 & 0 & 0 & 0 & 0 \\
0 & 0.0164 & 0 & 0 & 0 & 0 & 0.006544 & 0.9665 & 0 & 0.01052 & 0 & 0 & 0 & 0 \\
0 & 0 & 0 & 0 & 0 & 0 & 0 & 0 & 0.9929 & 0.007096 & 0 & 0 & 0 & 0 \\
0 & 0 & 0 & 0 & 0 & 0 & 0 & 0.01491 & 0.01587 & 0.8994 & 0.008809 & 0.06103 & 0 & 0 \\
0 & 0 & 0 & 0 & 0 & 0 & 0 & 0 & 0 & 0.02626 & 0.9737 & 0 & 0 & 0 \\
0 & 0.01182 & 0 & 0 & 0 & 0 & 0 & 0 & 0 & 0.006924 & 0 & 0.9763 & 0.004121 & 0.0008241 \\
0 & 0 & 0 & 0 & 0 & 0 & 0 & 0 & 0 & 0 & 0 & 0.01815 & 0.9819 & 0 \\
0 & 0 & 0 & 0 & 0 & 0 & 0 & 0 & 0 & 0 & 0 & 0.01423 & 0 & 0.9858
\end{array}
\right],$
}
\end{center}
\begin{center}
\scalebox{.625}{
$W_\1 = 
\left[
\begin{array}{cccccccccccccc}
0.9375 & 0.0625 & 0 & 0 & 0 & 0 & 0 & 0 & 0 & 0 & 0 & 0 & 0 & 0 \\
0.01961 & 0.9034 & 0 & 0.01471 & 0.007353 & 0 & 0 & 0.03529 & 0 & 0 & 0 & 0.01961 & 0 & 0 \\
0 & 0 & 0.9167 & 0.08333 & 0 & 0 & 0 & 0 & 0 & 0 & 0 & 0 & 0 & 0 \\
0 & 0.04167 & 0.00625 & 0.9521 & 0 & 0 & 0 & 0 & 0 & 0 & 0 & 0 & 0 & 0 \\
0 & 0.03333 & 0 & 0 & 0.9667 & 0 & 0 & 0 & 0 & 0 & 0 & 0 & 0 & 0 \\
0 & 0 & 0 & 0 & 0 & 0.9286 & 0.07143 & 0 & 0 & 0 & 0 & 0 & 0 & 0 \\
0 & 0 & 0 & 0 & 0 & 0.01042 & 0.9792 & 0.01042 & 0 & 0 & 0 & 0 & 0 & 0 \\
0 & 0.0404 & 0 & 0 & 0 & 0 & 0.0101 & 0.9192 & 0 & 0.0303 & 0 & 0 & 0 & 0 \\
0 & 0 & 0 & 0 & 0 & 0 & 0 & 0 & 0.9487 & 0.05128 & 0 & 0 & 0 & 0 \\
0 & 0 & 0 & 0 & 0 & 0 & 0 & 0.01923 & 0.02976 & 0.9295 & 0.001374 & 0.02015 & 0 & 0 \\
0 & 0 & 0 & 0 & 0 & 0 & 0 & 0 & 0 & 0.07143 & 0.9286 & 0 & 0 & 0 \\
0 & 0.001603 & 0 & 0 & 0 & 0 & 0 & 0 & 0 & 0.05536 & 0 & 0.9378 & 0.004371 & 0.0008741 \\
0 & 0 & 0 & 0 & 0 & 0 & 0 & 0 & 0 & 0 & 0 & 0.07407 & 0.9259 & 0 \\
0 & 0 & 0 & 0 & 0 & 0 & 0 & 0 & 0 & 0 & 0 & 0.04167 & 0 & 0.9583
\end{array}
\right].$
}
\end{center}
Separating the nodes in each cluster into those corresponding to authors and those corresponding to publications, we obtain 14 clusters of authors, of sizes 4, 7, 3, 4, 6, 2, 5, 11, 8, 5, 3, 14, 10, and 4, and 14 clusters of publications, of sizes 4, 10, 3, 12, 9, 5, 7, 22, 5, 23, 4, 38, 17, and 8.  Figure~\ref{fig:authors} shows the matrix $D$ after 0, 1, 2, and 14 iterations of Algorithm~\ref{alg}, with indices permuted so that clusters of authors and publications comprise contiguous blocks of rows and columns (respectively).

\begin{figure}
\centering
\includegraphics[scale=.38]{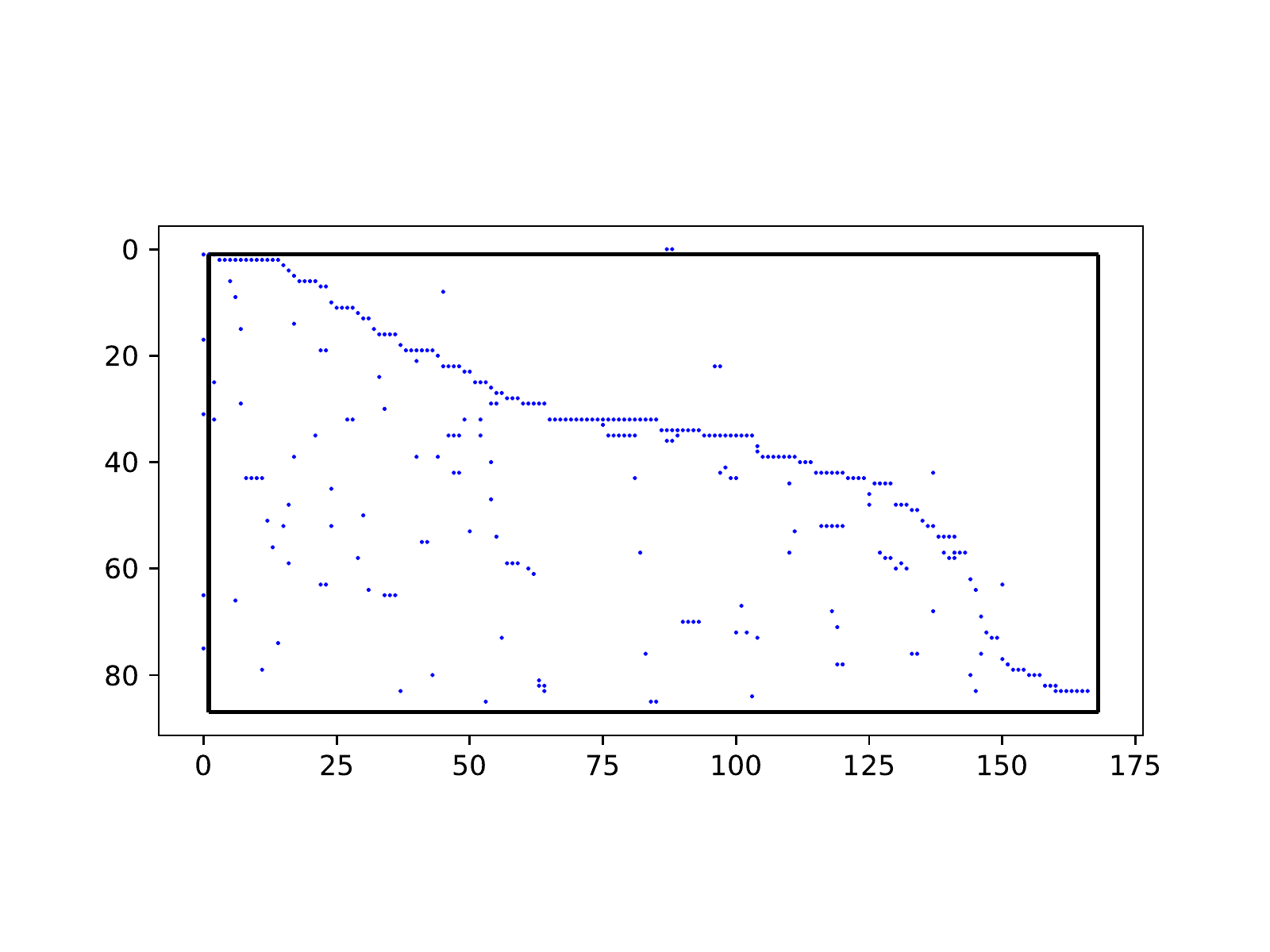}
\includegraphics[scale=.38]{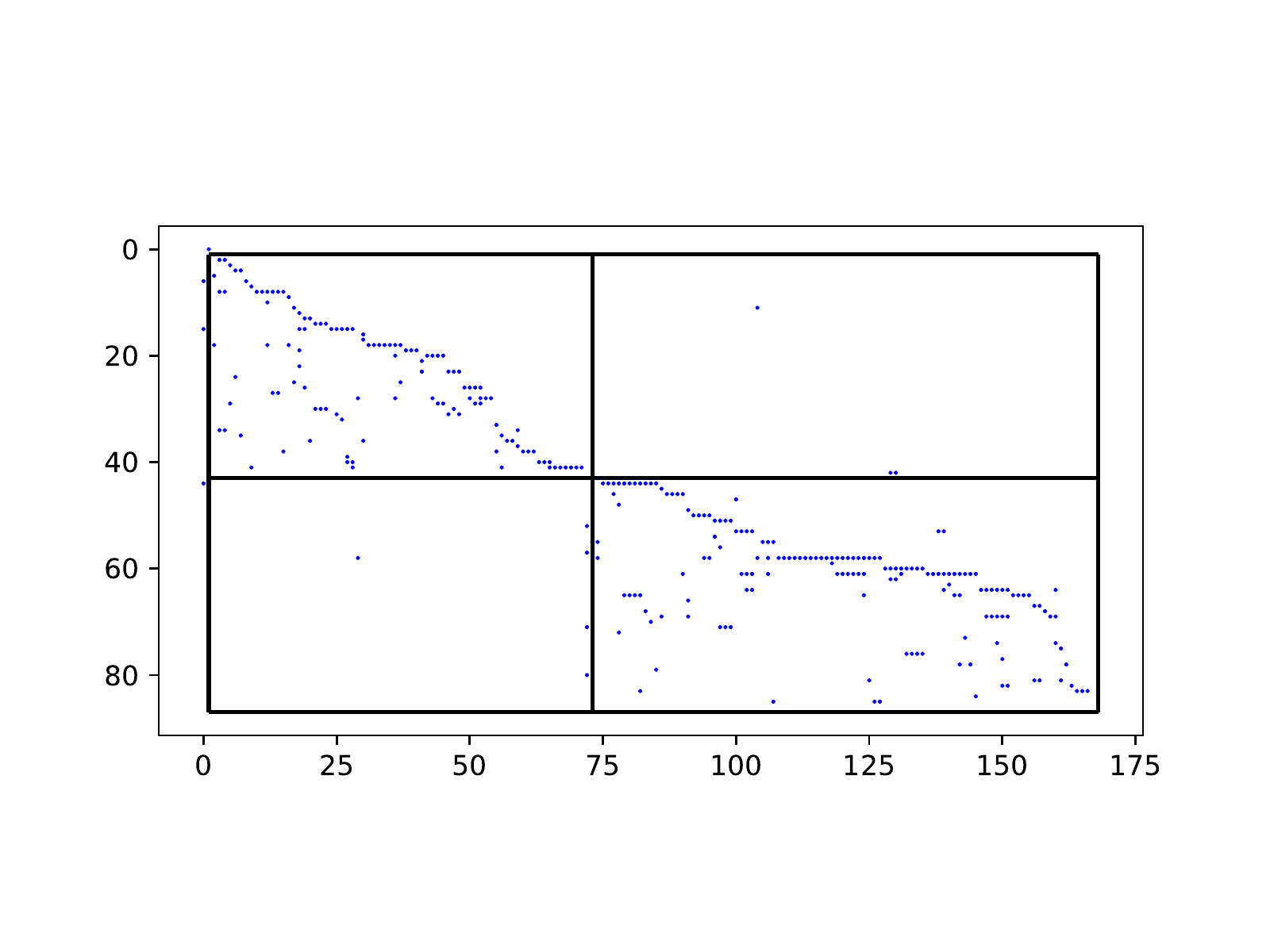}
\includegraphics[scale=.38]{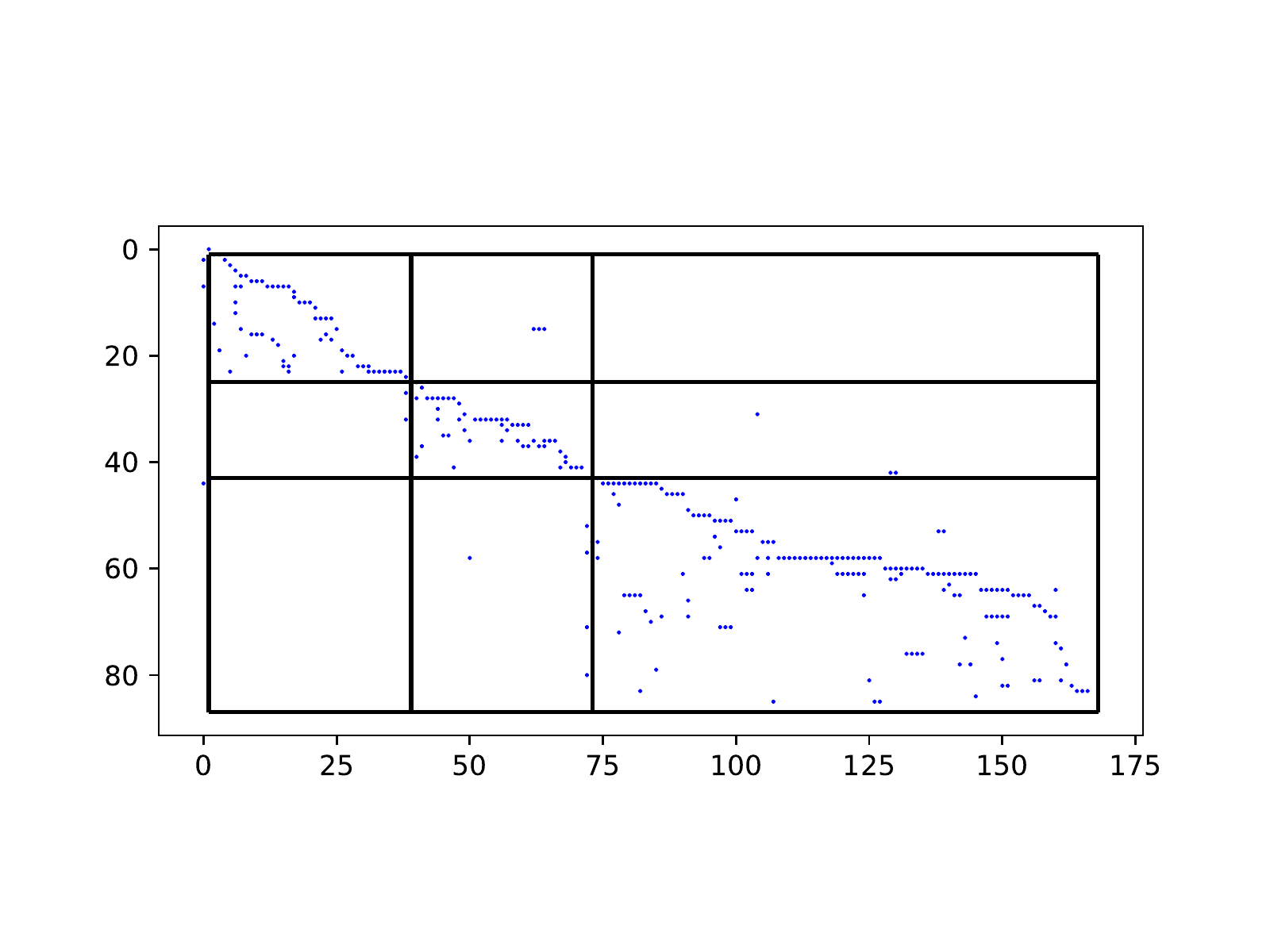}
\includegraphics[scale=.38]{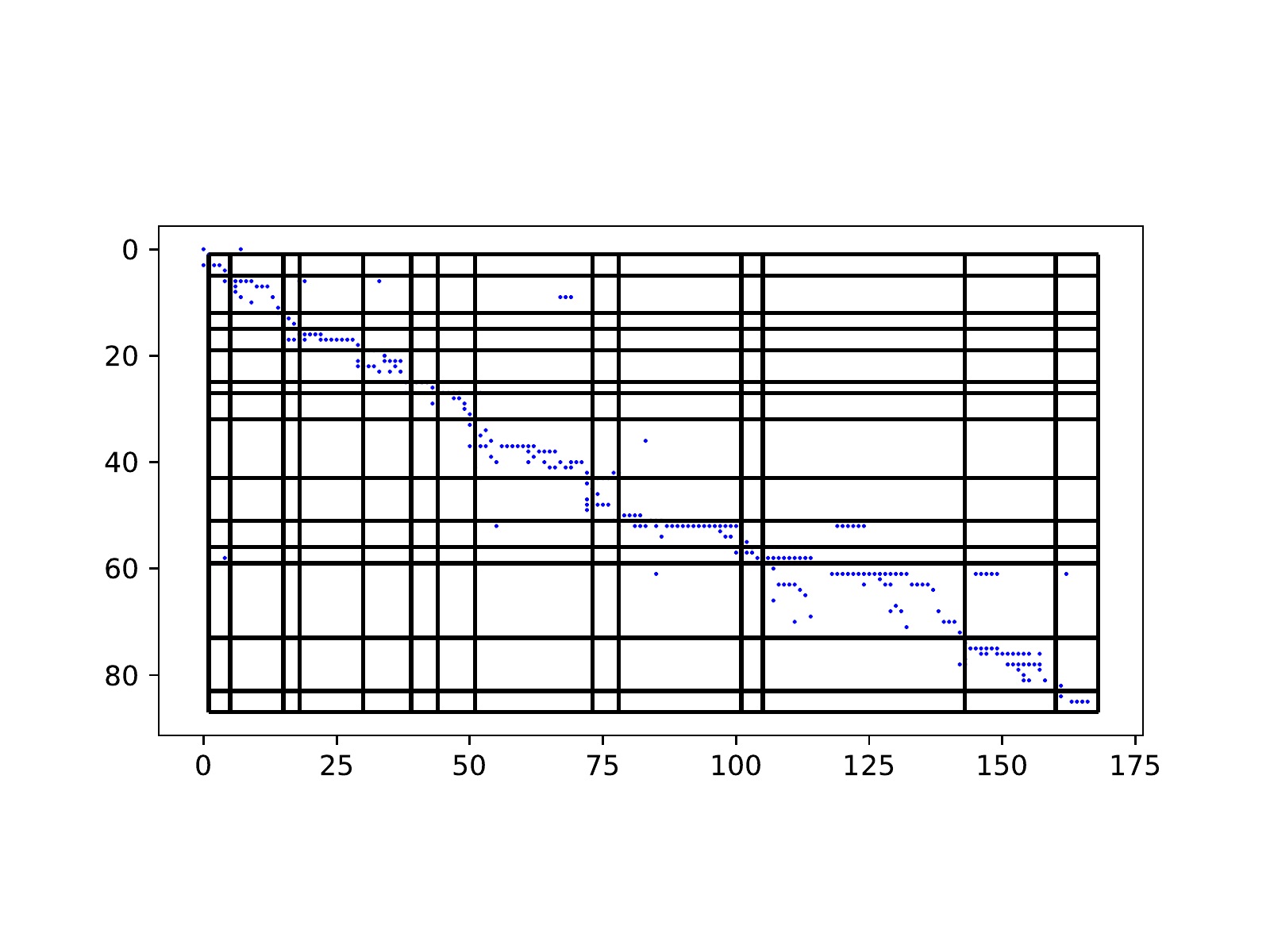}

\caption{The matrix $D \in \{0, 1\}^{86 \times 167}$ representing the authors--publications network after 0, 1, 2, and 14 iterations of Algorithm~\ref{alg}.}\label{fig:authors}
\end{figure}
}

\subsection{Metastable conformations of molecules}

In this section we apply Algorithm~\ref{alg} to the task of identifying metastable conformations of molecules.  We used the n-pentane example presented in~\cite{DHFS,fritzsche2008preprint,fritzsche2008svd}.  We ran our algorithm on the transition matrices generated from n-pentane at two different energy levels (Ph300 and Ph500).  We use the average diagonal entry of the coupling matrix to evaluate the strength of a clustering: a higher average diagonal entry indicates that the coupling matrix is ``more'' diagonally dominant, hence the Markov chain is more strongly clustered.

Our results were as follows.
\begin{enumerate}
\item	Ph300, $\tau = .1$: Our results were comparable to those obtained in~\cite[Section~6.2]{fritzsche2008preprint}.  Our algorithm identified 7 clusters, of sizes 36. 30, 60, 23, 33, 22, and 51 (see Figure~\ref{fig:Ph300}), while in~\cite{fritzsche2008preprint} they obtained 7 clusters, of sizes 46, 24, 36, 20, 42, 47, and 40.  Using the left-iterative weight vector $v$, we get the coupling matrix
\[W_v = 
\begin{bmatrix}
0.9420    & 0.0289    & 0.0003    & 0.0000    & 0.0015    & 0.0197    & 0.0076 \\
0.0310    & 0.9258    & 0.0420    & 0.0000    & 0.0003    & 0.0007    & 0.0002 \\
0.0001    & 0.0030    & 0.9821    & 0.0000    & 0.0017    & 0.0101    & 0.0030 \\
0.0000    & 0.0000    & 0.0000    & 0.9308    & 0.0279    & 0.0002    & 0.0411 \\
0.0013    & 0.0034    & 0.0044    & 0.0145    & 0.9555    & 0.0206    & 0.0003 \\
0.0039    & 0.0044    & 0.0096    & 0.0000    & 0.0042    & 0.9705    & 0.0073 \\
0.0010    & 0.0001    & 0.0027    & 0.0029    & 0.0001    & 0.0115    & 0.9817
\end{bmatrix},
\]
while using the ones vector we get
\[W_\1 = 
\begin{bmatrix}
0.7857	& 0.0869    & 0.0015    & 0.0001    & 0.0067    & 0.0799    & 0.0392 \\
0.2015  & 0.6107    & 0.0792    & 0.0001    & 0.0271    & 0.0739    & 0.0076 \\
0.0045  & 0.0426    & 0.7710    & 0.0012    & 0.0546    & 0.0975    & 0.0286 \\
0.0031  & 0.0009    & 0.0062    & 0.7149    & 0.2175    & 0.0038    & 0.0536 \\
0.0056  & 0.0128    & 0.0216    & 0.0535    & 0.8462    & 0.0593    & 0.0009 \\
0.0274  & 0.0311    & 0.0363    & 0.0001    & 0.0734    & 0.7986    & 0.0331 \\
0.0290  & 0.0004    & 0.0306    & 0.0494    & 0.0020    & 0.0937    & 0.7948
\end{bmatrix}.\]
These coupling matrices have average diagonal entries .9555 and .7603, respectively, while $W_\pi$ and $W_1$ computed in~\cite{fritzsche2008preprint} have average diagonal entries .9524 and .7792, respectively (where $W_\pi$ is the coupling matrix with respect to the stationary distribution of the transition matrix).

\begin{figure}
\centering
\includegraphics[scale=.33]{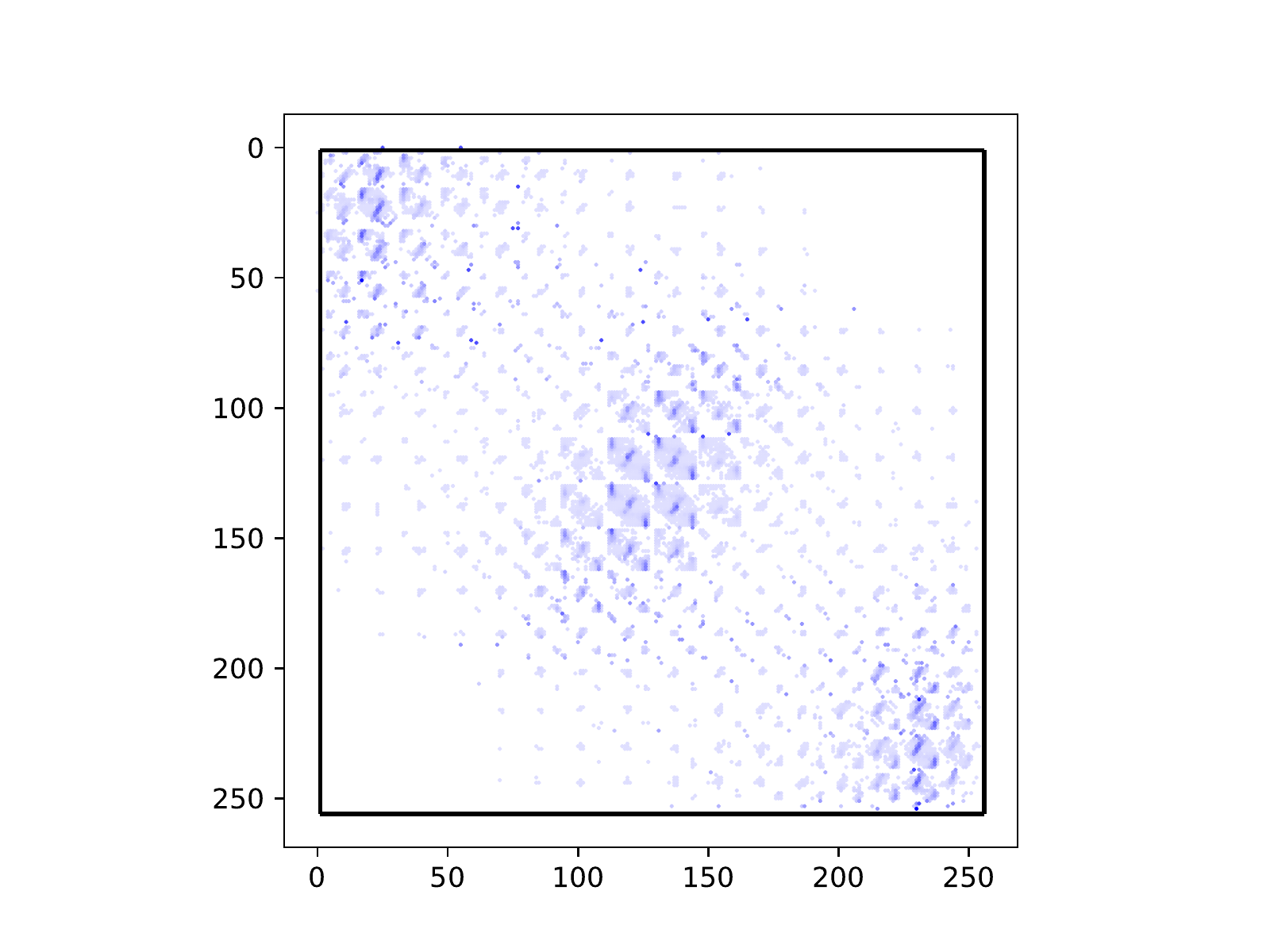}
\includegraphics[scale=.33]{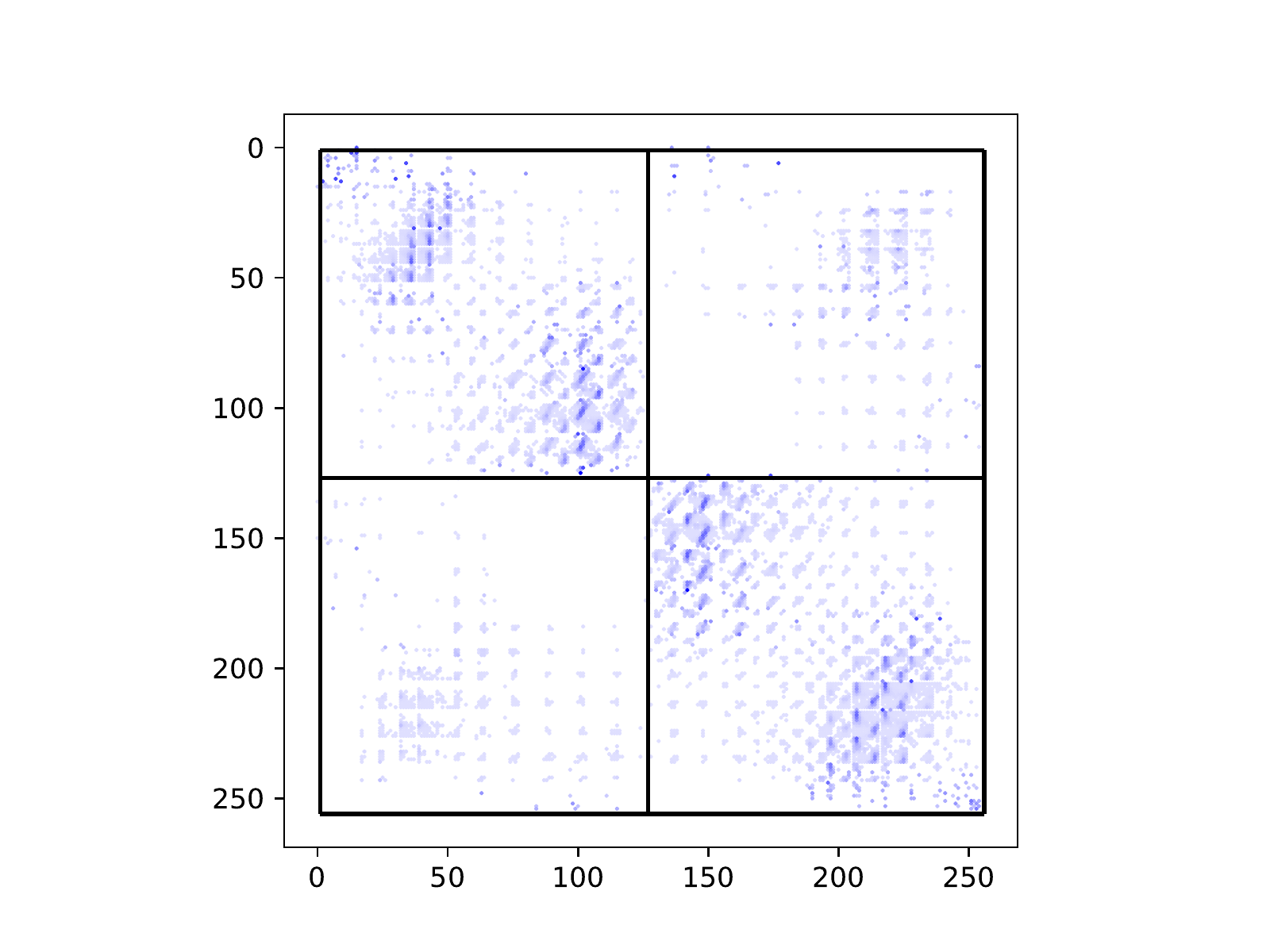}
\includegraphics[scale=.33]{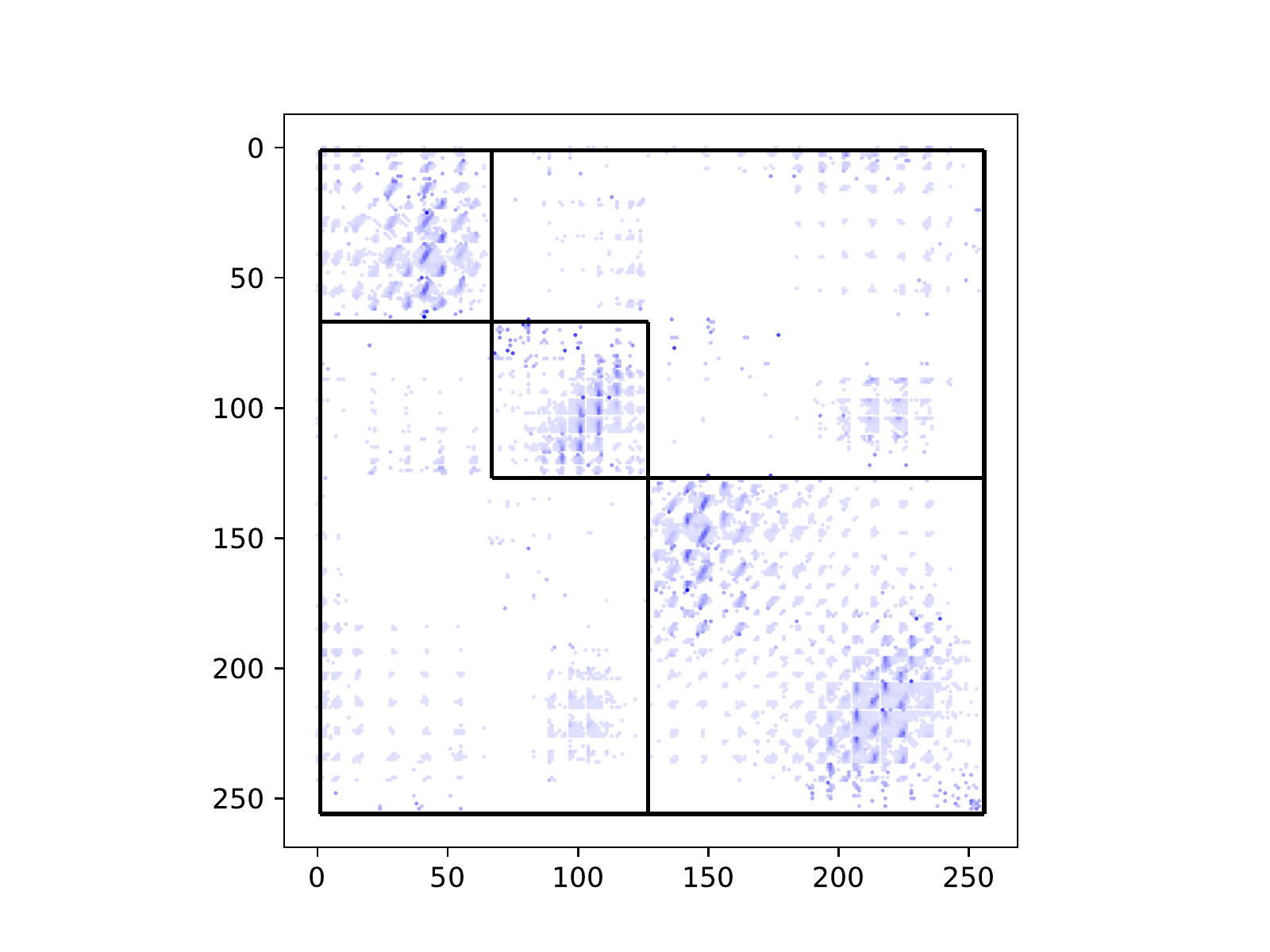}
\includegraphics[scale=.33]{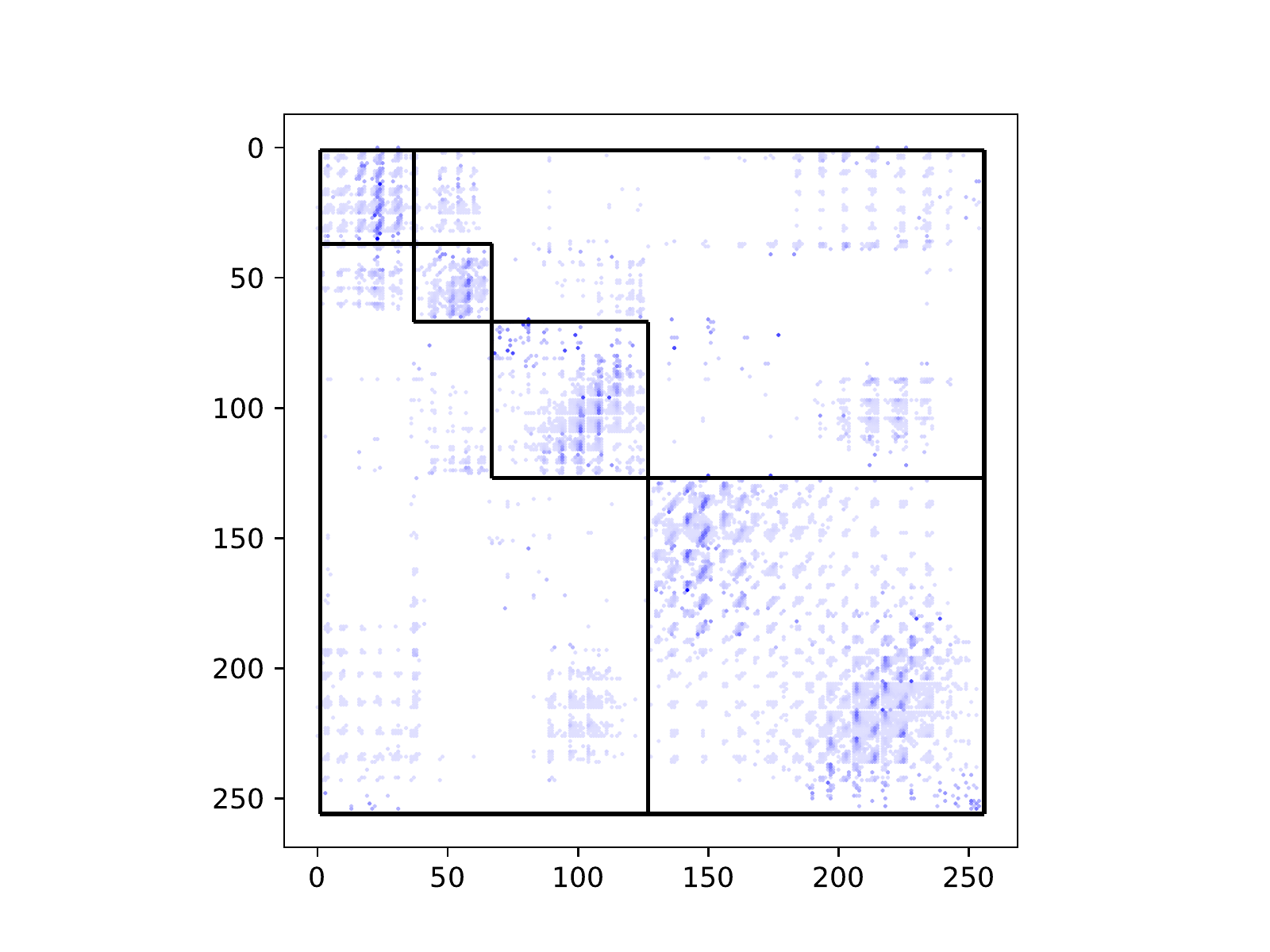}
\includegraphics[scale=.33]{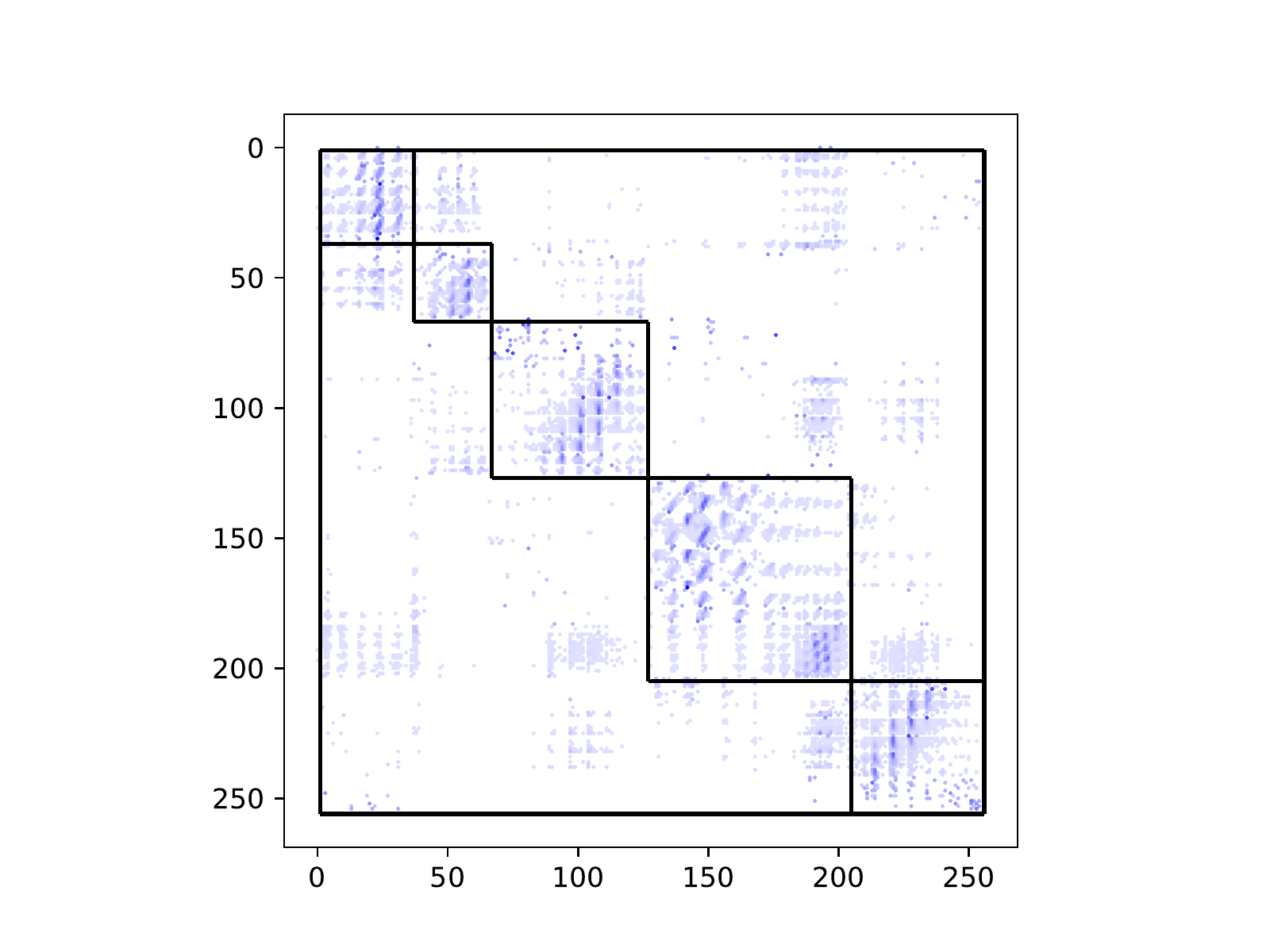}
\includegraphics[scale=.33]{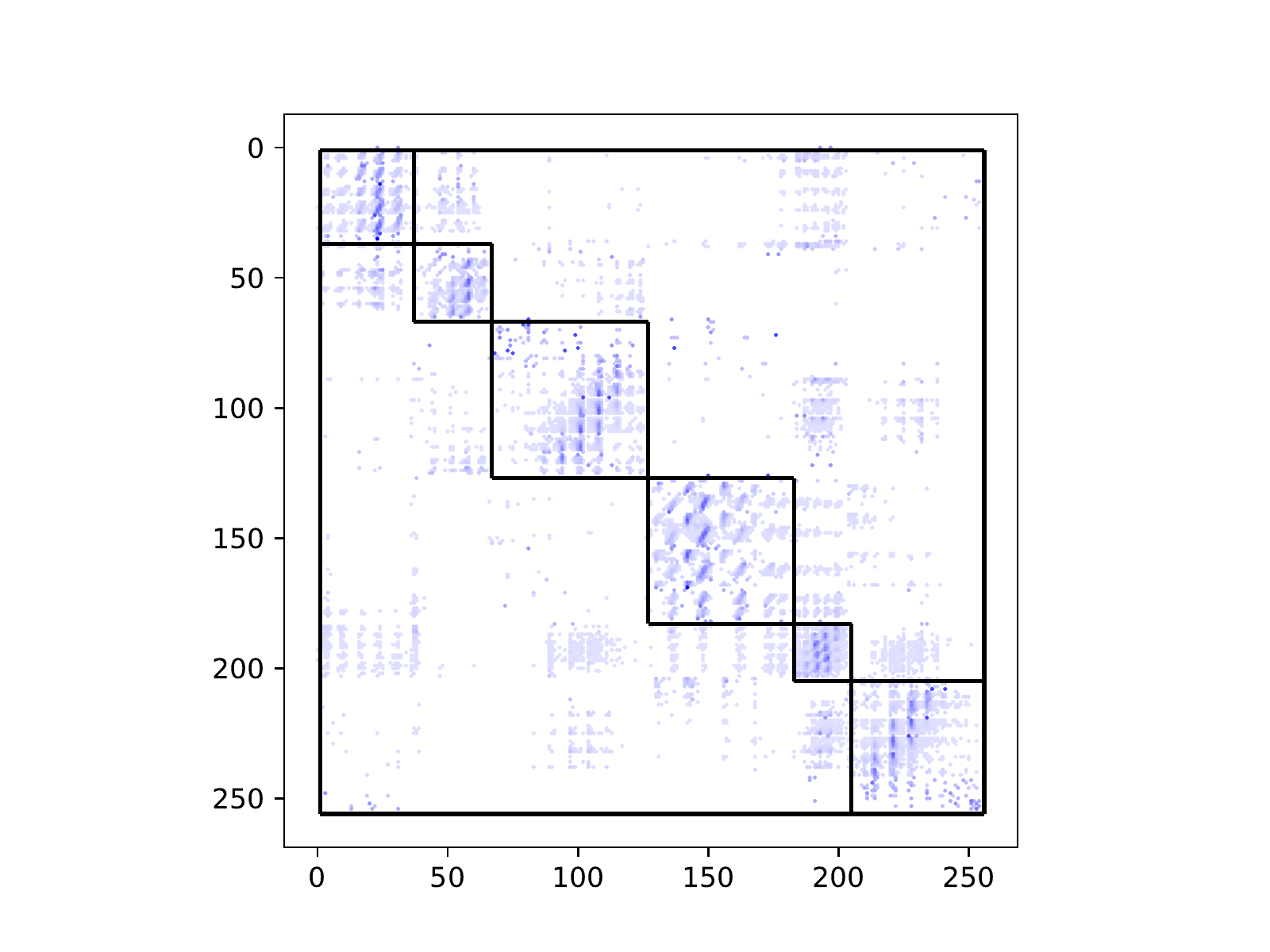}
\includegraphics[scale=.33]{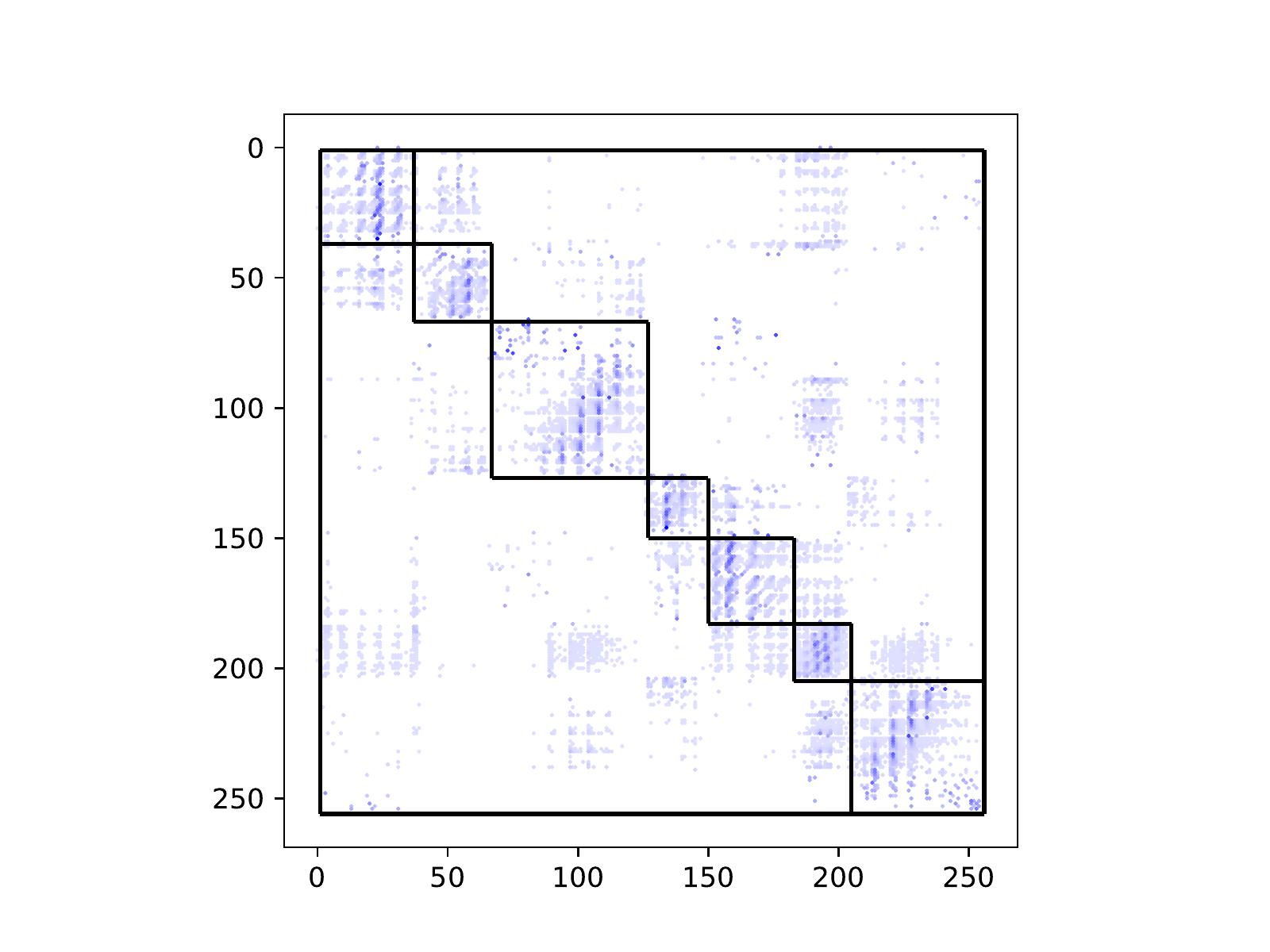}
\caption{Ph300, $\tau = .1$.  Algorithm~\ref{alg} produces 7 clusters, of sizes 36. 30, 60, 23, 33, 22, and 51.  Each image represents the transition matrix of the Markov chain after an iteration of the algorithm.  In each iteration, the states in one of the existing clusters are partitioned into two new clusters, then permuted so that the clusters comprise contiguous blocks of indices.}\label{fig:Ph300}
\end{figure}

\item	Ph500, $\tau = .2$: Our results compare favorably to those obtained in~\cite{fritzsche2008preprint}, using $\1$ in their algorithm's computation.  Our algorithm identified 5 clusters, of sizes 63, 33, 54, 76, and 81 (see Figure~\ref{fig:Ph500}), yielding
\[W_v = 
\begin{bmatrix}
	0.9012    & 0.0212    & 0.0160    & 0.0233    & 0.0383 \\
    0.1218    & 0.7921    & 0.0790    & 0.0045    & 0.0025 \\
    0.0120    & 0.0353    & 0.8708    & 0.0316    & 0.0502 \\
    0.0224    & 0.0017    & 0.0298    & 0.9011    & 0.0451 \\
    0.0268    & 0.0023    & 0.0284    & 0.0182    & 0.9243
\end{bmatrix},\] 
\[W_\1 = 
\begin{bmatrix}
    0.6841    & 0.0862    & 0.0428    & 0.1063    & 0.0806 \\
    0.1378    & 0.6045    & 0.1973    & 0.0311    & 0.0292 \\
    0.0318    & 0.0607    & 0.6789    & 0.0701    & 0.1585 \\
    0.0957    & 0.0042    & 0.0537    & 0.6763    & 0.1701 \\
    0.0477    & 0.0046    & 0.1025    & 0.0902    & 0.7550
\end{bmatrix},
\]
which have average diagonal entries .8779 and .6798, respectively.  In~\cite{fritzsche2008preprint} they produced 5 clusters, of sizes 37, 88, 71, 51, and 60, yielding coupling matrices $W_\pi$ and $W_\1$ with average diagonal entries .7428 and .6239, respectively. Thus, our algorithm significantly outperformed that of~\cite{fritzsche2008preprint} on this example.

\begin{figure}
\centering
\includegraphics[scale=.33]{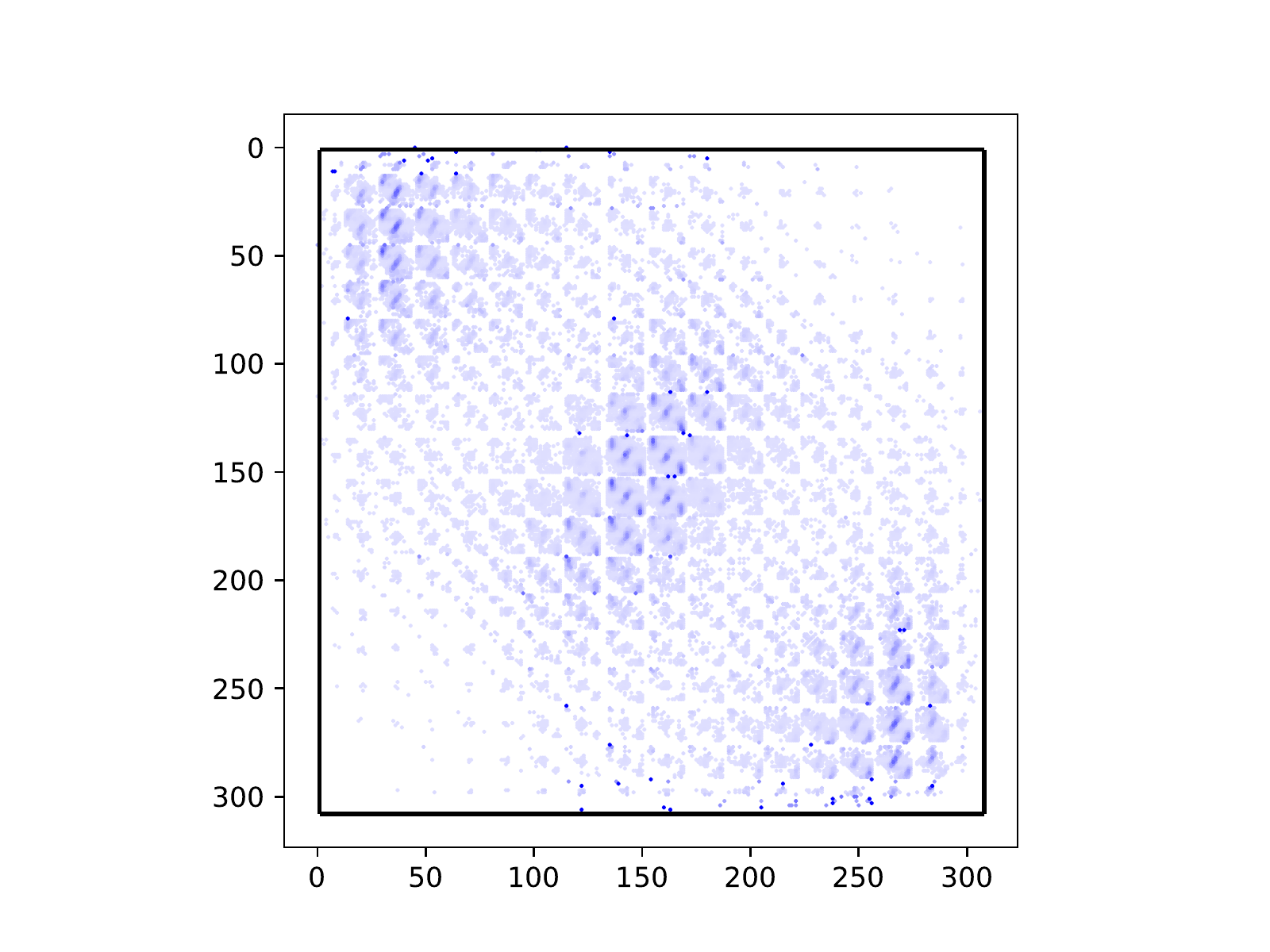}
\includegraphics[scale=.33]{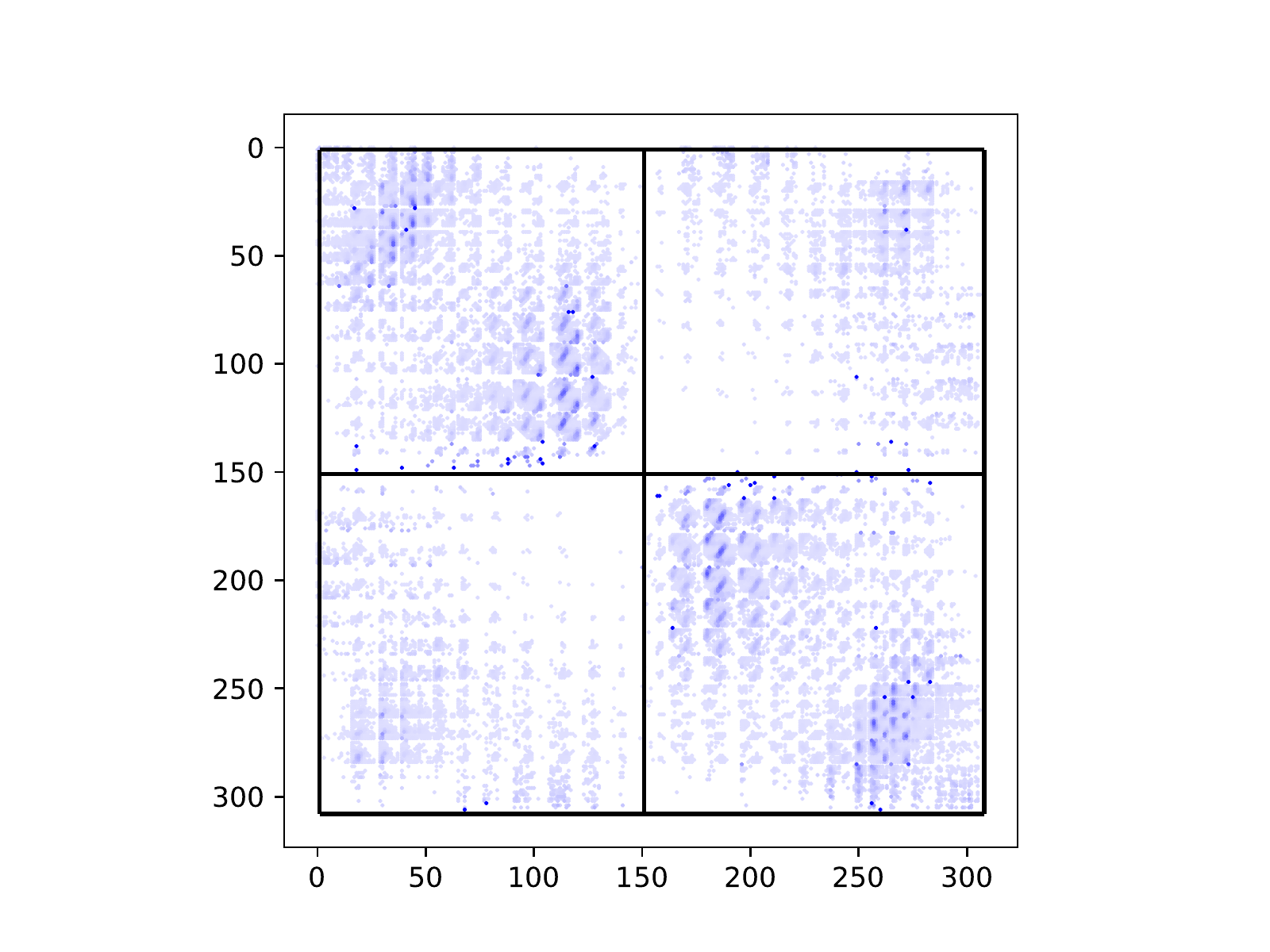}
\includegraphics[scale=.33]{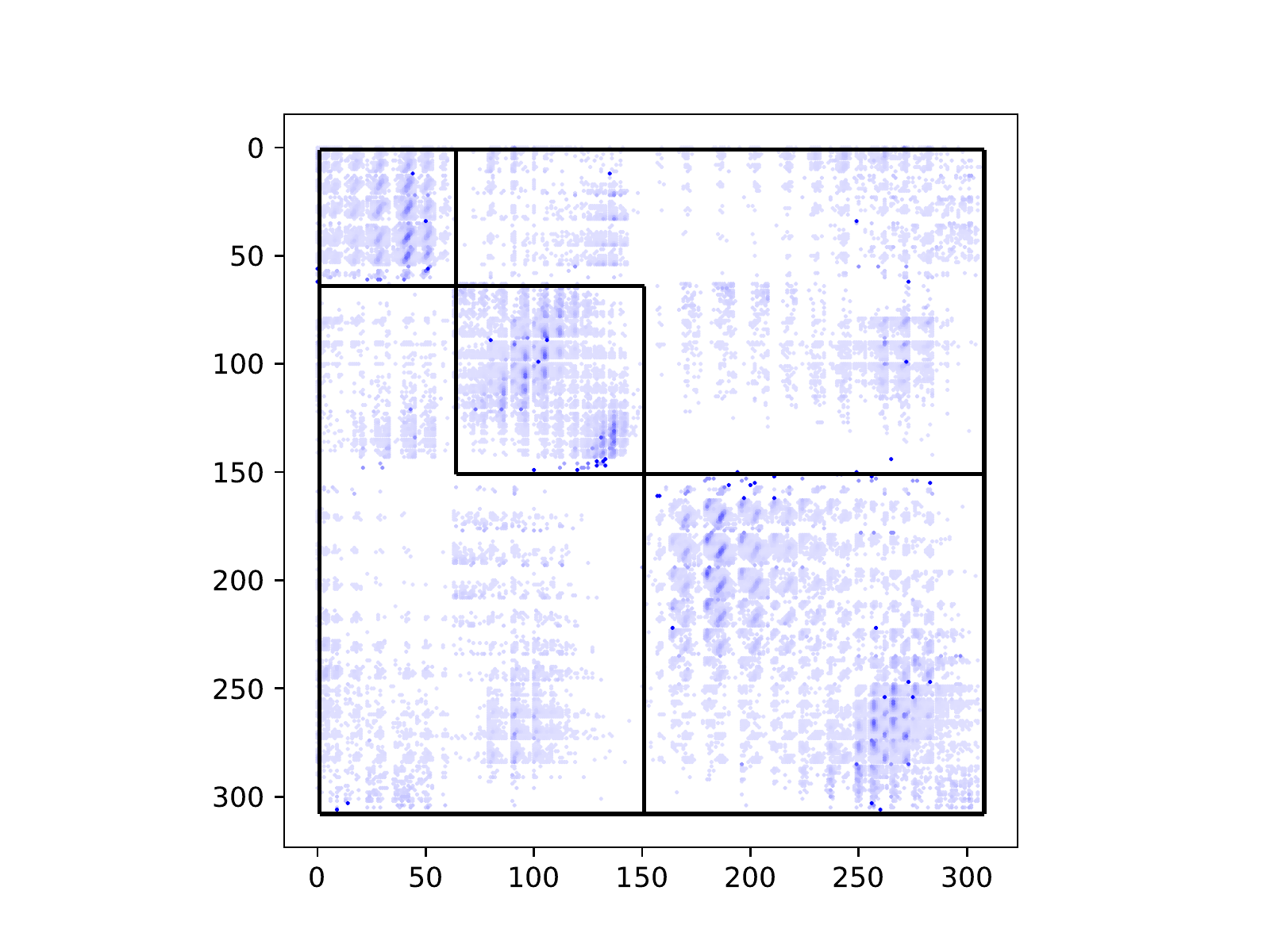}
\includegraphics[scale=.33]{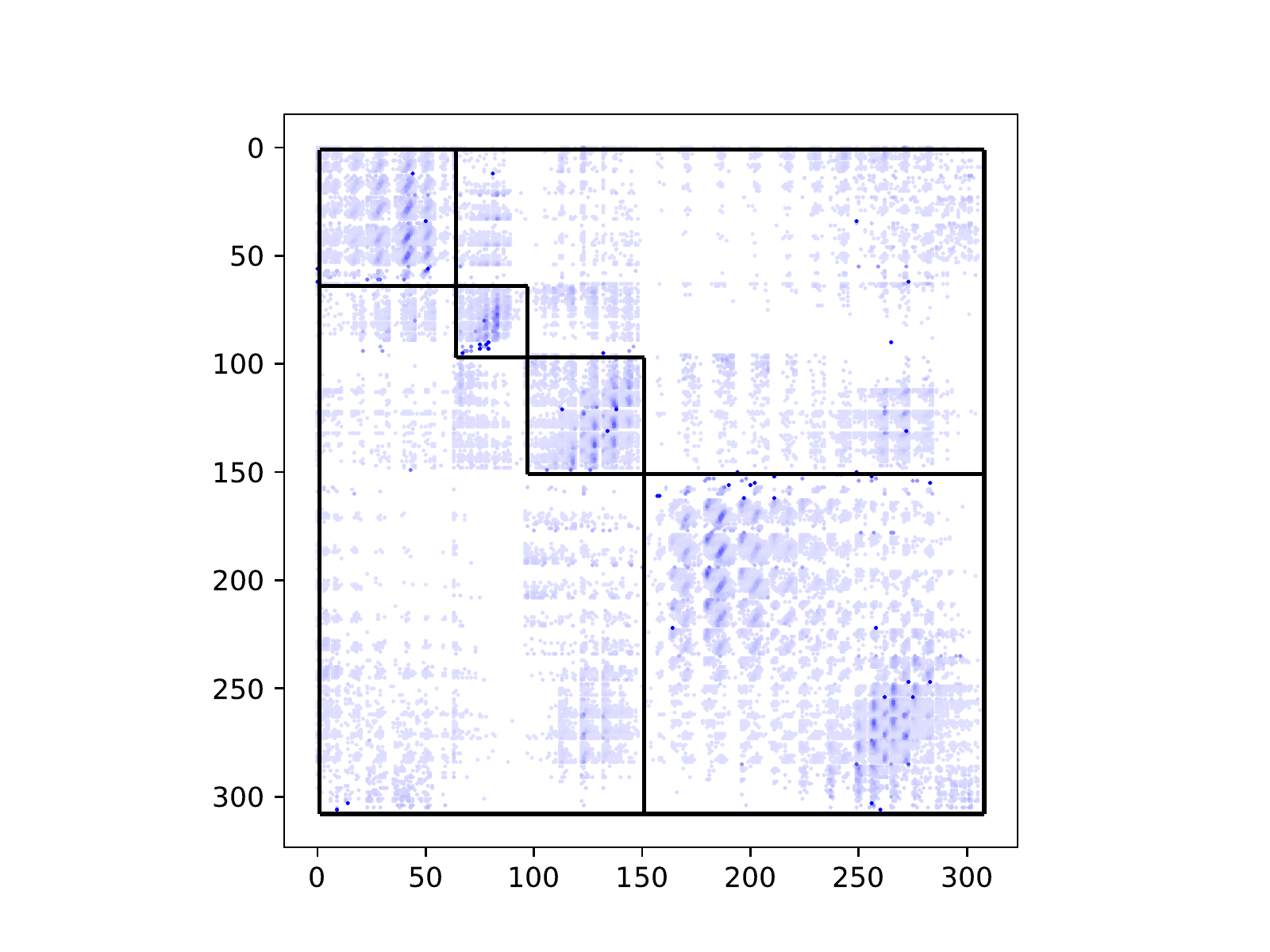}
\includegraphics[scale=.33]{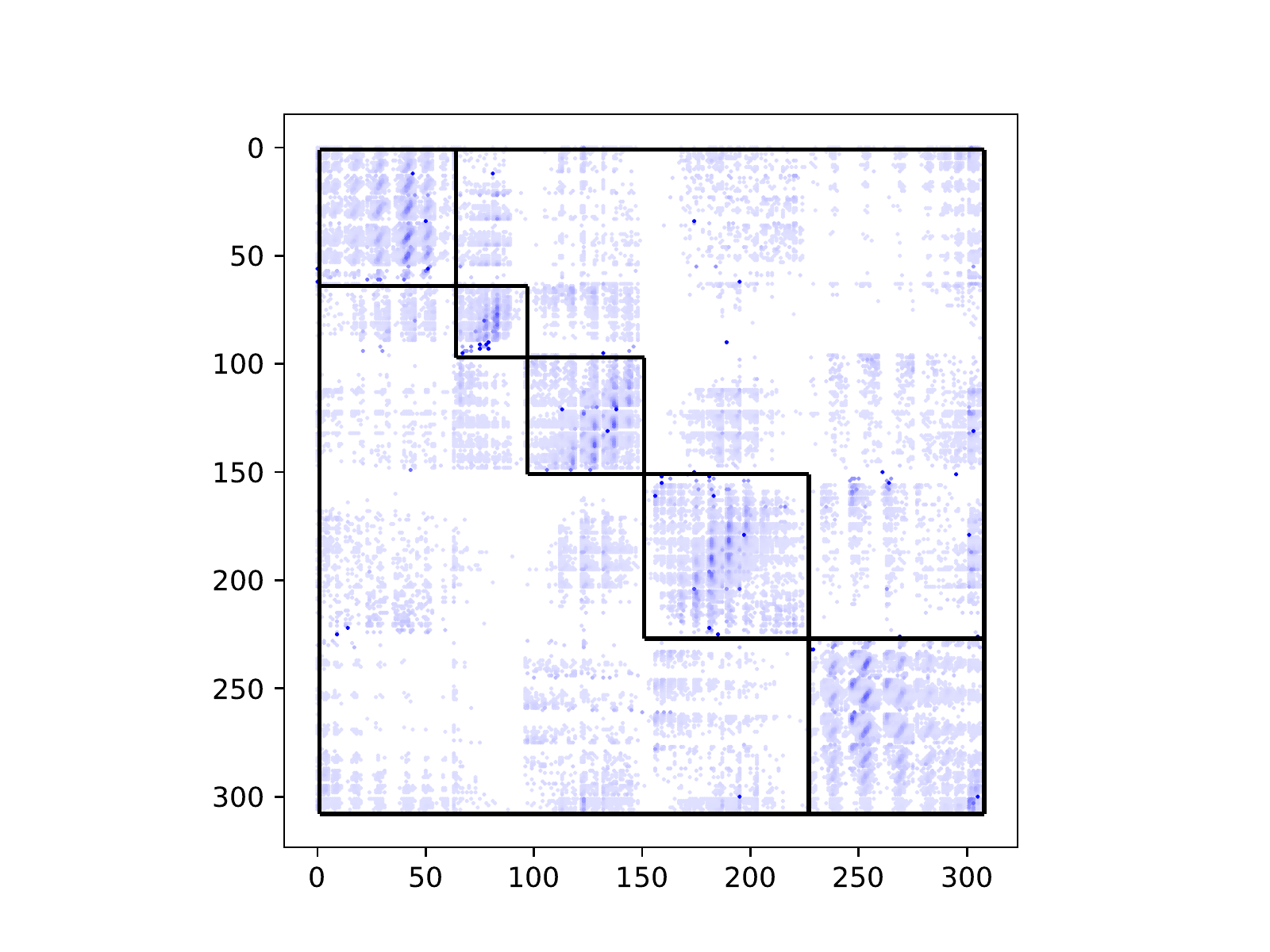}
\caption{Ph500, $\tau = .2$.  We find 5 clusters, of sizes 63, 33, 54, 76, and 81.}\label{fig:Ph500}
\end{figure}

\item	Ph500, $\tau = .3$.  Again, our algorithm outperforms that of~\cite{fritzsche2008preprint}, this time using the stationary vector $\pi$ of the input transition matrix in their algorithm's computations.  We obtained 6 clusters, of sizes 63, 33, 54, 76, 38, and 43 (see Figure~\ref{fig:Ph500_.3}).  These are the same clusters as those produced with $\tau = .2$, except the largest one has been split in two.  This yields the coupling matrices
\[W_v = 
\begin{bmatrix}
    0.9012    & 0.0212    & 0.0160    & 0.0233    & 0.0088    & 0.0296 \\
    0.1218    & 0.7921    & 0.0790    & 0.0045    & 0.0007    & 0.0018 \\
    0.0120    & 0.0353    & 0.8708    & 0.0316    & 0.0180    & 0.0323 \\
    0.0224    & 0.0017    & 0.0298    & 0.9011    & 0.0023    & 0.0428 \\
    0.0308    & 0.0023    & 0.0323    & 0.0089    & 0.7961    & 0.1298 \\
    0.0057    & 0.0005    & 0.0090    & 0.2083    & 0.0883    & 0.6882 
\end{bmatrix}
,\]
\[W_\1 = 
\begin{bmatrix}
    0.6841    & 0.0862    & 0.0428    & 0.1063    & 0.0188    & 0.0618 \\
    0.1378    & 0.6045    & 0.1973    & 0.0311    & 0.0082    & 0.0210 \\
    0.0318    & 0.0607    & 0.6789    & 0.0701    & 0.0622    & 0.0962 \\
    0.0957    & 0.0042    & 0.0537    & 0.6763    & 0.0126    & 0.1576 \\
    0.0404    & 0.0046    & 0.1162    & 0.0136    & 0.6253    & 0.2000 \\
    0.0543    & 0.0045    & 0.0904    & 0.1579    & 0.2265    & 0.4664
\end{bmatrix}
,\]
which have average diagonal entries .8249 and .6226, respectively.  In~\cite{fritzsche2008preprint} they found 6 clusters, of sizes 45, 43, 37, 71, 51, and 60, yielding coupling matrices $W_\pi$ and $W_\1$ with average diagonal entries .7441 and .5856, respectively.

\begin{figure}
\includegraphics[scale=.33]{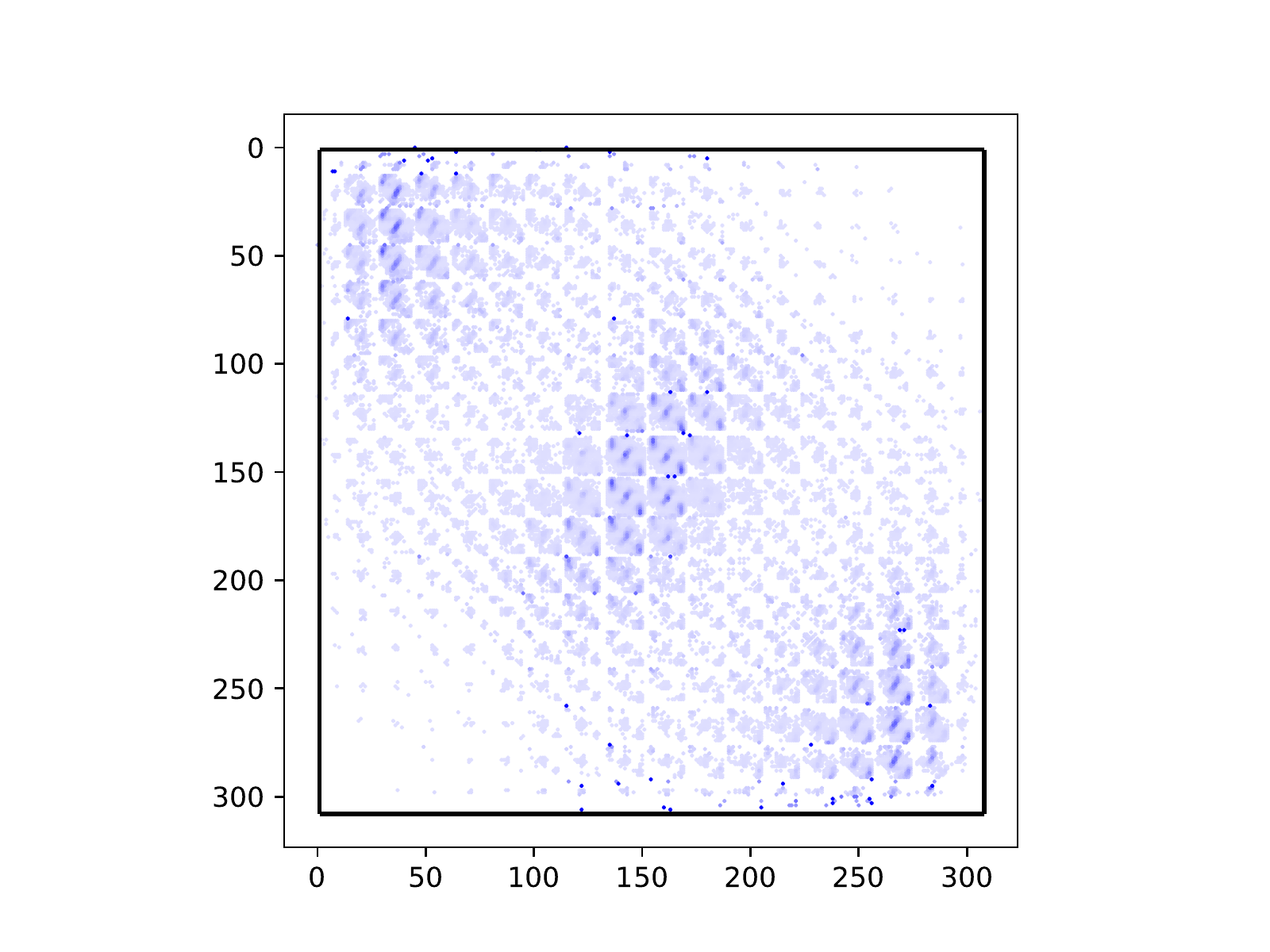}
\includegraphics[scale=.33]{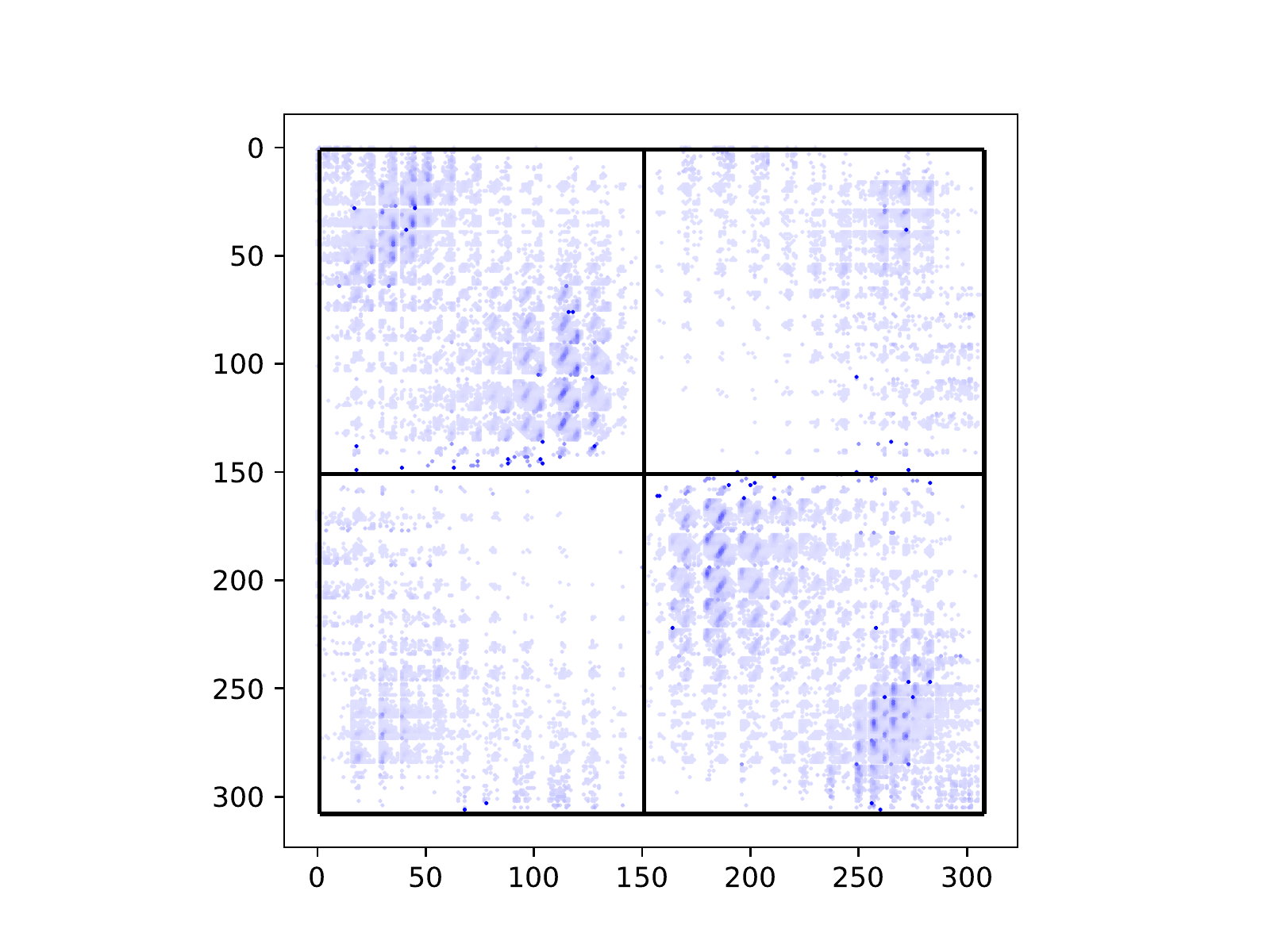}
\includegraphics[scale=.33]{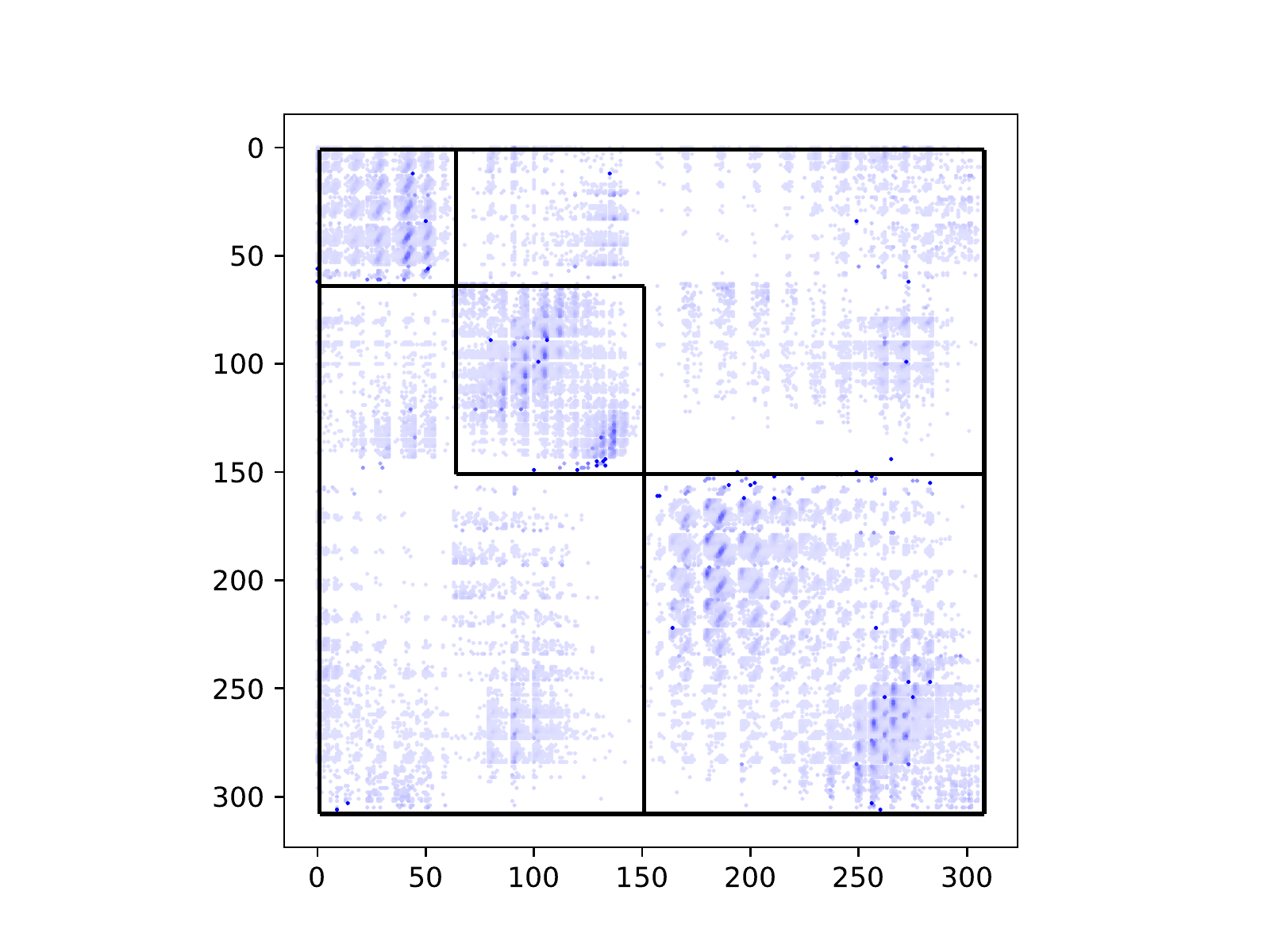}
\includegraphics[scale=.33]{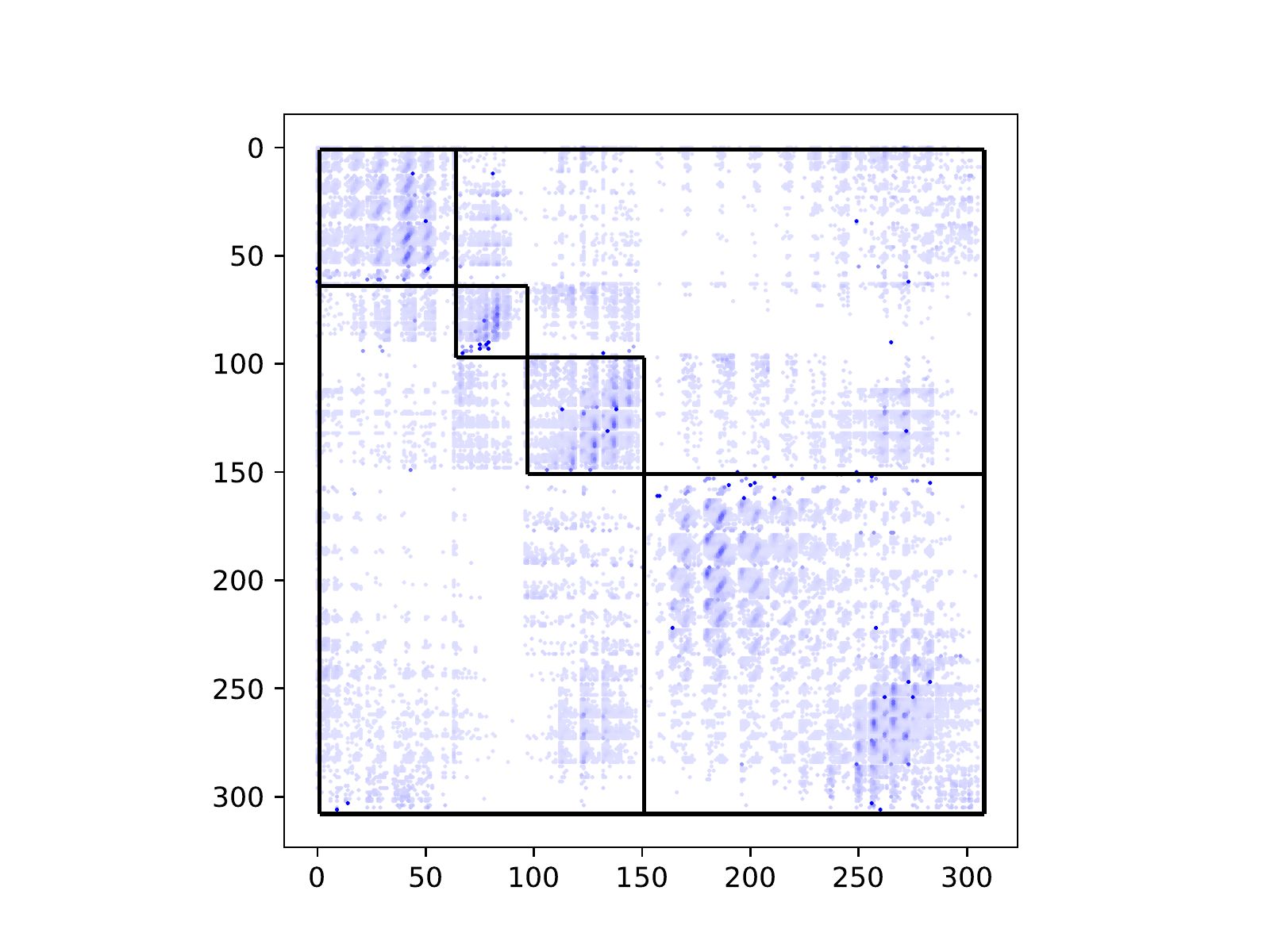}
\includegraphics[scale=.33]{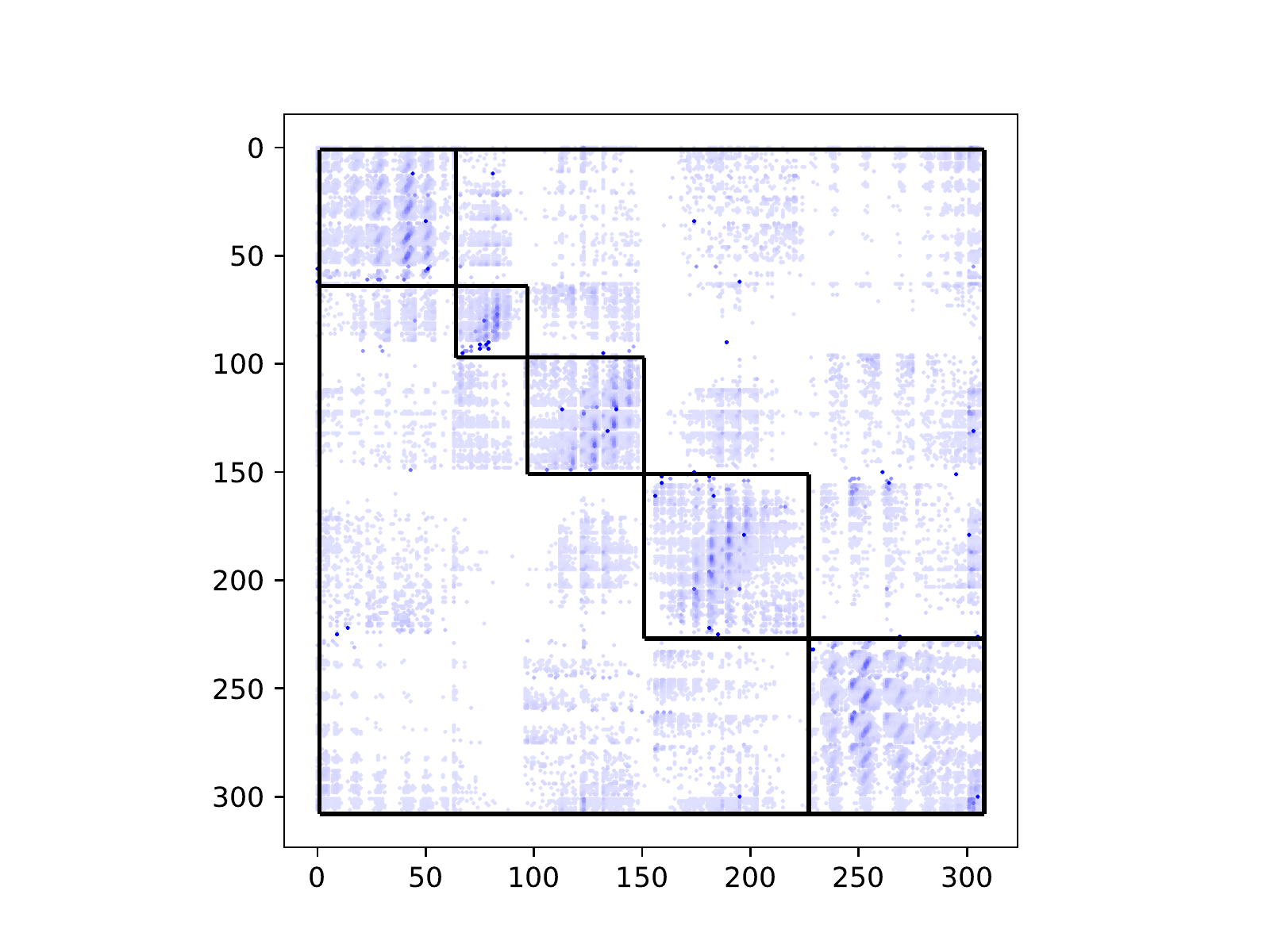}
\includegraphics[scale=.33]{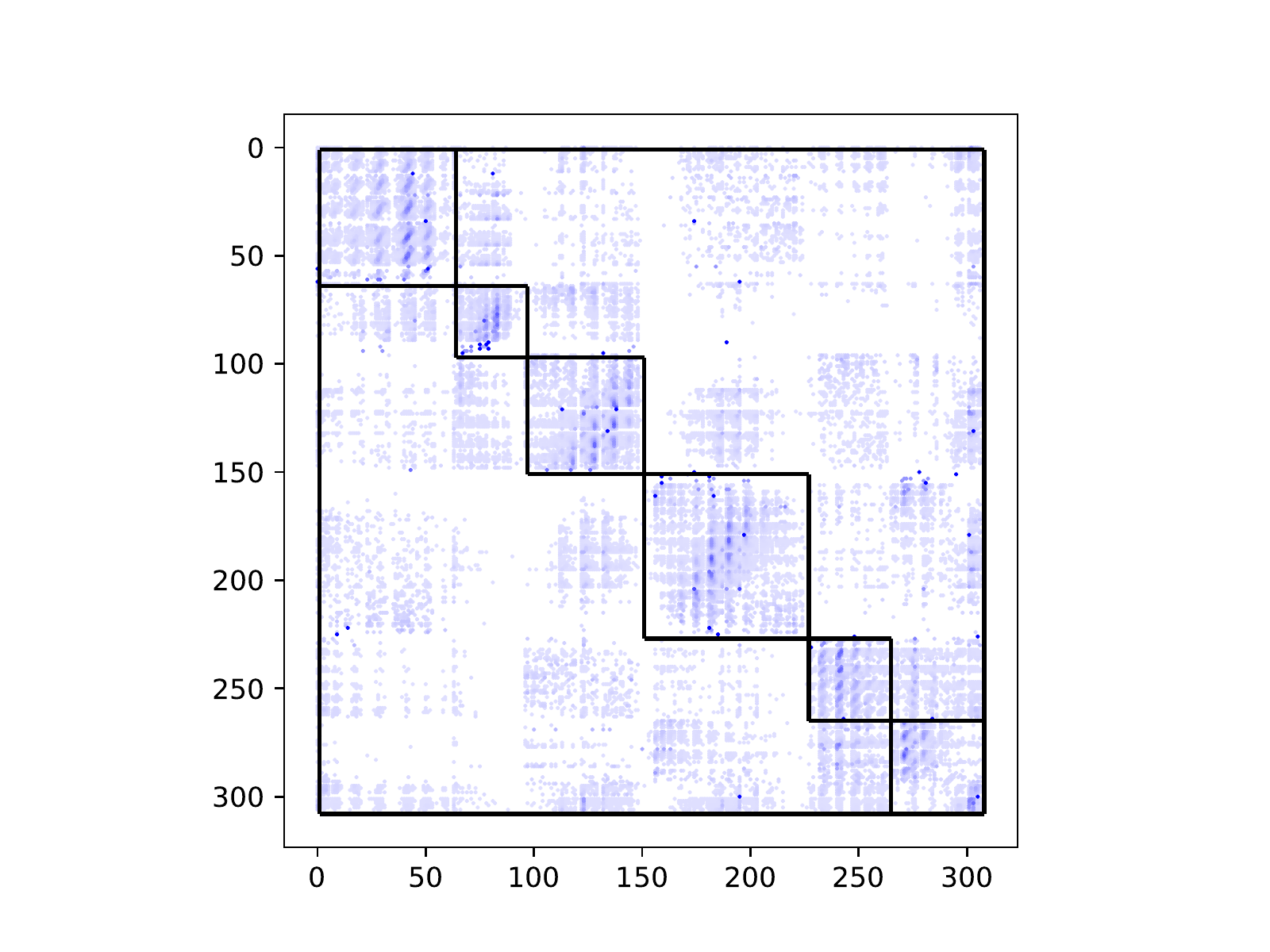}
\caption{Ph500, $\tau = .3$.  We find 6 clusters, of sizes 63, 33, 54, 76, 38, and 43.}\label{fig:Ph500_.3}
\end{figure}

\end{enumerate}

\subsection{A biological neural network} \label{bionet} 
Next we consider a data set for a neural network associated with the 
Caenorhabditis elegans worm; the data set was downloaded from \url{https://toreopsahl.com/datasets/}. The network consists of a weighted directed graph, where vertices represent neurons and the weight of an arc is the number of synapses and/or  gap junctions from one neuron to another. The full network has $306$ vertices, and from that we extracted the largest strongly connected component, which has $257$ vertices. From the weighted adjacency matrix of order $257$, we normalised by dividing each row by its sum so as to produce a stochastic matrix $T$. 

Running 
Algorithm~\ref{alg} with a tolerance of $0.1$ produces three clusters, of sizes $68, 118,$ and $71$. The Perron values of the principal submatrices of $T$ corresponding to the clusters are $0.8660, 0.9808,$ and $ 0.9836,$ respectively.  
The coupling matrix arising from the left-iterative weight vector $u$ computed by Algorithm~\ref{alg} is given by 
$$W_u(T) = 
\left[\begin{array}{ccc}
	0.7551  &  0.1148  &  0.1301\\
	0.0004   & 0.9886 &   0.0111\\
	0.0013    &0.0217&    0.9770
\end{array} \right], 
$$
while 
$$W_{\1}(T)=
\left[\begin{array}{ccc}
	0.5642  &  0.1910   & 0.2448\\
	0.0355   & 0.8033  &  0.1613\\
	0.0732    &0.3079 &   0.6189\\
\end{array} \right].$$
Evidently the diagonal entries of $W_u(T)$ better reflect the list of Perron values than the diagonal entries of $W_{\1}(T)$. Figure \ref{fig:neural} illustrates.

\begin{figure}
	\includegraphics[scale=.33]{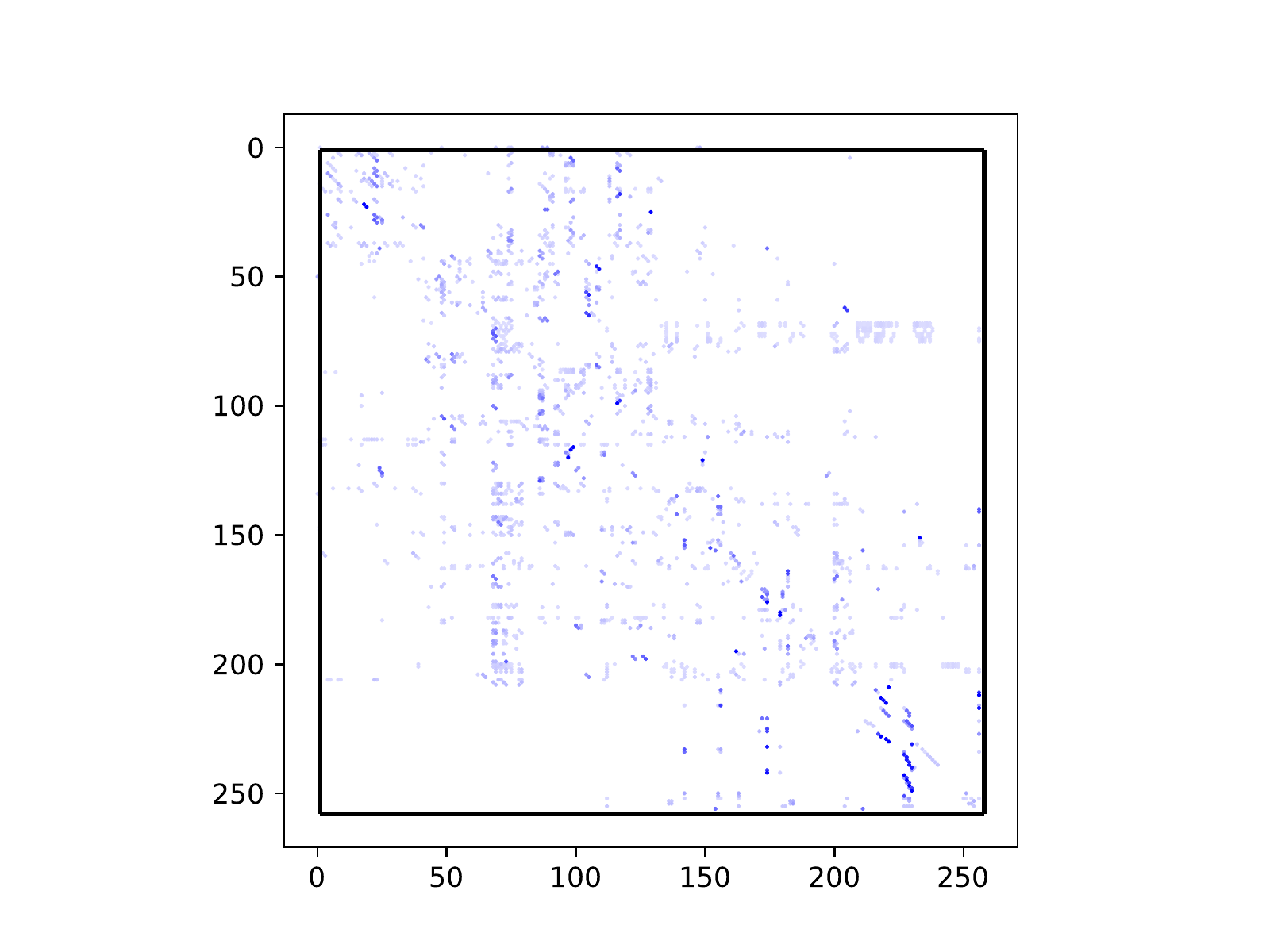}
	\includegraphics[scale=.33]{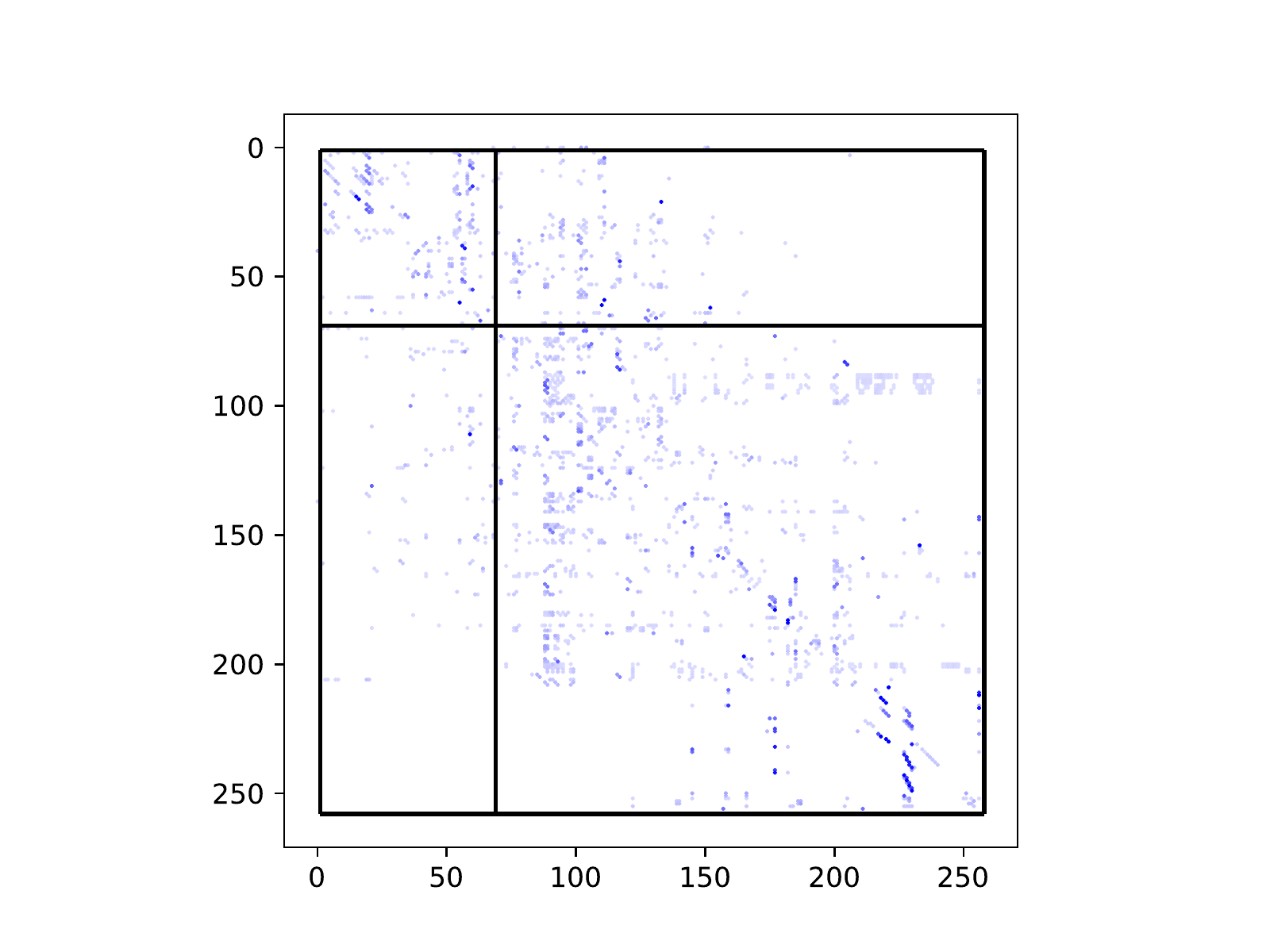}
	\includegraphics[scale=.33]{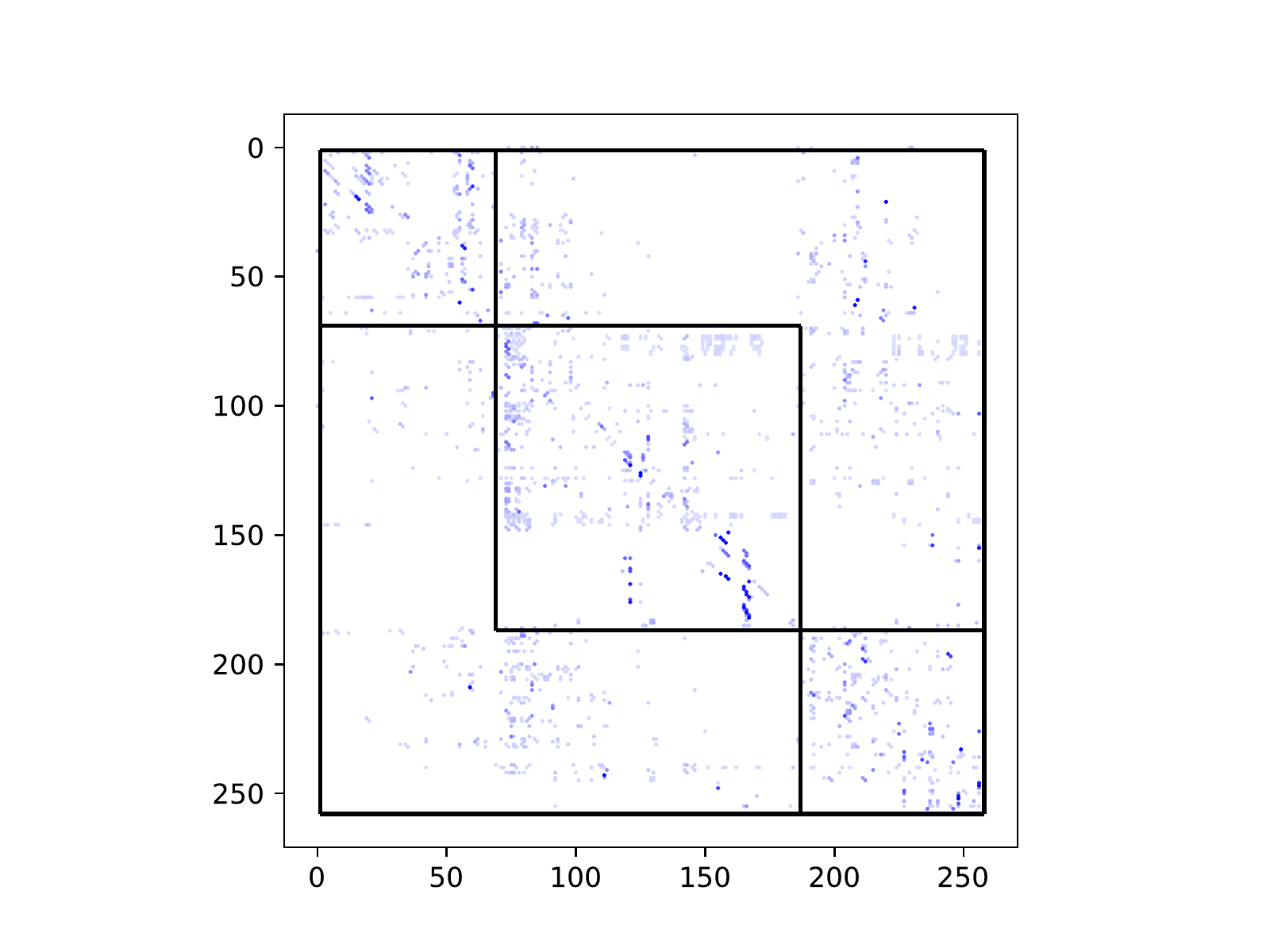}

	\caption{The results of Algorithm~\ref{alg} applied to the stochastic matrix arising from the biological neural net in subsection \ref{bionet}}\label{fig:neural}
\end{figure}

\subsection{Randomly generated Markov chains}\label{sec:random}

In this section we test Algorithm~\ref{alg} with tolerance $\tau = .5, .6$ on stochastic matrices generated randomly from several distributions with embedded cluster structure.  In all of our examples, the index set consists of 4 of sizes 100, 100, 50, and 25, and the $(i, j)$ entry has higher expectation when $i$ and $j$ are in the same cluster than when $i$ and $j$ come from different clusters.  We use the following distributions:
\begin{enumerate}
\item	\textbf{Normalised clustered uniform ensemble (NCUE).}  We generate the entries of a random matrix $X = [x_{ij}]$ independently at random, with $x_{ij} \sim \unif([0, 2p])$ if $i$ and $j$ are in the same cluster and $x_{ij} \sim \unif([0, 2q])$ if $i$ and $j$ are in different clusters, where $p > q$.  We then apply a random permutation to the rows and columns of $X$, and then normalise $X$ so that each row sums to 1.  
\item	\textbf{Normalised clustered Bernoulli ensemble (NCBE).}  We generate the entries of $X = [x_{ij}]$ independently at random, with $x_{ij} \sim \ber(p)$ if $i$ and $j$ are in the same cluster and $x_{ij} \sim \ber(q)$ if $i$ and $j$ are in different clusters, where $p > q$.  We then normalise $X$ so that each row sums to 1.  
\end{enumerate}

We run Algorithm~\ref{alg} on 1000 samples from each of these distributions
and report the following statistics:
\begin{itemize}
\item	\textbf{Average number of clusters identified.}
\item	\textbf{Average average diagonal entry of $W_v$ and $W_\1$} (the coupling matrices with respect to the left-iterative weight vector and $\1$, respectively).
\item	\textbf{Average minimum diagonal entry of $W_v$ and $W_\1$}.
\item	\textbf{Percentage of sample matrices whose clusters were \emph{fully recovered}.}  The clusters are fully recovered if the clusters returned by the algorithm are \emph{exactly} the same as the ground truth clusters used to generate the sample matrices.
\item	\textbf{Average number of errors.}  The number of errors in the set of empirical clusters returned by Algorithm~\ref{alg} is the number of \emph{misclassified} indices.  It is computed by assigning each empirical cluster a corresponding ground truth cluster (adding empty dummy clusters to either set if necessary), then taking $1 / 2$ the sum of the sizes of the symmetric differences of each pair of clusters.  The corresponding ground truth clusters are assigned to the empirical clusters so that this sum is minimized.

More concretely, let $C_1, \ldots, C_k$ be the ground truth clusters and $C_1', \ldots, C_{k'}'$ be the empirical clusters.  If $k \neq k'$, add copies of $\emptyset$ to either the empirical or ground truth clusters so that both collections have $l := \max\{k, k'\}$ clusters.  To each pair $C_i, C_j'$, assign a weight $w(i, j) := |(C_i \setminus C_j') \cup (C_j' \setminus C_i)|$ (the size of the symmetric difference between $C_i$ and $C_j'$).  Then the number of errors is given by 
\begin{equation}\label{eq:err}
\err(C_1, \ldots, C_l; C_1', \ldots, C_l') := \frac12\min_{\pi \in S_l}\sum_{j = 1}^lw(\pi(j), j),
\end{equation}
where $S_l$ denotes the set of permutations on the set $\{1, \ldots, l\}$.  Note that we divide by 2 because otherwise each error is counted twice, and that~\eqref{eq:err} can be computed efficiently by finding a \emph{minimum weight perfect matching} on the complete bipartite graph $K_{l, l}$ with edge weights given by $w(i, j)$; see Figure~\ref{fig:Kll}.
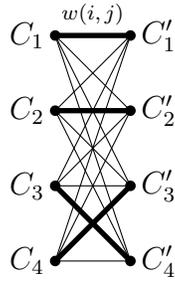
\begin{figure}
\centering
\begin{tikzpicture}
\foreach \y/\py in {1/1, 2/2, 3/4, 4/3} {
	\foreach \x/\an/\lbl in {1/east/C, 2/west/C'}
		\fill (\x, -\y) circle(2pt) node[anchor=\an]{$\lbl_\y$};
	\foreach \z in {1, 2, 3, 4}
		\draw (1, -\y) -- (2, -\z);
	\draw[line width=2pt] (1, -\y) -- (2, -\py);
}
\draw (1.5, -1) node[anchor=south]{\scriptsize$w(i, j)$};
\end{tikzpicture}
\caption{An optimal assignment of empirical clusters to ground truth clusters.  Each empirical cluster $C'_j$ is assigned a unique ground truth cluster $C_{\pi(j)}$ so that $\sum_jw(\pi(j), j)$ is minimized.}\label{fig:Kll}
\end{figure}

\end{itemize}

Our results are summarized in Table~\ref{table:results}, while Figures~\ref{fig:5err}-\ref{fig:37err} illustrate the performance of Algorithm~\ref{alg} on several examples with varying degrees of error.  One can observe the following about our results:

\newcommand{\colwidth}{.5in}
\newcommand{\headsize}{\scriptsize}

\begin{table}
\begin{center}
\scriptsize
\begin{tabular}{|p{.55in}|p{\colwidth}|p{\colwidth}|p{\colwidth}|p{\colwidth}|p{\colwidth}|p{\colwidth}|p{\colwidth}|}
\hline
	& \headsize{Avg.\ \# clusters}	& \headsize{Avg.\ avg.\ diag.\ entry $W_v$}	& \headsize{Avg.\ avg.\ diag.\ entry $W_\1$}		& \headsize{Avg.\ min.\ diag.\ entry $W_v$}	& \headsize{Avg.\ min.\ diag.\ entry $W_\1$}		& \headsize{\% fully recovered}	& \headsize{Avg.\ \# errors}	\\ 
\hline
\headsize{NCUE, $p = .95$, $q = .05$, $\tau = .5$}	& 2.9700	& 0.8583	& 0.8265	& 0.7257	& 0.6871	& 1.6\%	& 19.9960	\\	
\hline
\headsize{NCUE, $p = .95$, $q = .0095$, $\tau = .5$}	& 3.9790	& 0.9137	& 0.9064	& 0.7498	& 0.7372	& 48.4\%	& 3.0400	\\	
\hline
\headsize{NCUE, $p = .95$, $q = .05$, $\tau = .6$}	& 3.9550	& 0.7670	& 0.7489	& 0.5242	& 0.5046	& 6.6\%	& 13.9380	\\ 	
\hline
\headsize{NCUE, $p = .95$, $q = .0095$, $\tau = .6$}	& 4.0830	& 0.9281	& 0.9225	& 0.8218	& 0.8162	& 53.7\%	& 3.1850	\\	
\hline
\headsize{NCBE, $p = .95$, $q = .05$, $\tau = .5$}	& 3.5150	& 0.8071	& 0.7718	& 0.6203	& 0.5723	& 1.6\%	& 23.5570	\\	
\hline
\headsize{NCBE, $p = .95$, $q = .0095$, $\tau = .5$}	& 4.0190	& 0.9229	& 0.9033	& 0.8179	& 0.7722	& 42.4\%	& 7.1250	\\	
\hline
\headsize{NCBE, $p = .95$, $q = .05$, $\tau = .6$}	& 3.9620	& 0.7566	& 0.7287	& 0.5230	& 0.4747	& 1.5\%	& 19.0120	\\	
\hline
\headsize{NCBE, $p = .95$, $q = .0095$, $\tau = .6$}	& 4.0920	& 0.9150	& 0.8973	& 0.8013	& 0.7594	& 38.2\%	& 7.3700	\\	
\hline
\end{tabular}
\end{center}
\caption{Results of running Algorithm~\ref{alg} on randomly generated stochastic matrices.}\label{table:results}
\end{table}

\begin{figure}
\centering
	\includegraphics[scale=.3]{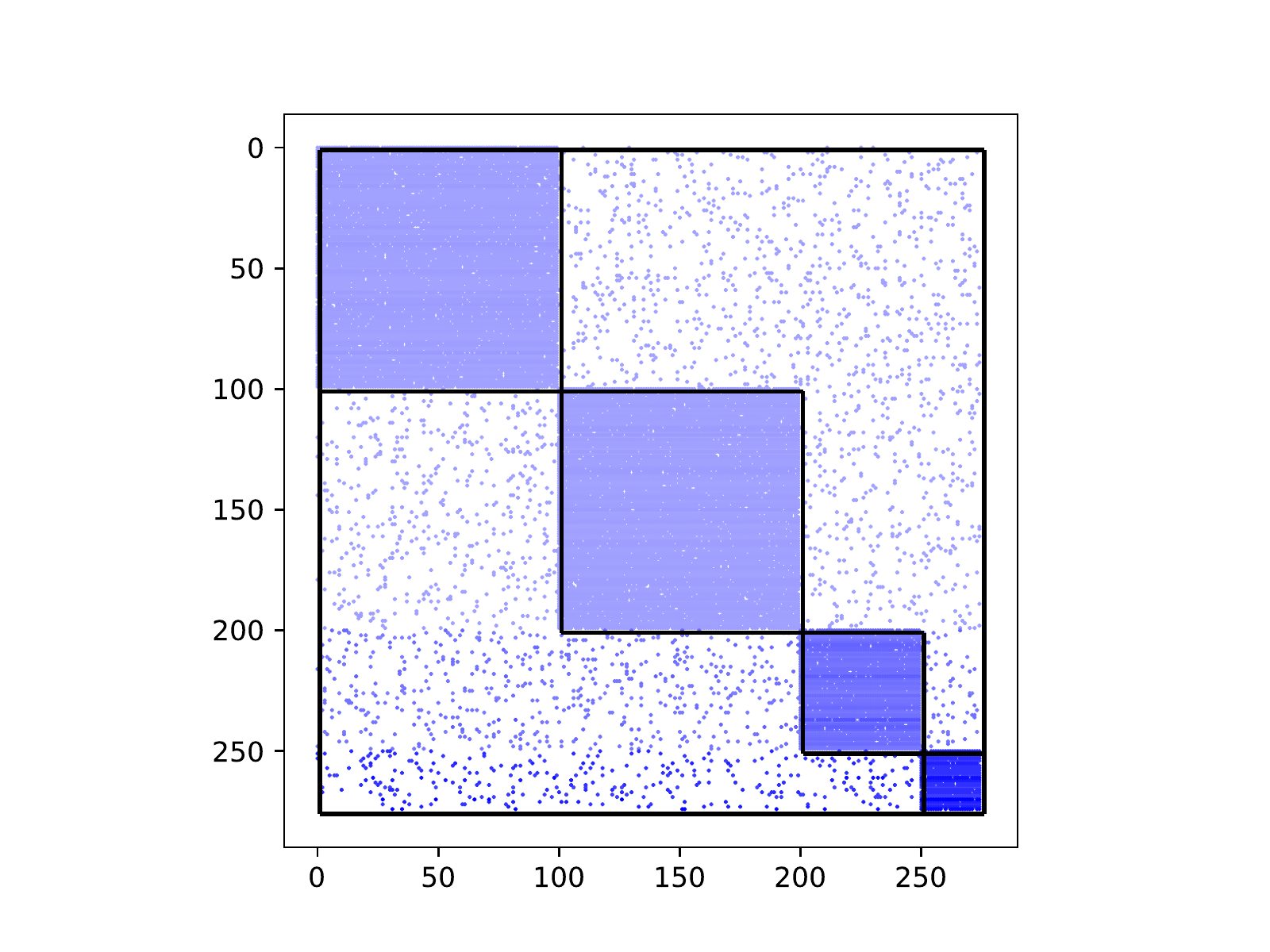}
	\includegraphics[scale=.3]{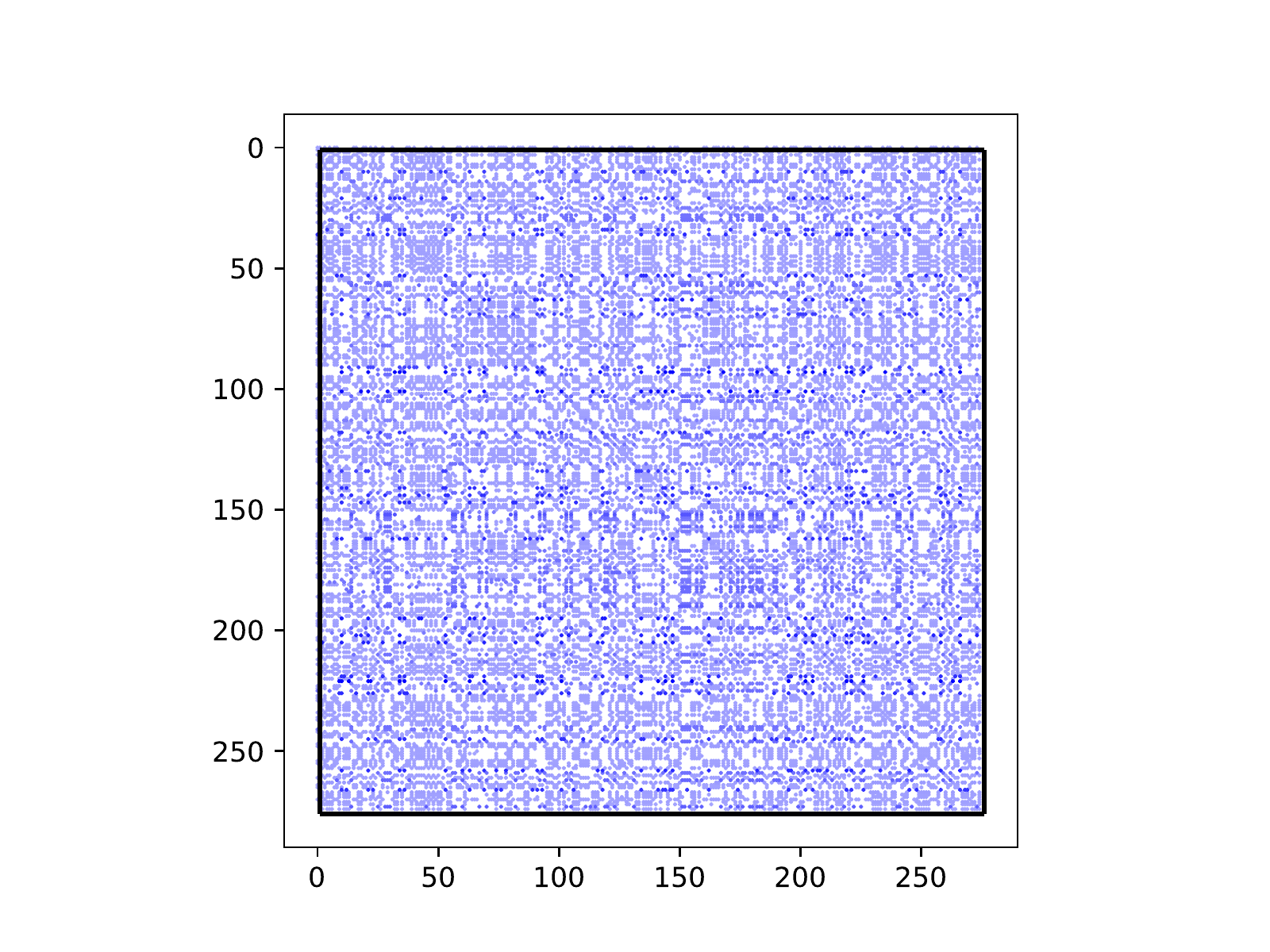}
	\includegraphics[scale=.3]{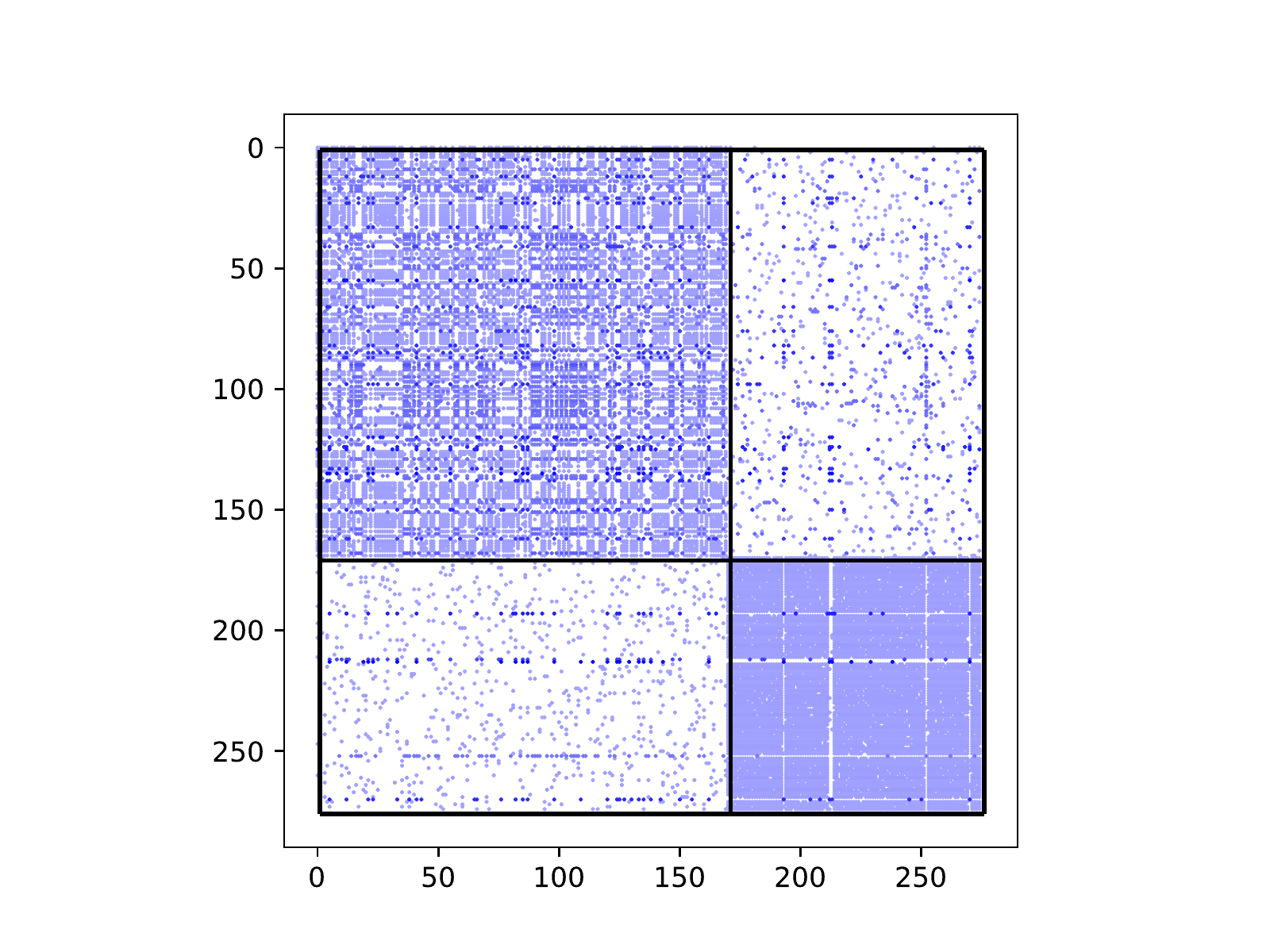}
	\includegraphics[scale=.3]{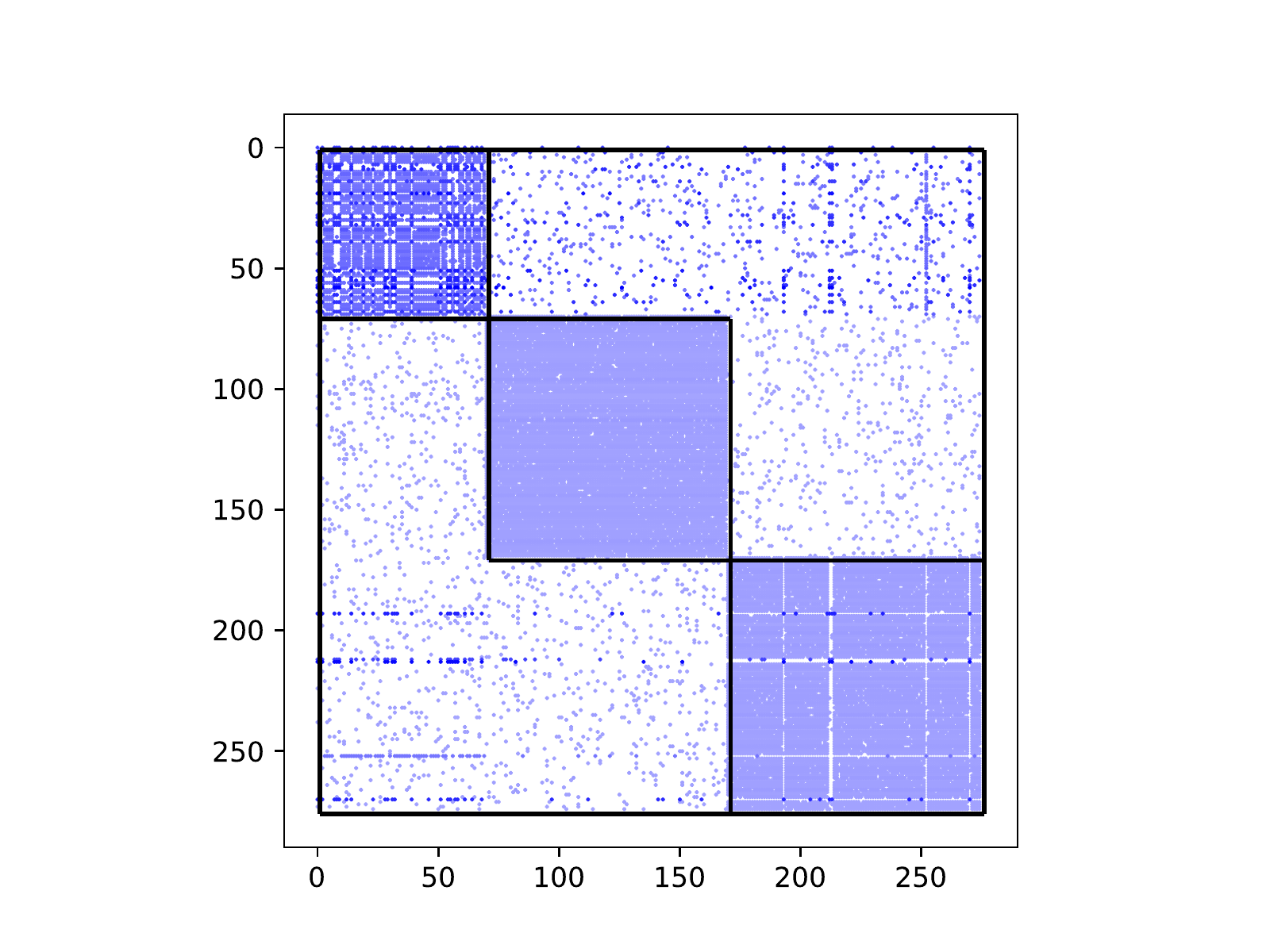}
	\includegraphics[scale=.3]{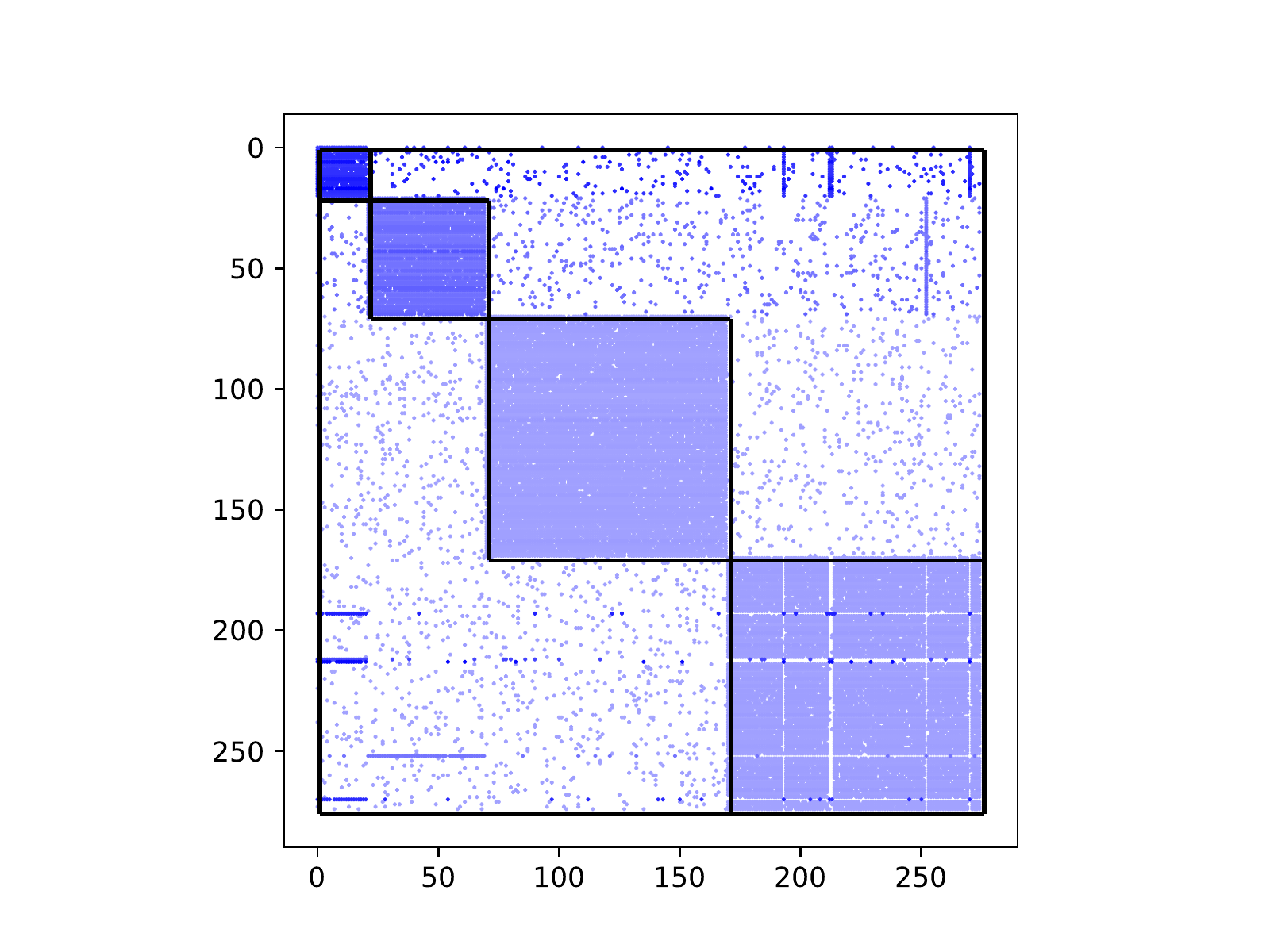}
	\caption{The results of running Algorithm~\ref{alg} on a sample from a NCBE with $p = .95, q = .05, \tau = .6$ and 5 errors.  The first image shows the original matrix, the second image shows the matrix after a random permutation of the rows and columns, and the remaining images show the iterations of Algorithm~\ref{alg}.  The algorithm produces 4 clusters of sizes 21, 49, 100, and 105, with 5 misclassified indices.}\label{fig:5err}
\end{figure}

\begin{figure}
\centering
	\includegraphics[scale=.3]{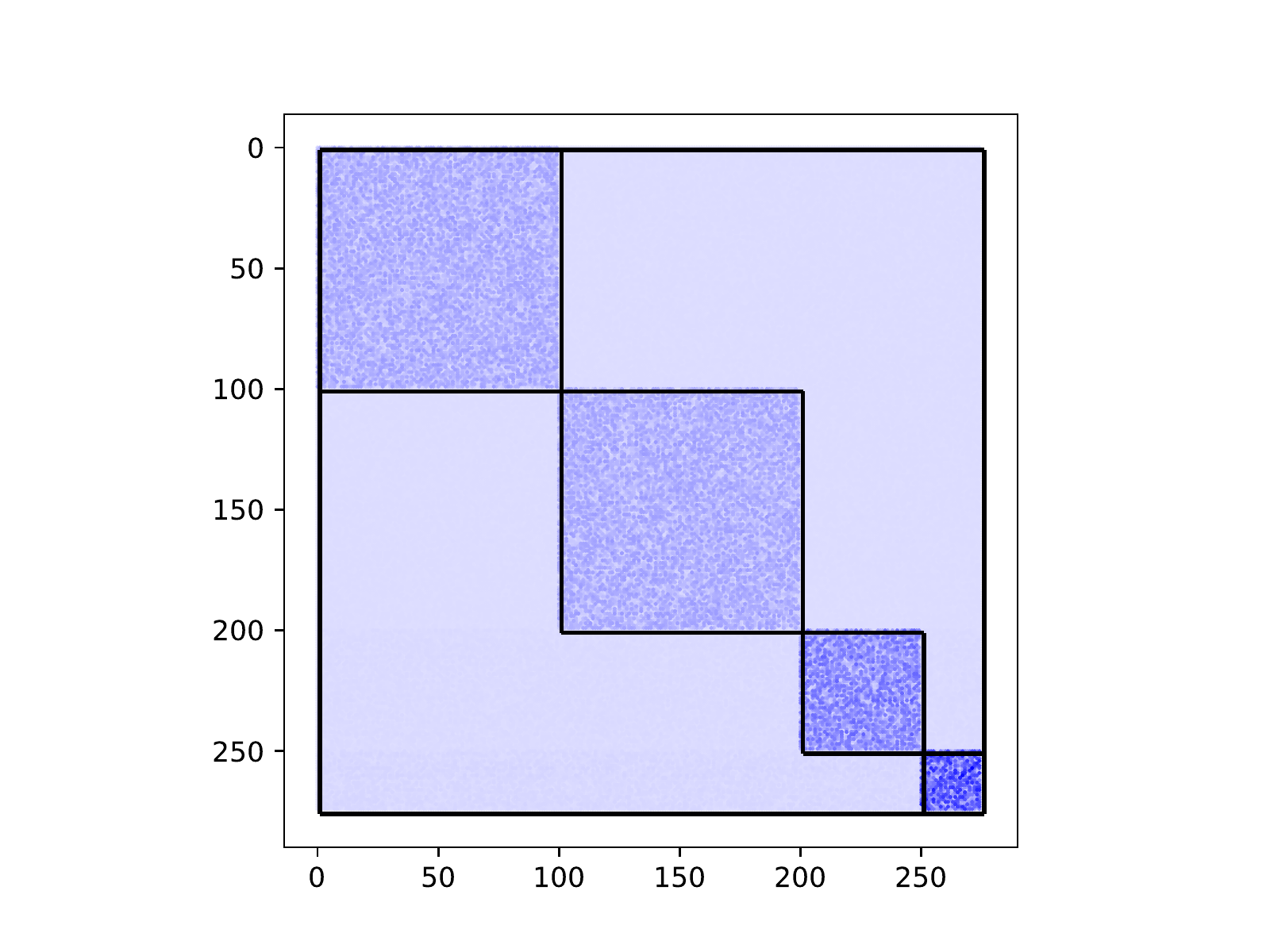}
	\includegraphics[scale=.3]{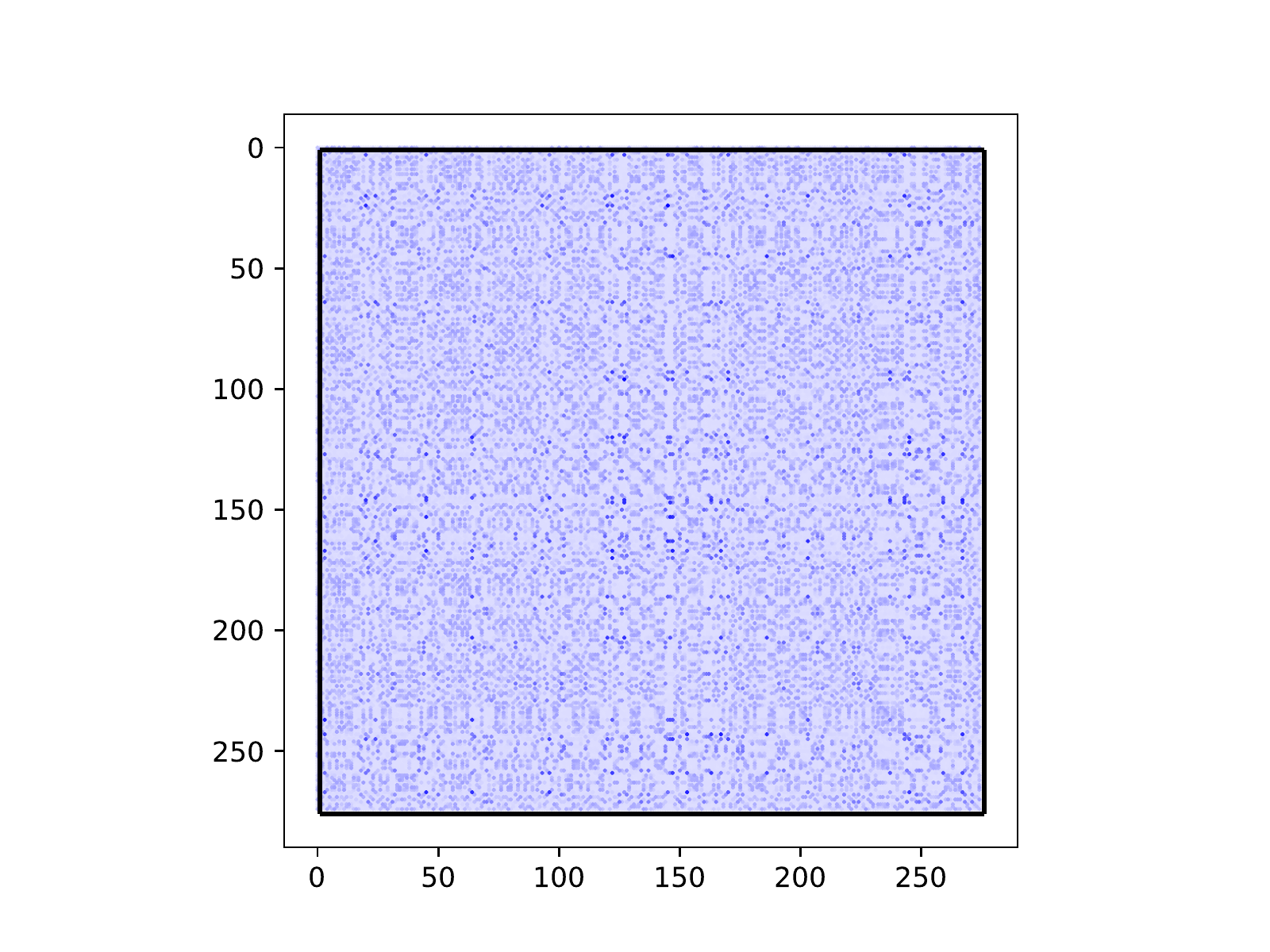}
	\includegraphics[scale=.3]{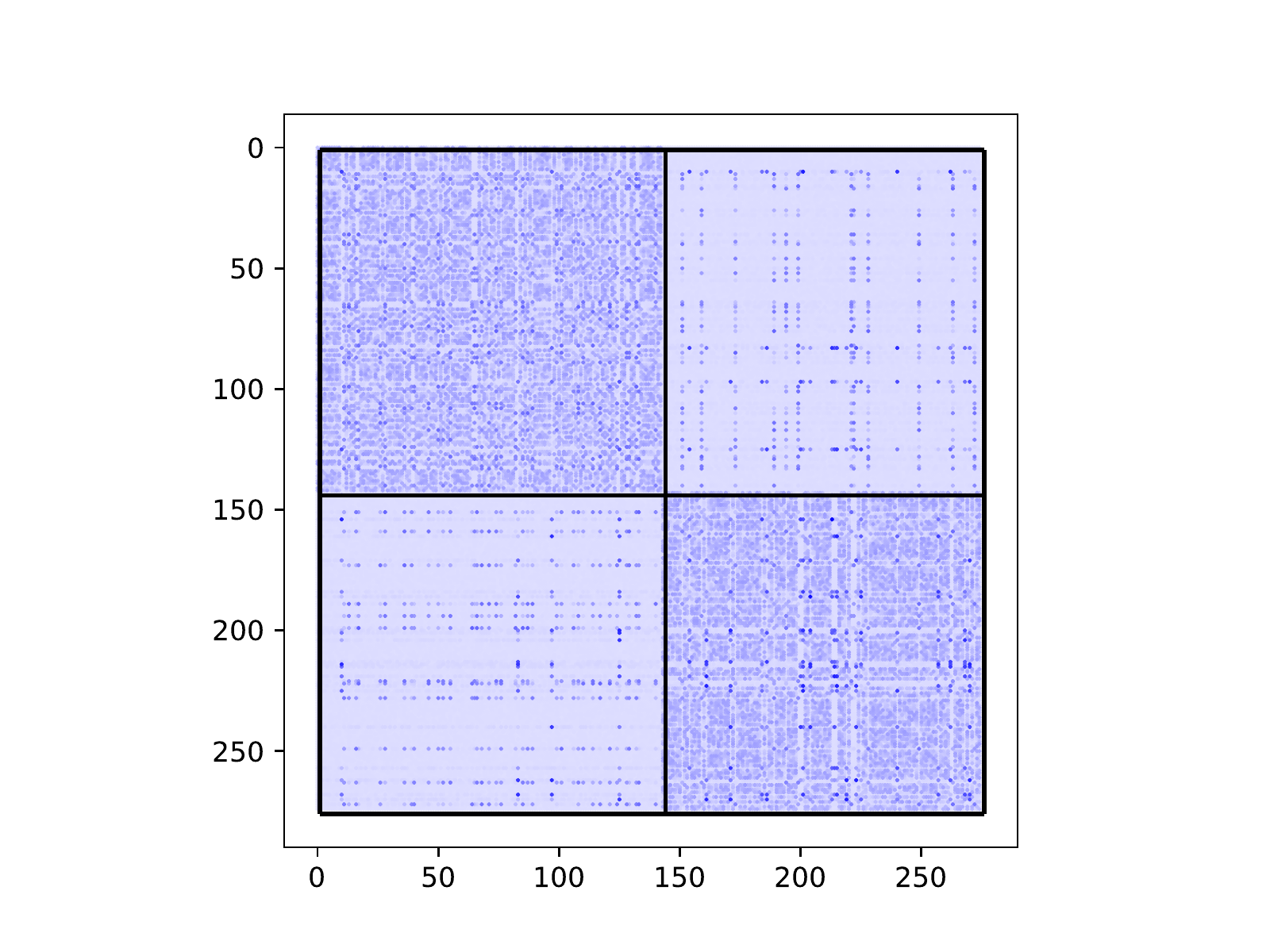}
	\includegraphics[scale=.3]{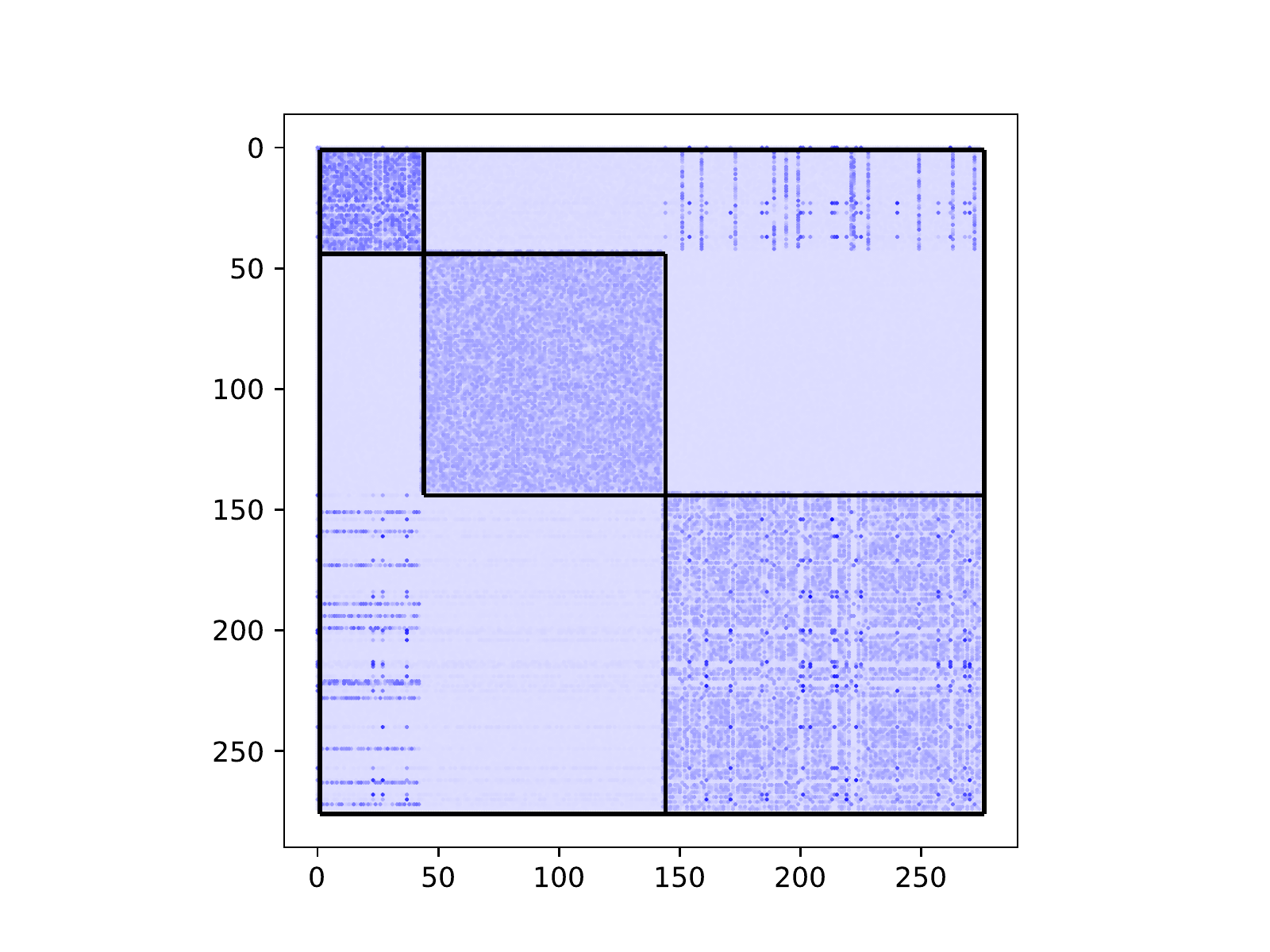}
	\includegraphics[scale=.3]{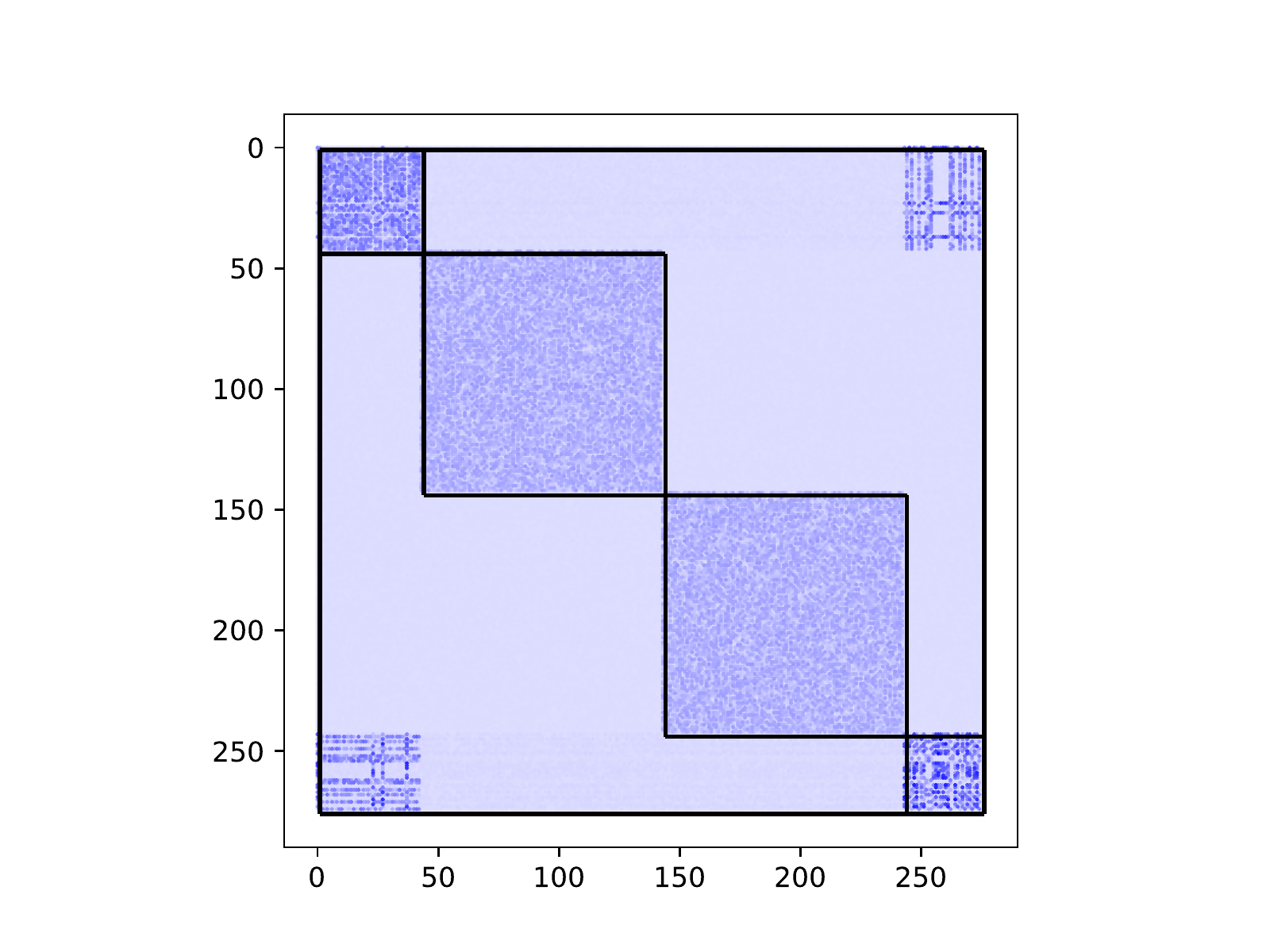}
	\caption{The results of running Algorithm~\ref{alg} on a sample from a NCUE with $p = .95, q = .05, \tau = .6$ and 17 errors.  The first image shows the original matrix, the second image shows the matrix after a random permutation of the rows and columns, and the remaining images show the iterations of Algorithm~\ref{alg}.  The algorithm produces 4 clusters of sizes 43, 100, 100, and 32, with 17 misclassified indices.}\label{fig:17err}
\end{figure}

\begin{figure}
\centering
	\includegraphics[scale=.3]{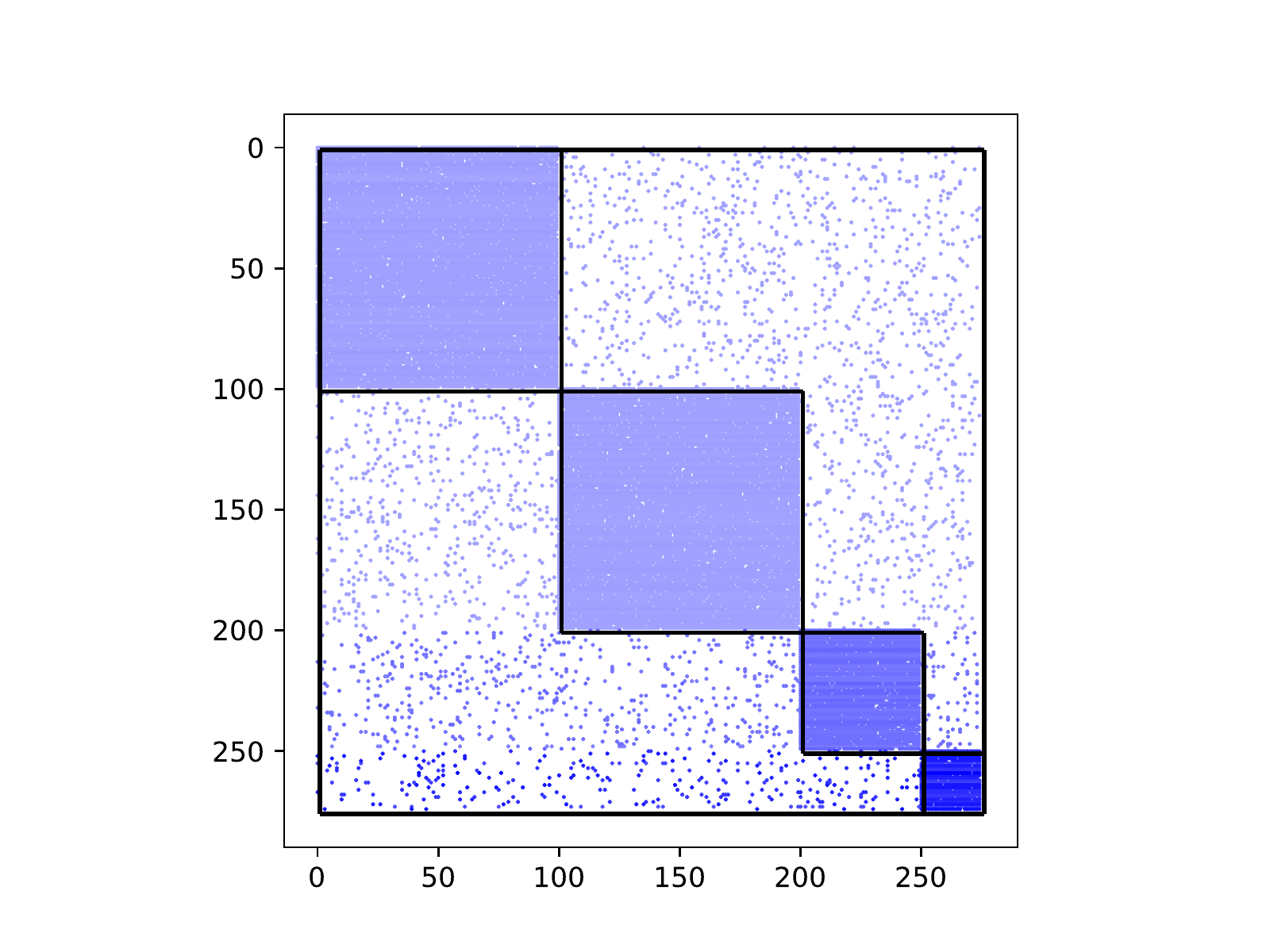}
	\includegraphics[scale=.3]{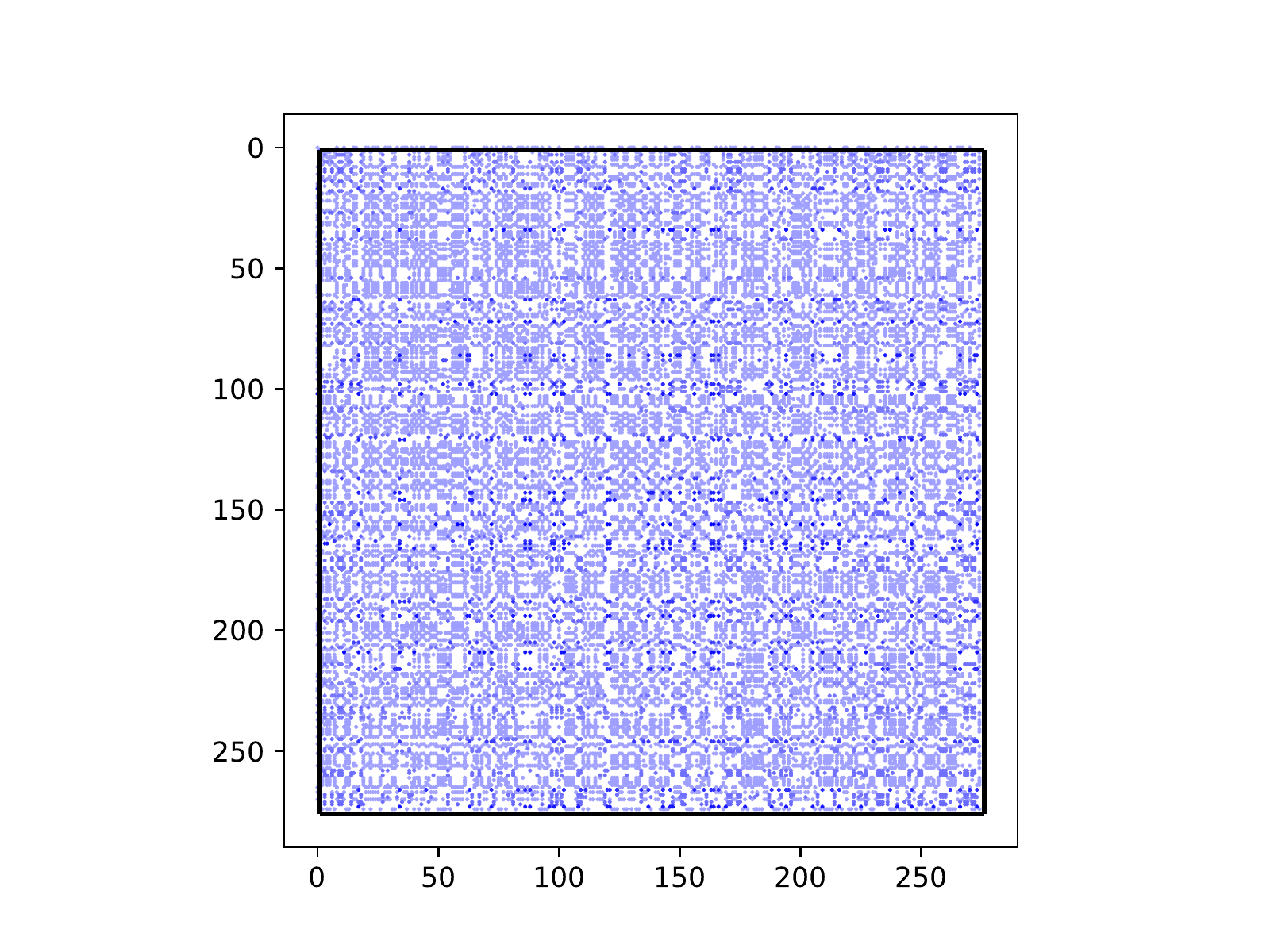}
	\includegraphics[scale=.3]{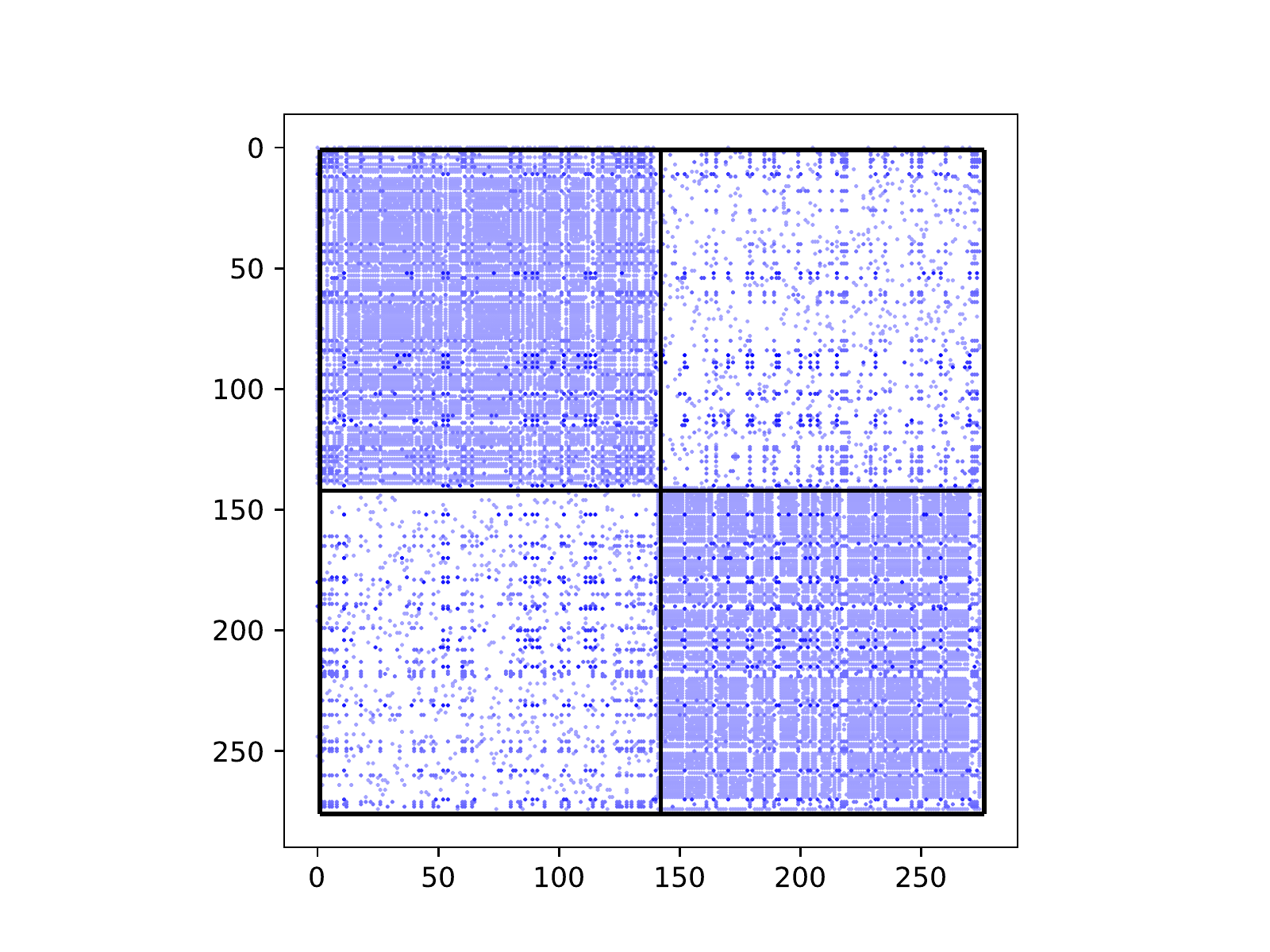}
	\includegraphics[scale=.3]{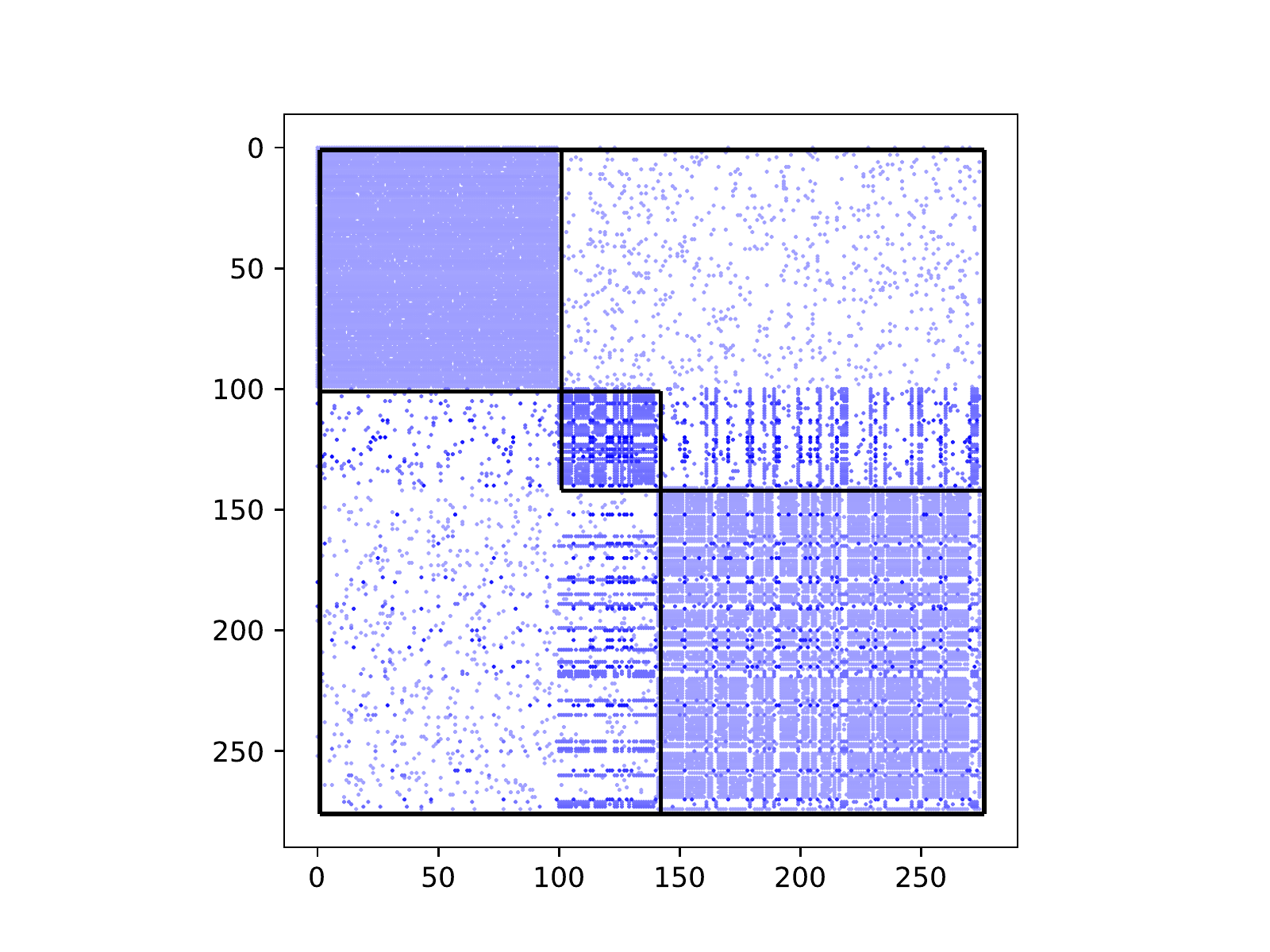}
	\includegraphics[scale=.3]{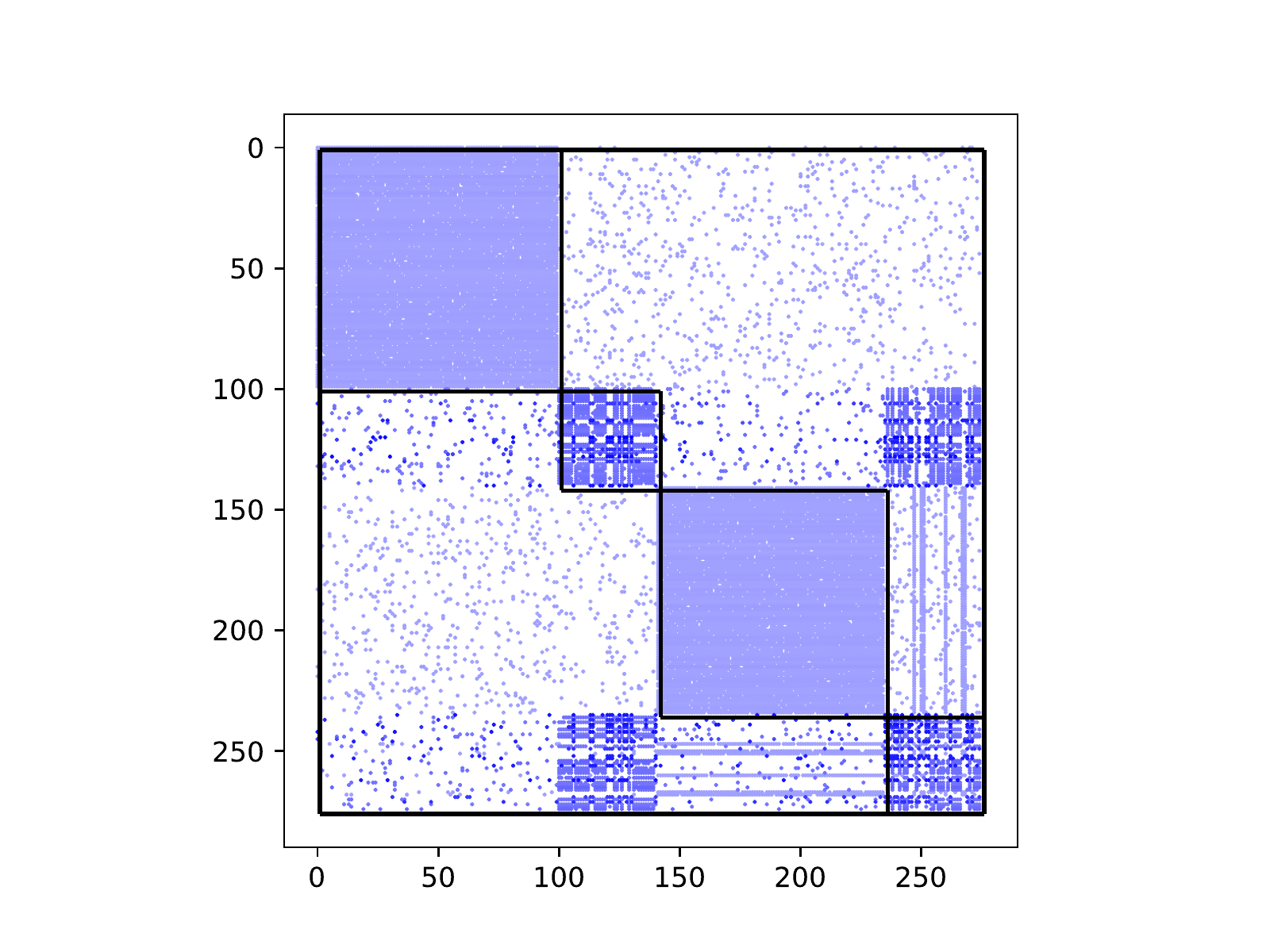}
	\caption{The results of running Algorithm~\ref{alg} on a sample from a NCBE with $p = .95, q = .05, \tau = .6$ and 37 errors.  The first image shows the original matrix, the second image shows the matrix after a random permutation of the rows and columns, and the remaining images show the iterations of Algorithm~\ref{alg}.  The algorithm produces 4 clusters of sizes 41, 100, 94, 40, with 37 misclassified indices.}\label{fig:37err}
\end{figure}

\begin{itemize}
\item	Unsurprisingly, when the inter-cluster probabilities are lower ($q = .0095$), the algorithm performs significantly better than when the inter-cluster probabilities are higher ($q = .05$).  Typically, when $q = .05$ the average number of clusters found is closer to 4, the diagonal entries of the coupling matrices and percentage fully recovered are higher, and the average number of errors is lower than when $q = .05$.  In particular, the percentage of fully recovered is much higher in these cases (38.2-53.7\% compared with 1.5-6.6\% when $q = .05$).

\item	The percentages of fully recovered may seem disappointing, especially when the inter-cluster probabilities are higher ($q = .05$), but in all cases the average number of errors is quite low.  In the worst case we have 23.557 errors on average, which means only 8.56\% of the states are misclassified.  Most of these are probably as a result of the algorithm identifying the wrong number of clusters: in this case, empty dummy clusters must be added to either the empirical or ground truth clusters when calculating the number of errors; the number of errors will necessarily count at least one entire cluster (either empirical or ground truth).

\item	There is a high degree of correlation between the first and last two columns.  

The Bernoulli case with $q = .05, \tau = .6$ is an interesting exception.

\item	In the regime with higher inter-cluster probabilities ($q = .05$), the algorithm seems to do better with the uniform ensemble than with the Bernoulli ensemble.  In the regime with lower inter-cluster probabilities ($q = .0095$), the results are comparable with the two distributions.


\item	There is a tradeoff between inter-cluster probabilities (determined by $q$) and the tolerance $\tau$.  When $q$ is larger, the algorithm identifies fewer clusters on average; hence, increasing the tolerance from .5 to .6 brings us closer to the correct number of clusters.  Hence, in both the uniform and Bernoulli cases with larger inter-cluster probabilities, a higher tolerance yields a higher percentage of fully recovered and lower average number of errors.

\item	The placement of the misclassified states differs in Figures~\ref{fig:5err}-\ref{fig:37err}.  In Figure~\ref{fig:5err}, (empirical) cluster 4 contains four states that should have been placed in (empirical) cluster 1 and one state that should have been placed in cluster 2; in Figure~\ref{fig:17err}, all 17 errors come from states in cluster 4 which should have been placed in cluster 1; and in Figure~\ref{fig:37err}, cluster 4 contains some states which should have been placed in cluster 2, and vice versa.
\end{itemize}

\section{Open problems}\label{sec:openproblems}

We have mentioned several open problems and areas for improvement throughout this paper.  We list them here:

\begin{itemize}
\item	\textbf{Left-iterative weight vector vs.\ ones vector.}  As noted in Remark~\ref{rmk:1vsLIWV}, there is a disconnect between our guarantees for the coupling matrix with respect to $\1$ and the left-iterative weight vector.  In Subsection~\ref{subsec:smallsing=>clustering} we show if the Laplacian of a Markov chain has a small positive singular value, then it exhibits clustering with respect to the left-iterative weight vector.  However, in Subsection~\ref{subsec:couplingmatrix}, we show a converse result \emph{with respect to $\1$}.  Thus, it remains an open problem to prove the converse implication of one or both of Theorems~\ref{thm:small_sing} and~\ref{thm:directsum_stochastic}.  Moreover, as noted in Remark~\ref{rmk:n^-3/2}, Theorems~\ref{small_sing} and~\ref{thm:directsum_stochastic} require the diagonal entries of $W_\1$ to be at least $1 - O(n^{3/2})$ in order to guarantee a meaningful bound on the second smallest singular value of the Laplacian; this seems like a rather strong requirement, and can perhaps be improved.

\item	\textbf{Inessential states.}  As mentioned in Remark~\ref{rmk:inessential}, our analysis does not handle the case when the Markov chain corresponding to the \emph{unperturbed} matrix contains one or more inessential states.  In such cases, the transition matrix is no longer a perturbation of a direct sum of stochastic matrices, but of a block-lower triangular stochastic matrix with irreducible diagonal blocks (after permuting the rows), with all inessential states in the lowest row of blocks.  Such Markov chains could still be considered ``clustered'' if the diagonal blocks have large transition probabilities compared to the off-diagonal blocks, with the exception of the last row of blocks (which are inessential).

\item	\textbf{Choosing a good tolerance.}  In practice, one way of choosing a tolerance for Algorithm~\ref{alg} is to simply look for a ``gap'' in the singular values of $I - T$.  Some care should be taken to quantify exactly how large a gap we should look for, and whether the tolerance should be updated in recursive calls to the algorithm.

Choosing the ``right'' tolerance is important because there is a tradeoff between the tolerance and the number of clusters identified: choosing too small a tolerance can result in no clusters being detected, while choosing too large a tolerance can result in ``overclustering,'' i.e.\ identifying a large number of small clusters, leading to a coupling matrix with small diagonal entries.

Rather than choosing a tolerance for the small singular values, we could use the diagonal entries of the coupling matrix as a stopping criterion for Algorithm~\ref{alg}, as is done in~\cite{fritzsche2008svd}.  In this approach, one fixes a lower bound $1 - \delta$ for the diagonal entries of the desired coupling matrix, and runs the algorithm for as many iterations as possible without violating this lower bound.


\item	\textbf{Optimizing for two-mode networks.}  While our algorithm appears to perform quite well on Markov chains generated from two-mode networks (see Section~\ref{sec:twomode}), our algorithm and results do not take into account any special combinatorial or spectral properties of such Markov chains.  Perhaps doing so could result in improved guarantees for our algorithm in that case.

\item	\textbf{Provable guarantees for randomly generated Markov chains.}  In Section~\ref{sec:random} we make several observations about the results of running our algorithm on randomly generated input.  It would be nice if we could prove that if a Markov chain is generated randomly according to a certain distribution, then the clusters produced by Algorithm~\ref{alg} have certain guarantees, e.g.\ few errors or diagonally dominant coupling matrix, with high probability.
\end{itemize}


\section*{Acknowledgements} S.C. was supported as a PIMS Postdoctoral Fellow when much of the research in this paper was conducted. S.K.'s research is supported in part by NSERC Discovery Grant RGPIN-2019-05408 and the University of Manitoba's University Research Grant Program.  We thank Prof.\ Shmuel Friedland and the University of Illinois at Chicago Department of Mathematics, Statistics, and Computer Science for allowing us to use their computing resources.

\bibliographystyle{plain}
\bibliography{clustered_markov}

\appendix

\section{A coupling matrix with a small diagonal entry}\label{sec:smalldiag}

In this appendix we give the details of Example~\ref{ex:smalldiag}.  We find that
 \begin{eqnarray*}
 &&(I-T)(I-T^{\top})= \\
&&\left[ \begin{array}{c|c|c} 
2\epsilon^2 & -\epsilon \delta & \epsilon (y-x) \1^{\top} \\ \hline 
-\epsilon \delta & \delta^2 \frac{n-1}{n-2} & -\frac{\delta}{n-2}(x+(n-1)y) \1^{\top} \\
\hline \epsilon (y-x) \1 & -\frac{\delta}{n-2}(x+(n-1)y) \1&(x+y)^2 I +(x^2+y^2)J
\end{array}\right].
  \end{eqnarray*}
It follows that $x+y$ is a singular value of $I-T$ of multiplicity $n-3,$ and that the remaining singular values are the square roots of the eigenvalues of the following matrix: 
$$ M = 
\left[ \begin{array}{ccc} 
2\epsilon^2 & -\epsilon \delta & \epsilon (y-x)(n-2) \\  
-\epsilon \delta & \delta^2 \frac{n-1}{n-2} & -\delta(x+(n-1)y) \\
 \epsilon (y-x)  & -\frac{\delta}{n-2}(x+(n-1)y) &(x+y)^2  +(x^2+y^2)(n-2)
\end{array}\right].
$$
We find that the characteristic polynomial of $M$ is $p(z) = z[z^2 - ((n-1)(x^2+y^2)+2xy+2\epsilon^2 + \delta^2 \frac{n-1}{n-2}  )z +
\epsilon^2 n(x+y)^2 +\delta^2(\frac{n\epsilon^2}{n-2}+nx^2)].$ Since both $\delta$ and $\epsilon$ are small and positive, with $\delta \ll \epsilon,$ it follows that the smallest positive eigenvalue of $M$ is $\frac{\epsilon^2 n(x+y)^2}{(n-1)(x^2+y^2)+2xy} + O(\epsilon^4)$; hence for $I-T, \sigma_{n-1} = 
\frac{\epsilon (x+y)\sqrt{n}}{\sqrt{(n-1)(x^2+y^2)+2xy}} + O(\epsilon^2)$

For the corresponding left singular vector, we see that it is a scalar multiple of the vector $\left[ \begin{array}{c}1\\-a \\-b \1 \end{array}
 \right],$ where 
 $$\left[ \begin{array}{cc} \sigma_{n-1}^2 - \delta^2 \frac{n-1}{n-2} & \delta(x+(n-1)y) \\ 
 \frac{\delta}{n-2}(x+(n-1)y) &  \sigma_{n-1}^2 - ((n-1)(x^2+y^2)+2xy) 
 \end{array} \right] \left[ \begin{array}{c} a\\b \end{array}
 \right] = \epsilon \left[ \begin{array}{c}-\delta \\ y-x\end{array}
\right].
 $$
We deduce that $\left[ \begin{array}{c} a\\b \end{array}\right]$ is a positive scalar multiple of the vector 
$$\left[ \begin{array}{c}
 \delta(nx(x+y)-\sigma_{n-1}^2  )\\ 
\delta^2(\frac{x+(n-1)y}{n-2}) + (y-x)(\sigma_{n-1}^2 - \delta^2(\frac{n-1}{n-2})) 
 \end{array}\right].$$
Consequently the left singular vector of $I-T$ corresponding to $\sigma_{n-1}$ is a scalar multiple of 
$$\left[ \begin{array}{c} 1\\
-\delta(nx(x+y)-\sigma_{n-1}^2  )\\ 
-[\delta^2(\frac{x+(n-1)y}{n-2}) + (y-x)(\sigma_{n-1}^2 - \delta^2(\frac{n-1}{n-2})) ]\1
\end{array}\right].$$ In particular the sign of the first entry of that singular vector is opposite to the sign of all the remaining entries.  Our clustering heuristic then places index $1$ in the first cluster and indices $2, \ldots, n$ in the second cluster. 

Next we consider the $(2,2)$ entry of the associated coupling matrix, which corresponds to the second cluster. Note that 
\begin{eqnarray*}
&&\left[ \begin{array}{c|c}
\delta(nx(x+y)-\sigma_{n-1}^2  ) & 
[\delta^2(\frac{x+(n-1)y}{n-2}) + (y-x)(\sigma_{n-1}^2 - \delta^2(\frac{n-1}{n-2})) ]\1^{\top}
\end{array}\right] \times \\ 
&&\left[ \begin{array}{c|c} 
1-\delta&\frac{\delta}{n-2} \1^{\top} \\ 
\hline y\1&(1-x-y)I
\end{array}\right] \1 \\ 
&&= \delta(nx(x+y)-\sigma_{n-1}^2  )+\\ &&(1-x)(n-2)\left[\delta^2\left(\frac{x+(n-1)y}{n-2}\right) + (y-x)\left(\sigma_{n-1}^2 - \delta^2\left(\frac{n-1}{n-2}\right)\right) \right].
\end{eqnarray*}
It now follows that the $(2,2)$ entry of the coupling matrix is given by 
$$\mu = 1 - \frac{x(n-2)[\delta^2(\frac{x+(n-1)y}{n-2}) + (y-x)(\sigma_{n-1}^2 - \delta^2(\frac{n-1}{n-2})) ]}
{\delta(nx(x+y)-\sigma_{n-1}^2  ) +(n-2)[\delta^2(\frac{x+(n-1)y}{n-2}) + (y-x)(\sigma_{n-1}^2 - \delta^2(\frac{n-1}{n-2})) ] }.$$ Since $\delta \ll  \epsilon,$ it follows that $\mu = 1-x+O(\epsilon^2).$ On the other hand, $$1-\sqrt{n} \sigma_{n-1} = 1 - \frac{\epsilon (x+y)n}{\sqrt{(n-1)(x^2+y^2)+2xy}} + O(\epsilon^2).$$ Since  $\epsilon$ is small, 
$1-x < 1 - \frac{\epsilon (x+y)n}{\sqrt{(n-1)(x^2+y^2)+2xy}} + O(\epsilon^2),$ so that $\mu < 
1-\sqrt{n} \sigma_{n-1}. $

\section{Explicit computation of the left-iterative weight vector}\label{sec:liwv}

In this section we show explicitly how the left-iterative weight vector is computed as part of the Left Singular Vector clustering algorithm.

\begin{algorithm}[H]\label{alg+liwv}

\caption{The Left Singular Vector algorithm + Left Iterative Weight Vector}
Input: Stochastic matrix $T$ with index set $S$, tolerance $\tau$, optional weight vector $w$\\
Output: A set of disjoint subsets (clusters) of the index set of $T$, left-iterative weight vector $v(T, \tau)$.


\begin{enumerate}
\item	\label{step:leftsingvec}Let $\sigma$ be the second smallest singular value of $I - T$ and $u$ a corresponding left singular vector with mixed signs.

\item	If $w$ was not provided, set $w := \1_{|S|}$. This initialization will take place in the top-level call to the algorithm.

\item	\label{step:partition}If $\sigma > \tau$, do nothing.  Otherwise, let $S_1$ and $S_2$ be the sets of indices corresponding to positive and negative entries of $u$, respectively, and replace the entries of $w$ corresponding to indices in $S$ with $|u|$.  

\item	\label{step:dnf}For $i = 1, 2$, let $\tilde T_i := \dnf(T[S_i])$.  
\item	Recurse on $\tilde T_1$, then on $\tilde T_2$, passing $w$ as the weight vector in each case.  The recursive call on $\tilde T_i$ returns a set of clusters $\mathcal P_i$, i.e., a set of disjoint subsets of $S_i$, and updates $w$.  Note that these recursive calls take place in succession, not in parallel, and that there may be some unclustered vertices in $S_i$ which do not belong to any of the sets in $\mathcal P_i$.  

\item	Return $\mathcal P := \mathcal P_1 \cup P_2$, i.e.\ the set of all clusters found, and $v(T, \tau) := w$.


\end{enumerate}
\end{algorithm}

\end{document}